\documentclass[a4paper]{article}

\usepackage{amsmath} 
\usepackage{amsthm}  
\usepackage{amssymb} 
\usepackage[pdftex,bookmarksnumbered=true]{hyperref}
\usepackage{color}
\usepackage{tikz}

\usepackage[all]{xy}

\usepackage{tocloft}
 
\newtheorem{theorem}{Theorem}[section]
\newtheorem{lemma}[theorem]{Lemma}
\newtheorem{proposition}[theorem]{Proposition}
\newtheorem{corollary}[theorem]{Corollary}
\newtheorem{claim}[theorem]{Claim}

\newtheorem{problem}{Problem}

\newenvironment{remark}%
   {\refstepcounter{theorem}%
        \medbreak\noindent{\bf Remark \thetheorem.\space}}%
   {\par\medbreak}%
\newenvironment{example}%
   {\refstepcounter{theorem}%
        \medbreak\noindent{\bf Example \thetheorem.\space}}%
   {\par\medbreak}%
\renewenvironment{proof}%
   {\medbreak\noindent{\it Proof:\space}}%
   {\par\noindent\vrule height 5pt width 5pt depth 0pt\smallbreak}%
\newcommand{\df}{\bf}
\newcommand{\tq}{\mathrel{:}}

\newtheorem*{theorem*}{Theorem}

\let\sauvegardetiret=\-
\renewcommand{\-}[1]{\ifx#1-\penalty10000\hbox{-\relax}\penalty10000\else\sauvegardetiret#1\fi}
\newcommand{\NN}{{\rm\bf N}}
\newcommand{\ZZ}{{\rm\bf Z}}
\newcommand{\QQ}{{\rm\bf Q}}
\newcommand{\RR}{{\rm\bf R}}

\newcommand{\GG}{{\rm\bf G}}

\newcommand{\PP}{{\rm\bf P}}  
\newcommand{\PN}{{
                  P}}
\newcommand{\UU}{{
                  U}}

\newcommand{\cA}{{\cal A}}
\newcommand{\cB}{{\cal B}}
\newcommand{\cC}{{\cal C}}
\newcommand{\cD}{{\cal D}}
\newcommand{\cE}{{\cal E}}
\newcommand{\cF}{{\cal F}}

\newcommand{\cH}{{\cal H}}
\newcommand{\cI}{{\cal I}}

\newcommand{\cO}{{\cal O}}
\newcommand{\cP}{{\cal P}}
\newcommand{\cQ}{{\cal Q}}

\newcommand{\cS}{{\cal S}}
\newcommand{\cT}{{\cal T}}
\newcommand{\cU}{{\cal U}}
\newcommand{\cV}{{\cal V}}

\newcommand{\cX}{{\cal X}}
\newcommand{\cY}{{\cal Y}}
\newcommand{\cZ}{{\cal Z}}

\newcommand{\Supp}{\mathop{\rm Supp}}
\newcommand{\Card}{\mathop{\rm Card}}
\newcommand{\CB}{\mathop{\rm CB}}
\newcommand{\Gr}{\mathop{\rm Gr}}
\newcommand{\Lring}{\mathop{{\cal L}_{\rm ring}}}

\newcommand{\Vect}{\mathop{\rm Vect}}
\newcommand{\Alg}{\mathop{\rm Alg}}
\newcommand{\coalg}{\mathop{\rm coalg}}

\newcommand{\Tp}{\mathop{\rm tp}}
\newcommand{\Id}{\mathop{\rm Id}}

\newcommand{\TR}[1]{\mathbf T_{#1}}

\title{Semi-algebraic triangulation over $p$-adically closed fields}
\author{Luck Darni\`ere}

\begin{document}

\maketitle

\begin{abstract}
  We prove a triangulation theorem for semi-algebraic sets over a
  $p$\--adically closed field, quite similar to its real counterpart.
  We derive from it several applications like the existence of
  flexible retractions and splitting for semi-algebraic sets. 
\end{abstract}

\begin{center}
  \begin{minipage}{.7\textwidth}
    \setlength{\parskip}{-2.5ex}
    \centerline{\bf Contents}
    \vskip-8mm
    \let\bfseries\relax
    \tableofcontents
  \end{minipage}
\end{center}


\section{Introduction}
\label{se:intro}

Our knowledge of geometric objects in affine spaces over $p$\--adic
fields, that is the field $\QQ_p$ of $p$\--adic numbers or a finite
extension of it, has grown tremendously in the past decades. Several
remarkable analogies have emerged with real geometry, in spite of the
striking differences between the real and the $p$\--adic metrics. The
present paper raises a new such analogy: we prove a triangulation
theorem over $p$\--adically closed fields, quite similar to its real
counterpart. Let us first recall the classical results in $p$-adic
geometry which will be used here. 
\\

Semi-algebraic sets over a field $F$ are the finite unions of sets
defined by finitely many conditions ``$f(x)=0$'' or ``$f(x)$ has a
non-zero $N$\--th root in $F$'', where $f$ is a polynomial function
with $m$ variables. Of course if $F$ is real closed we can restrict
the last conditions to $N=2$ (that is to ``$f(x)>0$'') and if $F$ is
algebraically closed to $N=1$ (that is to $f(x)\neq0$). It is shown in
\cite{maci-1976} that semi-algebraic sets over $\QQ_p$ are stable
under the taking of boolean combinations \emph{and} projections (from
$\QQ_p^m$ to $\QQ_p^{m-1}$, for every $m$). This is a $p$-adic version
of Tarski's theorem for real closed fields (and of Chevalley's theorem
for algebraically closed fields). Prestel and Roquette (see
\cite{pres-roqu-1984}) have generalized it to arbitrary $p$\--adically
closed fields (a $p$\--adic version of real closed fields). 
\\

Denef has proved in \cite{dene-1984} a Cell Decomposition Theorem for
$p$\--adic semi-algebraic sets very similar to its real counterpart,
and derived from it the rationality of Poincar\'e series. Another major
application of cell decomposition is that it provides a good dimension
theory for semi-algebraic sets (see \cite{drie-scow-1988}). Nowadays a
cell over $\QQ_p$ is generally defined as the set of $(x,t)\in\QQ_p^m\times\QQ_p$
such that
\begin{equation}
  |\nu(x)|\square_1|t-c(x)|\square_2|\mu(x)|\quad\mbox{and}\quad t-c(x)\in H
  \label{eq:cell-square}
\end{equation}
where $c$, $\mu$, $\nu$ are semi-algebraic functions (that is functions
whose graph is semi-algebraic), $\square_1$ and $\square_2$ are $\leq$, $<$ or no
relation, and $H$ is $\{0\}$ or a coset of some semi-algebraic subgroup
$\GG$ of $\QQ_p^\times=\QQ_p\setminus\{0\}$. We call it a cell mod $\GG$. Denef worked
with cells mod $\QQ_p^\times$ and implicitly with cells mod $\PN_N^\times$, the
multiplicative group of non-zero $N$\--th powers. This use of cells
mod $\PN_N^\times$ was then made more explicit by Cluckers. Gradually,
people began to replace them by cells mod
\begin{displaymath}
  Q_{N,M}^\times=\bigcup_{k\in\ZZ}p^{kN}(1+p^M\ZZ_p)
\end{displaymath}
where $\ZZ_p$ denotes the ring of $p$\--adic integers (and $\ZZ$ the ring
of integers). Indeed the Cell Decomposition Theorem only asserts that
every semi-algebraic set $S\subseteq\QQ_p^{m+1}$ has a finite partition in cells
mod $\PN_N^\times$ for some $N$. But it usually comes with a Cell
Preparation Theorem (similar to Weierstrass preparation) which says
that, given semi-algebraic functions $\theta_1,\dots,\theta_r$ from $S$ to $\QQ_p$, for
some positive integers $N$, $M$, $e$ there is a partition of $S$ in
finitely many cells mod $Q_{N,M}^\times$ on each of which 
\begin{displaymath}
  |f(x,t)|^e=|h(x)|\cdot|t-c(x)|^\alpha
\end{displaymath}
where $h$ is a semi-algebraic function, $\alpha\in\ZZ$ and $c$ is as in
(\ref{eq:cell-square}). Using such a preparation, Cluckers has proved
in \cite{cluc-2001} that for every two infinite semi-algebraic sets
$A$, $B$ over a $p$\--adically closed field, there is a semi-algebraic
bijection from $A$ to $B$ if and only if $A$ and $B$ have the same
dimension. 

Note that these semi-algebraic bijections are not continuous in
general: for example Clucker's theorem applies to the valuation ring
$\ZZ_p$, which is compact, and to $\ZZ_p\setminus\{0\}$, which is not. This lack of
global continuity conditions is due to the fact that the cell
decomposition techniques treat the various cells of the partition
independently, without giving any information on how their frontiers
touch each other. This is where triangulations come into the
picture\footnote{A different improvement of cell decompositions facing
  this question is given by stratifications. Such stratifications have
  been recently introduced in $p$\--adically closed field
  \cite{cluc-comt-loes-2012}, and in more general non-standard
  Henselian valued fields \cite{halu-2014}. However their relationship
  with the $p$\--adic triangulation is quite unclear at the moment,
  due to the very peculiar conditions involved in the definition of
$p$\--adic simplexes.\label{ft:strat}}.
\\

The real Triangulation Theorem says that every semi-algebraic set over
$\RR$ is semi-algebraically homeomorphic to the union of a simplicial
complex, that is (informally) a finite family of simplexes which touch
each other along their faces. We have introduced in \cite{darn-2017} a
notion of polytopes and simplexes adapted to $\QQ_p$. We delay precise
definitions to Section~\ref{se:notation} but give here some intuition
on it. The $p$\--adic polytopes share many structural properties with
their real counterpart:
\begin{itemize}
  \item
    As inverse images by the valuation (in $Q_{1,M}^r$) of subsets
    of $\ZZ^r$, they are defined by very special (simple) systems of
    $\QQ$\--linear inequalities. 
  \item
    There is a notion of ``faces'' attached to them with good
    properties: every face of a polytope $S$ is itself a polytope; if
    $S''$ is a face of $S'$ and $S'$ a face of $S$ then $S''$ is a
    face of $S$; the union of the proper faces of $S$ is a partition
    of its frontier.
  \item 
    Among the $p$\--adic polytopes, the simplexes are those whose
    number of facets is minimal\footnote{Real simplexes can be
    characterised, among the polytopes of a given dimension $d$, as
  those whose number is minimal (namely $d+1$).} in a very strong
  sense: a simplex has at most one facet, which is itself a simplex.
  \item
    Last but not least, every $p$\--adic polytope can be divided in
    simplexes by a certain uniform process of ``Monotopic Division''
    which offers a good control both on their shapes and their faces. 
\end{itemize}
Just as in the real case, we can then define a simplicial complex over
$\QQ_p$ essentially as a finite family of simplexes in $Q_{1,M}^r$, for
some positive integer $M$, which touch each other along their faces
(again, see Section~\ref{se:notation} for a more precise definition). A
simplified version of our main result, the Triangulation
Theorem~\ref{th:triang-m}, can then be stated as follows. 

\begin{theorem*}[Triangulation]
  For every semi-algebraic set $S\subseteq\QQ_p^m$ there is a semi-algebraic
  homeomorphism $\varphi:T\to S$ whose domain $T$ is the union of a simplicial
  complex $\cT$. 
 
  Moreover, given semi-algebraic functions $\theta_1,\dots,\theta_r$ from $S$ to
  $\QQ_p$, $(\cT,\varphi)$ can be chosen so that on every $T\in\cT$ the
  valuation of each $\theta_i\circ\varphi(y)$ is a $\ZZ$\--linear function of the
  valuations of the coordinates of $y\in T$. 
\end{theorem*} 

\begin{remark}\label{re:real-vs-p-adic}
  The simplexes in the above complex $\cT$ are not contained in
  $\QQ_p^m$ but in finitely many copies of $\QQ_p^q$ for various $q$,
  usually much larger than $m$. This is the main, but harmless,
  difference with the triangulation in the real case. 
\end{remark}

In the real case, cell decomposition and triangulation hold not only
for semi-algebraic sets over $\RR$ but also over any real-closed fields,
and more generally for definable sets in $o$\--minimal expansions of
such fields. In the $p$\--adic case, Denef's Cell Decomposition
Theorem holds in arbitrary $p$\--adically closed fields. Several
variants of it, sometimes weaker, have been proved to hold in some, if
not all, $P$\--minimal expansions of such fields (see
\cite{dene-drie-1988} and \cite{cluc-2004} for subanalytic sets,
\cite{hask-macp-1997}, \cite{darn-cubi-leen-2017},
\cite{cubi-leen-2016}, \cite{cham-cubi-leen-2017} and
\cite{darn-halu-2017} for definable sets in $P$\--minimal and
$p$\--optimal structures). 

In the present paper we do not restrict ourselves to $\QQ_p$ and its
finite extensions, but work in an arbitrary $p$\--adically closed
field $K$ fixed once and for all. Apart of the $p$\--adic fields there
are many natural examples of such: the algebraic closure of $\QQ$ inside
$\QQ_p$ (which is not complete), the $t$\--adic completion of the field
$\bigcup_{n\geq1}\QQ_p((t^{1/n}))$ of Puiseux series over $\QQ_p$ (whose value
group is not $\ZZ$, but $\ZZ\times\QQ$ lexicographically ordered), and many
others (every ultraproduct of $p$\--adically closed fields is still
$p$\--adically closed). We let $v$ denote the (unique) $p$-adic
valuation of $K$, $R$ its valuation ring and $\cZ$ its value group. As
usual $\cZ$ is augmented by an element $+\infty$ for $v(0)$, and we let
$\Gamma=\cZ\cup\{0\}$. 

Almost all the arguments in our proofs remain valid for
definable sets in $p$\--optimal structures on $K$ satisfying the
Extreme Value Property (see \cite{darn-halu-2017}). Unfortunately
there is one single exception: the proof of the Good Direction
Lemma~\ref{le:gd-open}, which involves polynomial functions, does not
generalize to the more general ``basic functions'' involved in the
definable sets in $p$\--optimal structures. Thus we will stick to the
semi-algebraic framework in all this paper. 
\\

It is organised as follows. All the needed prerequisites, in
particular those concerning $p$\--adic simplexes, are recalled in
Section~\ref{se:notation}, which culminates with the final statement
of the Triangulation Theorem for semi-algebraic sets and functions in
$m$ variables (Theorem~\ref{th:triang-m}). We denote it $\TR m$. All
the applications presented below are then derived from $\TR m$ in
Section~\ref{se:appli-triang}. By means of these applications and
Denef's Cell Preparation Theorem we prove in Section~\ref{se:cell-dec}
a ``largely continuous cell preparation up to a small deformation''
for semi-algebraic functions in $m+1$ variables
(Theorem~\ref{th:large-cont-prep}). Sections~\ref{se:cont-comp} to
\ref{se:cart-map} are then devoted to our main constructions, which
are summarized in Lemma~\ref{le:pretriang} and
Lemma~\ref{le:lift-triang} (see also Remark~\ref{re:comp-cell} below).
In Section~\ref{se:triang}, we finally derive $\TR {m+1}$ from $\TR m$
by means of these two technical lemmas, which finishes the proof of
our $p$\--adic triangulation theorem for every $m$. 

\begin{remark}\label{re:comp-cell}
  Denef's Cell Decomposition Theorem gives a partition of any
  semi-algebraic set $S\subseteq K^{m+1}$ in finitely many cells, but we do
  not control how these cells touch each other. On the other
  hand, if a cell $C$ is defined by functions $c$, $\mu$, $\nu$ which
  extend to continuous functions $\bar c$, $\bar \mu$, $\bar \nu$ on the
  closure of $\widehat{C}$, the frontier of $C$ naturally decomposes
  in cells, each of which is defined by means of $\bar c$, $\bar \mu$,
  $\bar \nu$. These auxiliary cells can be seen as ``faces'' of $C$. It
  allows us to speak of ``complexes of cells'', in a sense which will
  be made precise in sections~\ref{se:cont-comp} and
  \ref{se:cont-mono}. The main results of these sections prove that
  after only a linear deformation of $S$, which can be chosen
  arbitrarily ``small'' (that is close to the identity), it is
  possible to decompose the image of $S$ in a {\em complex} of cells.
  Moreover one can require this complex to be a tree with respect to
  the specialisation order. 
\end{remark}

We now present several other applications of the Triangulation
Theorem, all of which will be proved in Section~\ref{se:appli-triang}. 

\begin{theorem*}[Lifting]
  For every semi-algebraic function $f:X\subseteq K^m\to K$ such that $|f|$ is
  continuous, there is a continuous semi-algebraic function $h:X\to K$
  such that $|f|=|h|$. 
\end{theorem*}

In the above theorem $|x|=p^{-v(x)}$ is the usual $p$\--adic norm if
$\cZ=\ZZ$. Otherwise this $p$\--adic norm is not defined but can be
replaced without inconvenience by the following generalization: we let
$|a|=aR^\times=\{au\tq u\in R^\times\}$ for every $a\in K$, and $|K|=\{|a|\tq a\in K\}$.
The latter is naturally ordered by inclusion, and isomorphic to
$\Gamma$ with the reverse order~: $|a|\leq|b|$ if and only if $v(a) \geq
v(b)$. So $|a|$ is just a multiplicative notation for $v(a)$: we have
$|ab|=|a|.|b|$ and $|a+b|\leq\max(|a|,|b|)$, and of course $|a|=0$ if and
only if $a=0$.

The real counterpart of the above result is quite obvious. On the
contrary, the next two results do not hold in real geometry. In the
same vein as Clucker's result on classification of semi-algebraic sets
up to semi-algebraic bijection \cite{cluc-2001}, they confirm the
intuition that the lack of connectedness and of ``holes'' (in the
sense of algebraic topology, see below) makes semi-algebraic sets over
$p$\--adically closed fields much more flexible than over real closed
fields. 
\\

Recall that a {\df retraction} of a topological space $X$ onto a
subspace $Y$ is a continuous map $\sigma:X\to Y$ such that $\sigma(y)=y$ for every
$y\in Y$. When such a retraction exists on a Hausdorff space $X$, then
necessarily $Y$ is closed in $X$.

Over the reals, the main obstruction for the existence of retractions
is the existence of ``holes'' which are detected by homotopy. This
does not work over $p$-adic fields. Indeed, replacing the unit
interval $[0,1]$ in the reals by the unit ball in $K$, that is the
ring $R$ of the $p$\--adic valuation of $K$, we may say that a
non-empty semi-algebraic set $X\subseteq K^m$ is ``semi-algebraically
contractible'' if there is a continuous semi-algebraic function $H: X\times
R\to X$ and $a\in X$ such that $H(x,1)=x$ and $H(x,0)=a$ for every $x\in X$.
But this is always true: given any $a\in X$ the function $H(x,s)=x$ if
$s$ is invertible in $R$ and $H(x,s)=a$ otherwise, has all the
required properties. However it is another story to prove that
retractions actually exist.

\begin{theorem*}[Retraction] 
  For every non-empty semi-algebraic sets $Y\subseteq X\subseteq K^m$, there is a
  semi-algebraic retraction of $X$ onto $Y$ if and only if $Y$ is
  closed in $X$.
\end{theorem*}

It is worth mentioning that it is the next Splitting Theorem, already
conjectured in \cite{darn-2006}, which was the main motivation for the
research presented in this paper. Here $\partial X$ denotes the topological
frontier of $X$, see Section~\ref{se:notation}.

\begin{theorem*}[Splitting]
  Let $X$ be a relatively open non-empty semi-algebraic subset of
  $K^m$ without isolated points, and $Y_1,\cdots,Y_s$ a collection of
  closed semi-algebraic subsets of $\partial X$ such that\footnote{Note that
    $Y_1,\dots,Y_s$ are not assumed to be disjoint. All of them can
  be equal to $\partial X$, for example.} $Y_1\cup\cdots\cup Y_s=\partial X$.
  Then there is a partition of $X$ into non-empty\footnote{A {\df partition} of
  a set $X$ is for us a family of two-by-two disjoint subsets of $X$
  covering $X$. We do not assume that the pieces must be non-empty. So
  when it happens by exception, like here, that this property is required
  and does not follow from the context, we explicitly mention it.}
  semi-algebraic sets $X_1,\dots,X_s$ such that $\partial X_i= Y_i$ for $1\leq i\leq s$.
\end{theorem*}

The trivial remark that every ball $B\subseteq K^m$ is disconnected can be
seen as a very special case of the above Splitting Property (applied
to $X=B$ with $Y_1=Y_2=\emptyset$). This property is actually (in a sense
which can be made precise, see \cite{darn-2018-tmp}) the strongest
possible form of disconnectedness that can be observed in a finite
dimensional topological space whose points are closed. It is a
versatile property which we encountered in different contexts with
minor changes (see \cite{darn-2018-tmp}, \cite{darn-junk-2018}). In the
present paper, it plays a key role in the induction step. 
\\

A {\df limit value} for a function $f:X\subseteq K^m\to K$ at a point $x$
adherent to $X$, is a value  $l\in K$ such that $(x,l)$ is adherent to
the graph of $f$. Of course $f$ tends to $l$ at $x$ if and only if $l$
is the unique limit value of $f$ at $x$. Let us say that $f$ is {\df
largely continuous} on a subset $A$ of $X$ if the restriction of $f$
to $A$ has a unique limit value at every point adherent to $A$, that
is if $f$ extends to a continuous function on the topological closure
of $A$. If $A$ is not mentioned it simply means that $f$ is largely
continuous on its domain $X$. Finally $f$ is {\df piecewise largely
continuous} if there exists a finite partition of $X$ in
semi-algebraic pieces on which $f$ is largely continuous. Of course in
that case $f$ has finitely many limit values at every point adherent
to $X$. 

\begin{theorem*}[Largely Continuous Splitting]
  Let $f:X\subseteq K^m\to K$ be a semi-algebraic function whose graph has
  bounded\footnote{%
    This boundedness assumption could easily be removed. It suffices
    to add to $K$ a point at infinity and require that $f$ has
    finitely many limit values in $\tilde K=K\cup\{\infty\}$ at every point of the
    closure of $X$ in $\tilde K^m$, using the same construction as in
    the preparation of the proof of Lemma~\ref{le:split-level}.
  }
  domain. If $f$ has finitely many limit values at every point
  adherent to $X$ then $f$ is piecewise largely continuous.
\end{theorem*}

The real counterpart of this result is easily seen to be true, by
means of a triangulation and the trivial remark that every real
simplex is connected (see Section~\ref{se:appli-triang}). This
last argument is no longer valid in the $p$\--adic case but, as we will
see, the existence of retractions allows us to bypass this problem and
recover the full result in the $p$\--adic context. \\

We can also mention two other applications of the Triangulation
Theorem, to $p$\--adic semi-algebraic geometry and to model theory,
which are outside of the scope of this paper.
\begin{enumerate}
  \item
    [(i)] One of the main advantages of proving the Triangulation
    Theorem for every $p$\--adically closed field, not only for 
    $p$\--adic fields, is that it allows us to combine it with the
    very powerful model theoretic compactness theorem. This in turn
    provides ``uniform'' triangulations, which {\em almost} give us
    for free a $p$\--adic analogue of Hardt's Theorem\footnote{Hardt's
      Theorem in real geometry says that the fibers of a
    semi-algebraic projection have finitely many homeomorphism types.}.
    Some difficulties still remain because it is much less easy to
    construct homeomorphisms between $p$\--adic simplexes than between
    real simplexes (see Problem~\ref{pr:classification}). Hopefully this
    will be addressed in a further paper. 
  \item
    [(ii)] By \cite{darn-2006} the Splitting Property for $p$\--adic
    sets (which was only conjectural at this time) ensures that the
    complete theory of the lattice $L(K^m)$ of closed semi-algebraic
    subsets of $K^m$ is decidable. This is in contrasts with the real case, since we
    know from \cite{grze-1951} that the complete theory of $L(\RR^m)$ is
    undecidable for every $m\geq2$. Moreover the theory of $L(K^m)$
    only depends on $m$, not on the $p$\--adically closed field
    involved and not even on $p$, hence it is the same for $L(K^m)$
    and $L(\QQ_2^m)$ (see \cite{darn-2018-tmp}). 
\end{enumerate}

Finally let us present a few open problems tightly connected
with the present work. 

\begin{problem}
  Extend the Triangulation Theorem to $p$\--adic subanalytics sets,
  and more generally to definable sets in some $p$\--optimal
  expansions of $K$.
\end{problem}

\begin{problem}\label{pr:classification}
  By giving reasonable sufficient conditions for different $p$\--adic
  simplexes to be homeomorphic, classify $p$\--adic semi-algebraic
  sets up to semi-algebraic homeomorphisms.
\end{problem}

\begin{problem}
  For any semi-algebraic set $S\subseteq K^m$, construct a triangulation
  $(\varphi,\cT)$ such that the image of $\cT$ by $\varphi$ is a stratification of
  $S$. Or conversely use existing stratifications of $S$ (see
  footnote~\ref{ft:strat}) to construct a better (or a more general)
  triangulation.
\end{problem}

\section{Prerequisites and notation}

\label{se:notation}

We let $\NN$ denote the set of positive integers and $\NN^*=\NN\setminus\{0\}$. For
all integers $k,l$ we let $[\![k,l]\!]$ be the set of integers $i$ such
that $k\leq i\leq l$ (hence an empty set if $k>l$).

Recall that we have fixed once and for all a $p$\--adically closed
field $K$. Following \cite{pres-roqu-1984}, this is the fraction
field of a (unique) Henselian valuation ring $R$ such that the residue
field of $R$ is finite, the value group $\cZ$ of $R$ has a least
strictly positive element, and $\cZ/n\cZ$ has exactly $n$ elements for
every integer $n\geq1$. We fix once and for all a generator $\pi$ of the
maximal ideal of $R$, and let $R^\times=R\setminus \pi R$ denote the multiplicative
group of invertible elements of $R$.

Let $\cQ$ denote the divisible hull of $\cZ$ and $\Omega=\cQ\cup\{+\infty\}$. As
an ordered group, $\ZZ$ identifies naturally to the smallest non-trivial
convex subgroup of $\cZ$. We consider $\ZZ$ and $\QQ$ as embedded into
$\cQ$ {\it via} this identification. 

For every subset $X$ of $K$ we let $X^*=X\setminus\{0\}$. However, if $X^*$ is a
subgroup of the multiplicative group of $K$, we denote it $X^\times$ in
order to highlight this property (so $R^*=R\setminus\{0\}\neq R^\times$ but
$K^*=K\setminus\{0\}=K^\times$). For every subgroup $G$ of $K^\times$ we let $xG=\{xg\tq g\in
G\}$ for every $x\in K$, and $K/G=\{xG\tq x\in K\}$. For example $xR^\times=|x|$
and $K/R^\times=|K|$. Abusing the notation, $0G=\{0\}$ will be denoted $0$
whenever the context makes it unambiguous. 

In order to ease the notation, given $a\in K^m$, $A\subseteq K^m$ and $f:X\to K^m$
we will often write $va$ for $v(a)$, $vA$ for the direct image $v(A)$,
$vf$ for the composite $v \circ f$, and similarly for $|A|$ and $|f|$.

At some rare places it will be convenient to add to $K$ a new element
$\infty$ (and to $\Gamma$ and $|K|$ new elements $-\infty$ and $+\infty$ respectively)
with the natural convention that $0^{-1}=\infty$, $\infty^{-1}=0$, $v(\infty)=-\infty$,
$|\infty|=+\infty$, and $a.\infty=\infty$ for every $a\in K^\times$. We also let
$0.\infty=1$ and $0^0=1$ when needed.

\subsection{Topology and coordinate projections}

When an $m$\--tuple $a$ is given, it is understood that $(a_1,\dots,a_m)$
are its coordinates, except if otherwise specified. For every $a\in K^m$
we let
\begin{displaymath}
  va=\big(va_1,\dots,va_m\big)
  \quad\mbox{ and }\quad
  |a|=\big(|a_1|,\dots,|a_m|\big).
\end{displaymath}
This should not be confused with $\|a\|=\max(|a_1|,\dots,|a_m|)$. For $r\in
K^\times$ the (clopen) {\df ball} of center $a$ and radius $r$ is defined as
\begin{displaymath}
  B(a,r)=\big\{x\in K^m\tq \|x-a\|\leq|r|\big\}. 
\end{displaymath}

The valuation induces a topology on $K$, which is inherited by $|K|$
and $\Gamma$. The topology generated on $\Omega$ by the open intervals and the
intervals $]a,+\infty]$ for $a\in\cQ$, extends the topology of $\Gamma$. The
direct products of these topological spaces are endowed with the
product topology. For every subset $X$ of any of these spaces,
$\overline{X}$ denotes the topological closure of $X$. In particular
$\overline{\cZ}=\Gamma$ and $\overline{\cQ}=\Omega$. Note that $\Gamma$ is closed in
$\Omega$. The {\df specialisation preorder} on the subsets of $X$ is
defined by $B\leq A$ iff $B\subseteq \overline{A}$. 

We let $\partial X=\overline{\overline{X}\setminus X}$ denote the {\df frontier} of
$X$. We say that $X$ is {\df relatively open} if it is open in
$\overline{X}$, that is if $\partial X=\overline{X}\setminus X$. 

When a function $f$ is largely continuous (see Section~\ref{se:intro}) we usually
denote $\overline{f}$ the continuous extension of $f$ to the closure
of its domain. On the contrary, the restriction of
$f$ to some subset $A$ of its domain is denoted $f_{|A}$.
\\

The {\df support} of an element $a$ of $K^m$ (or $|K|^m$),
denoted $\Supp a$, is the set of indexes $k$ such that $a_k\neq0$. The
support of an element $b$ of $\Gamma^m$, denoted $\Supp b$, is the set of
indexes $k$ such that $b_k=+\infty$. Note that with this definition, one
has that for every $a\in K^m$
\begin{displaymath}
   \Supp a = \Supp |a| = \Supp v(a). 
\end{displaymath}

For every subset $S$ of $K^m$ (resp. $|K|^{m+1}$, resp. $\Gamma^{m+1}$) and
every $I\subseteq\{1,\dots,m\}$ we let 
\begin{displaymath}
  F_I(S)=\big\{a\in\overline{S}\tq\Supp a=I\big\}. 
\end{displaymath}
When $F_I(S)\neq\emptyset$ we call it the {\df face of $S$ with support $I$}. The
{\df coordinate projection} of $K^m$ (resp. $|K|^m$, $\Gamma^m$) onto its
face with support $I$ will be denoted $\pi_I$. So $\pi_I(a)$ is the unique
point $b$ with support $I$ such that $b_i=a_i$ for every $i\in I$. 
\\

For every $a\in K^{m+1}$ (resp. $|K|^{m+1}$, resp. $\Gamma^{m+1}$) we let
$\widehat{a}$ denote the tuple of the first $m$ coordinates of $a$, so
that $a=(\widehat{a},a_{m+1})$. If $A$ is a set of $(m+1)$\--tuples we
let $\widehat{A}=\{\widehat{a}\tq a\in A\}$, and if $\cA$ is a family of
such sets we let
\begin{displaymath}
  \widehat{\cA}=\big\{\widehat{A}\tq A\in\cA\big\}. 
\end{displaymath}
We call $\widehat{A}$ (resp. $\widehat{\cA}$) the {\df socle} of $A$
(resp. $\cA$).

Given two families $\cH$, $\cA$ of subsets of $K^{m+1}$ we say that
$\cH$ is {\df finer} than $\cA$ if every $H\in\cH$ which meets a set $A\in\cA$
is contained in $A$. If moreover $\cH$ is a partition of $\bigcup\cA$ we
say that $\cH$ {\df refines} $\cA$. We will often distinguish between
``{\em horizontal refinements}'' for which
$\widehat{\cH}=\widehat{\cA}$, and ``{\em vertical refinement}'' for
which $\cH$ is the family of $A\cap(X\times K)$ where $A$ ranges over $\cA$
and $X$ over a refinement of the socle of $\cA$.

\subsection{Semi-algebraic sets} 

For every integer $N\geq1$ let $\PN_N^\times=\PN_N\setminus\{0\}$ with\footnote{The
  notation $\PN_N$ is sometimes used for the set of {\em non-zero}
  $N$-th powers. The conventions used here leads to denote it
  $\PN_N^\times$, so as to highlight its multiplicative group structure.}
\begin{displaymath}
  \PN_N=\big\{ a\in K\tq \exists x\in K,\ a = x^N\big\}.
\end{displaymath}
$\PN_N^\times$ is a clopen subgroup of $K^\times$ with finite index,
and $\PN_1=K^\times$. Hence a subset $K^m$ is a {\df semi-algebraic set} if
it is a boolean combination of finitely many sets $S_i$ defined by
conditions
\begin{equation}\label{eq:semi-alg}
  f_i(x)\in \PN_{N_i}
\end{equation}
where the $f_i$'s are $m$\--ary polynomial functions. A {\df
semi-algebraic map} is a function whose graph is semi-algebraic.
Rational functions, root functions and monomial functions (see below)
are semi-algebraic, among many others. 

Abusing a little bit the terminology, we also say that a subset $S$ of
$K^m\times|K|^n$ is semi-algebraic if $\{(x,t)\in K^{m+n}\tq (x,|t|)\in S\}$ is
semi-algebraic. Similarly a function $f:X\subseteq K^m\to|K|^n$ is
semi-algebraic if its graph is. When a map $\varphi$ is defined on the
disjoint union of finitely many semi-algebraic sets $A_i$ living in
different copies of $K^m$, we say that $\varphi$ is semi-algebraic if its
restriction to each $A_i$ is semi-algebraic in the classical sense. 

\begin{remark}\label{re:N-mac}
  If $N'$ divides $N$ then $\PN_N^\times$ is a clopen subgroup of
  $\PN_{N'}^\times$ with finite index. For this reason, all the
  integers $N_i$ appearing in (\ref{eq:semi-alg}) can be replaced
  by any common multiple $N$. Note also that $0\in\PN_N$ is an empty
  condition, equivalent to $1\in\PN_N$, hence all the $f_i$'s can be
  assumed to be non-zero polynomials.
\end{remark}

\begin{theorem}[Macintyre]\label{th:macintyre}
  If $S\subseteq K^{m+1}$ is semi-algebraic then $\widehat{K}$ is
  semi-algebraic. 
\end{theorem}

This fundamental result has many consequences. The most prominent one
is that a subset $S$ of $K^m$ is semi-algebraic if and only if there
is a first-order formula\footnote{For the notion of first order
  formula, we refer the reader to any introductory book of
  model-theory, such as \cite{hodg-1997} for example.} $\varphi(x)$ in
$\Lring=\{0,1,+,-,\times\}$ (the language of rings), possibly with
parameters in $K$, such that 
\begin{displaymath}
  S=\big\{a\in K^n\tq K\models\varphi(a)\}.
\end{displaymath}

\begin{remark}
  Given $m$\--ary definable functions $f$, $g$, the set of points in
  $K^m$ satisfying the condition ``$|f(x)|\leq|g(x)|$'' is known to be
  semi-algebraic\footnote{This follows from the non-trivial fact that
    $R$ is definable by means of the Kochen operator (see
    \cite{pres-roqu-1984}).}. Thus we will consider these expressions
  as abbreviations for some first order formulas in the language of
  rings stating the same property). Similarly, if $\varphi(x,y)$ is a
  formula with $m+n$ variables and $S\subseteq K^n$ is definable by a formula
  $\psi(y)$ then we will consider $\exists y\in S,\varphi(x,y)$ as a formula since it
  is an abbreviation for the genuine formula $\exists y,\psi(y)\land\varphi(x,y)$. 
\end{remark}

Another important consequence of Macintyre's theorem is that every
$p$\--adically closed field is elementarily equivalent to a finite
extension of $\QQ_p$ (see \cite{pres-roqu-1984}). In other words, there
is a finite extension $L$ of $\QQ_p$ such that $K$ and $L$ satisfy
exactly the same parameter-free formulas in $\Lring$. The following
property transfers from $L$ to $K$ by means of this elementary
equivalence. Recall that a family $(C_a)_{a \in A}$ of semi-algebraic
subsets of $K^n$ is {\df uniformly semi-algebraic} if $A\subseteq K^m$ is
definable and there is a formula $\varphi(x,y)$ with $m+n$ free variables
such that $C_a=\{b\in K^m\tq K\models\varphi(a,b)\}$ for every $a\in A$.

\begin{theorem}\label{th:comp-inter}
  Let $(C_\alpha)_{\alpha\in R^*}$ be a uniformly definable family of non-empty,
  closed and bounded subsets of $K^n$, such that $|\beta|\leq|\alpha|$ implies
  that $C_\beta\subseteq C_\alpha$. Then $\bigcap_{\alpha\in R^*}C_\alpha$ is non-empty. 
\end{theorem}

The next classical properties can easily be derived from this theorem (or
transfered from $L$ to $K$ by elementary equivalence).

\begin{theorem}\label{th:extr-val}
  For every continuous semi-algebraic function $f:X\subseteq K^m\to K^n$ whose
  domain $X$ is closed and bounded, $f(X)$ is closed and bounded. As a
  consequence:
  \begin{enumerate}
    \item
      $\|f\|$ is bounded and attains its bounds.
    \item 
      If $f$ is injective then it is a homeomorphism from $X$ to
      $f(X)$. 
  \end{enumerate}
\end{theorem}

\begin{corollary}\label{co:max}
  For every bounded semi-algebraic subset $X$ of $K^m$ which is
  non-empty, there is an element $x\in X$ such that $\|x\|$ is maximal on
  $X$. 
\end{corollary}

Another crucial property of the semi-algebraic structure on a
$p$\--adically closed fields is the existence of so called ``built-in
Skolem functions'' (see \cite{drie-1984}, or the appendix of
\cite{dene-drie-1988} for a more constructive proof). Basically, this
property says that for every semi-algebraic subset $A$ of $K^{m+n}$,
the coordinate projection of $A$ onto $K^m$ has a {\em semi-algebraic}
section.

\begin{theorem}[Skolem functions]\label{th:skol}
  Let $X\subseteq K^m$ be semi-algebraic set and $\varphi(x,t)$ a formula  with
  $m+n$ free variables. If, for every
  $a\in X$ there is $b\in K^n$ such that $K\models\varphi(a,b)$, 
  then there exists a semi-algebraic function $\xi:X\to K^n$ (called a
  Skolem function) such that $K\models\varphi(x,\xi(x))$ for every $x\in X$. 
\end{theorem}

For example, if a semi-algebraic function $f:X\to K$ takes values in
$\PN_N$, then Theorem~\ref{th:skol} applied to the formula $\varphi(x,t)$
saying that ``$f(x)=t^N$'' gives a semi-algebraic function $\xi:X\to K$
such that $f=\xi^N$. 
\\

There is a good dimension theory for semi-algebraic sets over
$p$\--adically closed fields, see \cite{drie-scow-1988} and
\cite{drie-1989}. We will repeatedly use the following properties of
this dimension, for every semi-algebraic sets $A$, $B$ and
semi-algebraic map $f$ defined on $A$. By convention $\dim\emptyset=-\infty$.

\begin{enumerate}
  \item 
    $\dim A = 0$ if and only if $A$ is finite non-empty.
    
  \item
    $\dim A\cup B=\max(\dim A,\dim B)$. 
  \item 
    If $A\neq\emptyset$, $\dim \partial A<\dim A$. 
  \item 
    $\dim f(A)\leq \dim A$. 
\end{enumerate}

The {\df local dimension} of a semi-algebraic set $A\subseteq K^m$ at a point
$a\in A$ is the minimum of $\dim U$, for every semi-algebraic
neighbourhood $U$ of $a$ in $A$ (with respect to the relative
topology, induced by $K^m$ on $A$). $A$ is {\df pure dimensional} if
it has the same local dimension at every point. Note that if a
semi-algebraic set $B$ is open in $A$ and $A$ is pure dimensional then
so is $B$, and that a cell is pure dimensional if and only if its
socle is. This last point, combined with Denef's Cell Decomposition
Theorem~\ref{th:D1} and a straightforward induction, shows that every
semi-algebraic set $A$ is the union of finitely many pure dimensional
ones.

\subsection{Root functions and monomial functions}

Following Lemma~1.3 in \cite{cluc-leen-2012} there is for each integer
$M>0$ a unique group homomorphism $\overline{ac}_M$ from $K^\times$ to
$(R/\pi^MR)^\times$ such that $\overline{ac}_M(\pi)=1$ and 
$\overline{ac}_M(u)=u+\pi^MR$ for every $u\in R^\times$. The construction of
$\overline{ac}_M$ given in \cite{cluc-leen-2012} shows that for each
integer $N>0$ the set \begin{displaymath}
  Q_{N,M} = \{0\}\cup \big\{ x\in\PN_N^\times\cdot(1+\pi^MR) \tq \overline{ac}_M(x)=1 \big\} 
\end{displaymath}
is semi-algebraic. $Q_{N,M}^\times=Q_{N,M}\setminus\{0\}$ is a clopen subgroup of
$K^\times$ with finite index. When $v(K^\times)=\ZZ$ then 
$Q_{N,M}^\times = \bigcup_{k\in\ZZ} \pi^{kN}(1+\pi^MR)$ so the above definition of
$Q_{N,M}$ is compatible with the notation of the introduction.

If $M>2v(N)$, Hensel's Lemma implies that $1+\pi^MR\subseteq\PN_N$, hence
$Q_{N,M}$ is contained in $\PN_N$. The importance of $Q_{N,M}$ comes
from the following property, which also follows from Hensel's lemma
(see for example lemma~1 and corollary~1 in \cite{cluc-2001}).

\begin{lemma}\label{le:Hensel-DP}
  The function $x\mapsto x^e$ is a group endomorphism of $Q_{N,M}^\times$. If
  $M>v(e)$ this endomorphism is injective and its image is
  $Q_{eN,v(e)+M}^\times$.  
\end{lemma}

In particular $x\mapsto x^e$ defines a continuous bijection from
$Q_{1,v(e)+1}$ to $Q_{e,2v(e)+1}$. We let $x\mapsto x^{1/e}$ denote
the reverse continuous bijection. In particular it is defined on
$Q_{N,M}$ for every $N$, $M$ such that $e$ divides $N$ and $M>2v(e)$. 

\medskip

For all positive integers $e$, $n$ we let 
\begin{displaymath}
  \UU_e=\{x\in K\tq x^e=1\} 
  \quad\mbox{ and }\quad
  U_{e,n}=\UU_e\cdot(1+\pi^nR).
\end{displaymath}
Analogously to Landau's notation $\cO(x^n)$ of calculus, we let
$\cU_{e,n}(x)$ denote {\em any} semi-algebraic function in the
multi-variable $x$ with values in $U_{e,n}$. Any such function
is the product of two semi-algebraic functions, with values in
$\UU_e$ and $1+\pi^nR$ respectively. 
So, given a family of functions $f_i$, $g_i$ on the same domain $X$,
we write that $f_i=\cU_{e,n}g_i$ for every $i$, when there are
semi-algebraic functions $\omega_i:X\to R$ and $\chi_i:X\to\UU_e$ such that for
every $x$ in $X$
\begin{displaymath}
  f_i(x)=\chi_i(x)\big(1+\pi^n\omega_i(x)\big)g_i(x).
\end{displaymath}
$\cU_{1,n}(x)$ is simply denoted $\cU_n(x)$. 

\begin{remark}\label{re:racine-de-U}
  If $f(x)=\cU_n(x)$ for some $n>2v(e)$ then $f^{1/e}$ is well defined
  and takes values in $1+\pi^{n-v(e)}R$. Therefore we can write
  $\cU_n(x)=(\cU_{n-v(e)}(x))^e$. 
\end{remark}

A function $g$ is {\df $N$\--monomial} on $S\subseteq K^q$ if either it is
constantly equal to $\infty$ or there exists $\xi\in K$ and $\beta_1,\dots,\beta_q\in\ZZ$ such that
\begin{displaymath}
  \forall x=(x_1,\dots,x_q)\in S,\ g(x)=\xi\prod_{i=1}^q x_i^{N\beta_i}.
\end{displaymath}
In this definition we use when necessary our convention that $0^0=1$. A
function $f$ is {\df $N$\--monomial mod $U_{e,n}$} if
$f=\cU_{e,n}g$
with $g$ an $N$\--monomial function.

\subsection{Discrete and $p$\--adic simplexes}

We say that $f:S\subseteq F_I(\Gamma^q)\to\Omega$ is {\df affine} if either
it is constantly equal to $+\infty$, or there are elements $\alpha_0\in\cQ$ and
$\alpha_i\in\QQ$ for $i\in I$ such that 
\begin{displaymath}
  \forall x\in S,\ f(x)=\alpha_0+\sum_{i\in I}\alpha_ix_i.
\end{displaymath}

Polytopes\footnote{In \cite{darn-2017} we introduced discrete
polytopes in $\Gamma^q$ as ``largely continuous precells mod $N$'', for an
arbitrary $q$\--tuple $N$ of positive integers. In the present paper
$N=(1,\dots,1)$ will not play any role so we remove it from the
definition.} in $\Gamma^q$ are defined by induction on $q$. The only
polytope in $\Gamma^0$ is $\Gamma^0$ itself (which is a one-point set). For
every $I\subseteq[\![1,q+1]\!]$, a subset $A$ of $F_I(\Gamma^{q+1})$ is a {\df
discrete polytope} of $\Gamma^{q+1}$ if $\widehat{A}$ is a discrete
polytope of $\Gamma^q$ and if there is a pair $(\mu,\nu)$ of {\em largely
continuous} affine maps from $\widehat{A}$ to $\Omega$, called a {\df
presentation} of $A$, such that $0\leq \mu\leq \nu$ and
\begin{displaymath}
  A=\Big\{a\in F_I(\Gamma^{q+1})\tq \widehat{a}\in\widehat{A}\mbox{ and }
  \mu(\widehat{a})\leq a_{q+1} \leq \nu(\widehat{a})\Big\}.
\end{displaymath}

\begin{example}\label{ex:simplex}
  \begin{itemize}
    \item 
      $A=\NN\times\NN$ is a discrete polytope with two facets
      $F_{\{1\}}(A)=\NN\times\{+\infty\}$ and $F_{\{2\}}(A)=\{+\infty\}\times\NN$.
    \item
      $B=\{(x,y)\in\ZZ^2\tq 0\leq y\leq x\}$ is a discrete simplex, with proper
      faces $F_{\{2\}}(B)=\{+\infty\}\times\NN$ and $F_\emptyset(B)=\{(+\infty,+\infty)\}$.
    \item 
      $C=\{(x,y,z)\in\ZZ^3\tq (x,y)\in B$ and $z=2y-2x\}$ is a subset of $\ZZ^3$
      defined by linear inequalities, whose proper faces $F_{\{3\}}(C)$
      and $F_\emptyset(C)$ are linearly ordered by specialization. However the
      linear map $\nu(x,y)=2y-2x$ defining $C$ is not largely continuous
      on $B$: it has no limit when $(x,y)$ tends to $(+\infty,+\infty)$ in $B$.
      Note that $F_{\{3\}}(C)=\{+\infty\}^2\times 2\NN$ can not be defined
      by linear inequalities. Thus $F_{\{3\}}(C)$ is definitely not a
      polytope, and so neither is $C$. 
  \end{itemize}  
\end{example}

All the references in the next proposition are taken from
\cite{darn-2017}.

\begin{proposition}\label{pr:pres-face}
  Let $q\geq1$ and $A\subseteq F_I(\Gamma^q)$ be a discrete polytope. Let $(\mu,\nu)$ be a
  largely continuous presentation of $A$, let $J$ be a subset of $I$, and
  $\widehat{J}=J\setminus\{q\}$. Finally let $Y=F_{\widehat{J}}(\widehat{A})$. 
  Then $F_J(A)\neq\emptyset$ if and only if either $q\in J$ and $\bar\mu<+\infty$ on $Y$,
  or $q\notin J$ and $\bar\nu=+\infty$ on $Y$ (Proposition~3.11). When this
  happens:
  \begin{enumerate}
    \item
      $F_J(A)=\pi_J(A)$ (Proposition~3.3).
    \item 
      The socle of $F_J(A)$ is a face of $\widehat{A}$: 
      $\widehat{F_J(A)}=F_{\widehat{J}}(\widehat{A})=Y$
      (Proposition~3.7).
    \item\label{it:pres-face}
      $F_J(A)$ is a discrete polytope and $(\bar\mu_{|Y},\bar\nu_{|Y})$ is a
      presentation of it (Proposition~3.11):
      \begin{displaymath}
        F_J(A)=\big\{b\in F_J(\Gamma^q)\tq \widehat{b}\in Y,\mbox{ and }
        \bar\mu(\widehat{b})\leq b_q\leq \bar\nu(\widehat{b})\big\}.
      \end{displaymath}
  \end{enumerate}
\end{proposition}

We will also use the next result (Proposition~3.5 in
\cite{darn-2017}).

\begin{proposition}\label{pr:prol-proj}
  Let $A\subseteq\Gamma^q$ be a discrete polytope, $f:A\to\Omega$ be an affine map and
  $B=F_J(A)=\pi_J(A)$ a face of $A$. Assume that $f$ extends to a
  continuous map $f^*:A\cup B\to\Omega$. Then $f^*$ is affine and if $f^*\neq+\infty$
  then $f=f^*_{|B}\circ{\pi_J}_{|A}$. In particular if $f^*\neq+\infty$ then
  $f(A)=f^*(B)$.
\end{proposition}

A {\df discrete simplex} is a discrete polytope whose faces are
linearly ordered by specialization. This is a ``monohedral largely
continuous precell mod $(1,\dots,1)$'' in \cite{darn-2017}. Of course
every face of a simplex is a simplex (see Remark~3.12 of
\cite{darn-2017}).

For every $M\geq1$ we let $D^MR^q=(R\cap Q_{1,M})^q$ and define {\df
$p$\--adic simplexes of index $M$} as the inverse images of discrete
simplexes by the restriction of the valuation to $D^MR^q$. The faces
of a simplex $S$ of index $M$ are obviously the pre-images in $D^MR^q$
of the faces of $vS$. In particular they are linearly ordered by
specialization. $S$ is closed if and only if $vS$ is a singleton in
$\Gamma^q$. If $S$ is not closed, its largest proper face $T$ is called its
{\df facet} and $\partial S=\overline{T}$. 

\begin{remark}\label{re:face-proj}
  With the notation of Proposition~\ref{pr:prol-proj}, if
  $S=v^{-1}(A)\cap D^MR^q$ and $T=v^{-1}(B)\cap D^MR^q$ then $T=F_J(S)$, and
  so by Proposition~\ref{pr:prol-proj} $T=\pi_J(S)$. We will
  sometimes refer to the restriction of $\pi_J$ to $A$ (resp. $S$) as to
  ``the coordinate projection of $A$ onto $B$ (resp. of $S$ onto
  $T$)''. 
\end{remark}

\begin{example}
  $S=\{(x,y)\in D^1R^2\tq 0<|x|\leq|y|\leq1\}$ is a simplex of index $1$ (it is
  the inverse image in $D^1R^2$ of the discrete simplex $B$ in
  example~\ref{ex:simplex}). Intuitively we can visualise it (more
  exactly its image in $|K|\times|K|$) in the next figure, with its faces
  $F_{\{1\}}(S)=D^1R\times\{0\}$ and $F_\emptyset(S)=\{(0,0)\}$. More general simplexes
  will be defined by triangular systems of inequalities between norms
  of largely continuous monomial functions with rational exponents,
  hence their intuitive representations will usually have curved
  shapes. 
\end{example}

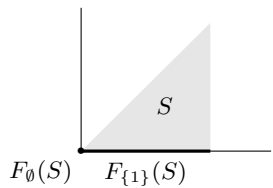
\begin{center}
  \small
  \begin{tikzpicture}
    \fill[color=gray!20] (0,0) -- (1.7,0) -- (1.7,1.7) -- cycle;
    \draw[thin] (0,1.9) -- (0,0) -- (2.5,0); 
    \draw[very thick] (0,0) node{\tiny$\bullet$} 
              -- (1.7,0) node[midway,below]{$F_{\{1\}}(S)$} ;
    \draw (0,0) node[below left]{$F_\emptyset(S)$}
          (1.1,.6) node{$S$};
  \end{tikzpicture}
\end{center}

\subsection{Simplicial complexes}

We will have to consider complexes of sets, of cells and of simplexes.
All of them are finite families of subsets of a topological space $X$,
organised in a such a way that one controls how the closures of these
sets intersect. 

Recall first that an ordered set $\cA$ is a {\df tree} if for every
$A$ in $\cA$, the set of elements in $\cA$ smaller than $A$ is
linearly ordered. It is a {\df rooted tree} if it has one smallest
element. A {\df lower subset} of $\cA$ is a subset $\cB$ of $\cA$ such
that whenever an element of $\cA$ is smaller than an element of $\cB$,
it belongs to $\cB$.

Now, given a finite family $\cA$ of pairwise disjoint subsets of $X$,
we call $\cA$ a {\df closed complex} if every $A\in\cA$ is relatively
open and if its frontier $\partial A$ is a union of elements of $\cA$. The
specialization preorder is then an order on $\cA$. If $\cA$, ordered
by specialization, is a tree (resp. a rooted tree) we call it a {\df
closed monoplex} (resp. {\df rooted closed monoplex}). A {\df complex}
(resp. {\df monoplex}) is then an arbitrary subfamily of a closed
complex (resp. closed monoplex). Of course a complex $\cA$ is a closed
complex if and only if $\bigcup\cA$ is closed.

\begin{remark}\label{re:comp-pure}
  Using that every semi-algebraic set is the disjoint union of
  finitely many pure dimensional ones, and that $\dim \partial A< \dim A$
  for every semi-algebraic set $A$, a straightforward induction shows
  that every finite family of semi-algebraic subsets of $K^m$ can be
  refined by a complex of pure dimensional semi-algebraic sets. 
\end{remark}

A {\df simplicial complex} $\cS$ in $D^MR^q$ (resp. in $\Gamma^q$) is a
complex of simplexes in $D^MR^q$ (resp. in $\Gamma^q$). 

\begin{remark}\label{re:well-dispatched}
  We do not require in our definition of a simplicial complex $\cS$ in
  $D^MR^q$ that different simplexes must have different supports.
  However it will follow from our construction that the simplicial
  complexes produced by $\TR m$ do have this additional property and
  more: for every $S,S'\in\overline{\cS}$, $S'\leq S$ if and only if $\Supp
  S' \subseteq \Supp S$ (see Remark~\ref{re:well-disp-proof}). So the tree $\cS$,
  ordered by specialisation, is isomorphic to the set $\{\Supp S\tq
  S\in\cS\}$ ordered by inclusion.
\end{remark}

Let $\cS$ be a finite family of simplexes in $D^MR^q$ (or $\Gamma^q$). Then
$\cS$ is a simplicial complex if and only if for every $S,T\in\cS$,
$\overline{S}\cap\overline{T}$ is the union of the common faces of $S$
and $T$. When this happens:
\begin{enumerate}
  \item
    $\cS$ is a monoplex;
  \item 
    every subset $\cS_0$ of $\cS$ in $D^MR^q$ is again a simplicial
    complex;
  \item 
    $\bigcup\cS_0$ is closed in $\bigcup\cS$ if and only if $\cS_0$ is a lower
    subset of $\cS$. 
\end{enumerate}

Let $\overline{\cS}$ denote the family of all the faces of the
elements of $\cS$. We call it the {\df closure} of $\cS$, and say that
$\cS$ is {\df closed} if $\cS=\overline{\cS}$. Note that $\cS$ is a
complex (resp. a closed complex) if and only if $\cS\subseteq\overline{\cS}$
(resp. $\cS=\overline{\cS}$) and the elements of $\overline{\cS}$ are
pairwise disjoint. 

If $\cS$ is a simplicial complex, we say $\cT$ is a {\df
simplicial subcomplex} of $\cS$ of if $\cT$ is a simplicial
complex such that $\overline{\cT}$ \emph{refines} a lower subset
of $\overline{\cS}$, and $\bigcup\cT$ is a closed subset of $\bigcup\cS$. 

The following results are respectively Theorem~6.3 and
Proposition~6.4 of \cite{darn-2017}. 

\begin{theorem}[Monotopic Division]\label{th:mono-div}
  Let $S$ be a simplex in $D^MR^q$ and $\cT$ a simplicial complex in
  $D^MR^q$ which is a partition of $\partial S$. Let $\varepsilon:\partial S\to K^\times$ be a
  definable function such that the restriction of $|\varepsilon|$ to every
  proper face of $S$ is continuous. Then there exists a finite
  partition $\cU$ of $S$ such that $\cU\cup\cT$ is a simplicial complex
  in $D^MR^q$, $\cU$ contains for every $T\in\cT$ a unique simplex $U$ with
  facet $T$, and moreover $\big\|u-\pi_J(u)\big\| \leq \big|\varepsilon(\pi_J(u))\big|$
  for every $u\in U$, where $J=\Supp(T)$.
\end{theorem}

\begin{proposition}\label{pr:split-p-adic}
  Let $A\subseteq D^MR^q$ be a relatively open set. Assume that $A$ is the union
  of a simplicial complex $\cA$ in $D^MR^q$. Then for every integer
  $n\geq1$ there exists a finite partition of $A$ in semi-algebraic sets
  $A_1,\dots,A_n$ such that $\partial A_k=\partial A$ for every $k$. 
\end{proposition}

Finally, a {\df simplicial complex of index $M$} is a collection
$\cS=\{\cS_i\}_{i\in I}$ of finitely many\footnote{Possibly zero if the
index set $I$ is empty.} rooted simplicial complexes $\cS_i$ in
$D^MR^{q_i}$, for various integers $q_i$. The {\df closure} of $\cS$
is the collection of the closures of the $\cS_i$'s. It has {\df
separated supports} if each $\cS_i$ is. If $\cT=(\cT_i)_{i\in\cI}$ is a
collection of families $\cT_i$ of subsets of $D^MR^q_i$ we let $\biguplus\cT$
denote the disjoint union of the $\bigcup\cT_i$'s. We say that $\cT$ is a
{\df simplicial subcomplex} of $\cS$ if each $\cT_i$ is a simplicial
subcomplex of $\cS_i$.

Given a semi-algebraic homeomorphism $\varphi$ from $\biguplus\cS$ to a subset
$X$ of $K^m$, we will let
\begin{displaymath}
  \varphi(\cS)=\big\{ \varphi(S)\tq S\in\cS \big\}. 
\end{displaymath}
If $\cS$ is closed, $\varphi(\cS)$ is obviously a closed monoplex of pure
dimensional semi-algebraic sets partitioning $X$.

\begin{remark}\label{re:closed-triang}
  With $\varphi$ as above, $\cS$ is closed if and only if $X$ is closed and
  bounded. Indeed, each $\bigcup\cS_i$ is clopen in $\biguplus\cS$, hence its
  homeomorphic image $X_i$ by $\varphi$ is clopen in $X$. In particular $X$
  is closed and bounded in $K^m$ if and only if so is each $X_i$. Let
  $\varphi_i$ be the semi-algebraic homeomorphism from $\bigcup\cS_i$ to $X_i$
  induced by restriction of $\varphi$. Note that $\bigcup\cS_i$ is bounded (it is
  contained in $R^{q_i}$). By Theorem~\ref{th:extr-val} applied to
  $\varphi_i$ and $\varphi_i^{-1}$ it follows that $\bigcup\cS_i$ is closed in
  $K^{q_i}$, that is $\cS_i$ is closed, if and only if $X_i$ is closed
  and bounded in $K^m$.
\end{remark}

We can now state precisely our main result. 

\begin{theorem}[Triangulation $\TR m$]\label{th:triang-m}
  Given a finite family $(\theta_i:A_i\subseteq K^m\to K)_{i\in I}$ of semi-algebraic
  functions and integers $n,N\geq1$, for some integers $e,M$ which can be
  made arbitrarily large\footnote{The exact meaning of ``$e$, $M$ can
    be made arbitrarily large'' is a bit special here: it says that
    for any given integers $e_*\geq1$ and $M_*\geq1$, the integers $e$, $M$ can
    be chosen so that $e_*$ divides $e$ and $M_*\leq M$.
    \label{ft:arbit-large}}, there exists a simplicial complex $\cT$
  of index $M$ and a semi-algebraic homeomorphism $\varphi$ from the
  disjoint union of the simplexes in $\cT$ to $\bigcup_{i\in I}A_i$ such that
  for every $i$ in $I$:
  \begin{enumerate}
    \item
      $\displaystyle 
      \{ \varphi(T)\tq T\in\cT \mbox{ and } \varphi(T)\subseteq A_i\}$ is a 
      partition of $A_i$.
    \item 
      $\forall T\in\cT$ such that $\varphi(T)\subseteq A_i$, $\theta_i\circ\varphi_{|T}$ is
      $N$\--monomial mod $U_{e,n}$.
  \end{enumerate}
\end{theorem}

We call the pair $(\cT,\varphi)$ given by $\TR m$ a {\df triangulation} of
the $\theta_i$'s with {\df parameters} $(n,N,e,M)$. When a finite family
$(A_i)_{i\in I}$ of semi-algebraic sets is given, the result of the
application of $\TR m$ to the indicator functions of the $A_i$'s is
called a triangulation of $(A_i)_{i\in I}$.

\section{Applications}
\label{se:appli-triang}

In all this section we assume $\TR m$ and derive some applications.
The proof of the Triangulation Theorem goes by induction on $m$, and
most of the following applications are actually needed in the
induction step. So it is important to emphasize that throughout this
section, the integer $m$ will be fixed.

\begin{theorem}\label{th:mod-f-cont}
  If $f:X\subseteq K^m\to K$ is semi-algebraic and $|f|$ is continuous, then there
  exists a function $h:X\to K$ semi-algebraic {\em and} continuous such that
  $|f|=|h|$ on $X$.
\end{theorem}

\begin{proof}
$\TR m$ gives a triangulation $(\cT,\varphi)$ of $f$ with parameters
$(1,1,e,M)$.
On every $S\in\cT$, $f\circ\varphi_{|S}=\cU_{e,1}\psi$ with $\psi:S\to K$ a $1$\--monomial
function. Thus for some $q_S$ such that $S$ is contained in $D^MR^{q_S}$, there are
$\lambda_S$ in $K$ and $\alpha_{1,S},\dots,\alpha_{q_S,S}$ in $\ZZ$ such that:
\begin{equation}\label{eq:mod-f-cont}
  \forall x\in S,\ |f\circ\varphi(x)|=\bigg|\lambda_S\prod_{i=1}^{q_S} x_i^{\alpha_{i,S}}\bigg|
\end{equation}
Let $\alpha_{0,S}=v\lambda_S$ and $\xi_S:vS\to \Gamma$ be defined by:
\begin{displaymath}
  \forall a\in vS,\ \xi_S(a)=\alpha_{0,S} + \sum_{i=1}^{q_S} \alpha_{i,S}a_i
\end{displaymath}
By construction $\xi_S(vx)=vf(\varphi(x))$ for every $x\in S$ (in particular
$\xi_S$ only depends on $f\circ\varphi$, even if the coefficients $\alpha_{i,S}$ in
(\ref{eq:mod-f-cont}) are not uniquely determined by $f\circ\varphi$ on $S$). By
assumption $vf$ is continuous on $X$ hence so is $vf\circ\varphi$ on $\biguplus\cT$. In
particular $\xi_S$ extends continuously to $vT$ for every face $T$ of
$S$ in $\cT$, and the restriction to $vT$ of such an extension
$\overline{\xi}_S$ is
precisely $\xi_T$. By proposition~\ref{pr:prol-proj} it follows that if
$\xi_T\neq+\infty$ (that is if $f\circ\varphi\neq0$ on $T$) then $\xi_S=\xi_T\circ\pi_T$
where $\pi_T$ denotes the coordinate projection of $vS$ to $vT$ (see
Remark~\ref{re:face-proj}).

Now, for every $S$ in $\cT$ let $g_S:S\to K$ be defined (by induction on
$\cT$ ordered by specialization) as follows:
\begin{enumerate}
  \item
    If $f\circ\varphi=0$ on $S$, $g_S=0$. 
  \item 
    If $S$ is minimal (with respect to the specialisation preorder) among
    the simplexes in $\cT$ on which $f\circ\varphi\neq0$ then for every $x\in S$:
    \begin{displaymath}
      g_S(x)=\pi^{\alpha_{0,S}}\prod_{i=1}^{q_S} x_i^{\alpha_{i,S}} 
    \end{displaymath}
  \item 
    Otherwise $g_S=g_T\circ\pi_T$ where $\pi_T$ is the coordinate projection
    (see Remark~\ref{re:face-proj}) of $S$ onto its smallest proper face
    $T$ in $\cT$ on which $f\circ\varphi\neq0$.
\end{enumerate} 
By construction $vg_S(x)=\xi_S(vx)$ for every $x\in S$ hence $|g_S|=|f\circ\varphi|$
on $S$. Moreover for every face $T$ of $S$ in $\cT$ and every $y\in T$,
$g_S(x)$ tends to $g_T(y)$ as $x$ tends to $y$ in $S$ (because
$g_S(x)=g_T(\pi_T(x))$ if $g_T\neq0$, and otherwise because $|g_S|=|f\circ\varphi|$
on $S$, $|g_T|=|f\circ\varphi|=0$ on $T$ and $|f\circ\varphi|=|f|\circ\varphi$ is continuous by
assumption). 

The function $h:X\to K$ defined by $h=g_S\circ\varphi^{-1}$ on every $\varphi(S)$ with
$S$ in $\cT$, is clearly semi-algebraic. By construction $|f|=|h|$ on $X$,
and by the above argument $h$ is continuous on $X$.
\end{proof}

\begin{theorem}\label{th:retraction}
  For all non-empty semi-algebraic sets $Y\subseteq X\subseteq K^m$, there is a
  semi-algebraic retraction of $X$ onto $Y$ if and only if $Y$ is
  closed in $X$. 
\end{theorem}

\begin{proof}
One direction is general. For the converse we assume that $Y$ is
closed in $X$. Let $(\cS,\varphi)$ be a triangulation of $X$, $Y$ given by
$\TR m$, and let $\cT$ be the family of simplexes $T$ in $\cS$ such that
$\varphi(T)\subseteq Y$. It suffices to construct a continuous retraction of $\biguplus\cS$
onto $\biguplus\cT$. 

Let $\cS_0=\cT$ and $\sigma_0$ be the identity map on $\biguplus\cT$. Because $Y$
is closed in $X$, $\cT$ is a lower subset of $\cS$. Let $k$ be a
positive integer and assume that there is a lower subset $\cS_{k-1}$
of $\cS$ containing $\cT$, and a retraction $\sigma_{k-1}$ of $\biguplus\cS_{k-1}$
to $\cT$. If $\cS_{k-1}=\cS$ we are done. Otherwise let $S$ be a
minimal element (with respect to the specialisation order) in
$\cS\setminus\cS_{k-1}$, and let\footnote{We are abusing the notation here:
  $\cS$ is a finite collection of simplicial simplexes $\cS^{(i)}$ in
  $D^MR^{q_i}$ for various $q_i$, $\cS_{k-1}$ is a collection of lower
  subsets $\cS_{k-1}^{(i)}$ of $\cS^{(i)}$, there is an index $i_0$ such that
  $S$ belongs to $\cS^{(i_0)}$, and what we have denoted abusively
  $\cS_{k-1}\cup\{S\}$ is actually the collection of all the
  $\cS_{k-1}^{(i)}$'s
  for $i\neq i_0$ and of $\cS_{k-1}^{(i_0)}\cup\{S\}$.} $\cS_k=\cS_{k-1}\cup\{S\}$. It
only remains to build a retraction $\tau$ of $\biguplus\cS_k$ onto
$\biguplus\cS_{k-1}$. Indeed $\sigma_{k-1}\circ\tau$ will then be a continuous retraction
of $\cS_k$ onto $\cT$, and the result will follow by induction. 

If $S$ has no proper face in $\cS$ then it is clopen in $\biguplus\cS_k$. So
the map $\tau$ which is the identity map on $\biguplus\cS_k$ and which sends
every point of $S$ to an arbitrary given point of $\biguplus\cS_{k-1}$ is
continuous on $\biguplus\cS_k$, and a retraction of $\biguplus\cS_k$ onto $\biguplus\cS_{k-1}$.

Otherwise let $T$ be the largest proper face of $S$ in $\cS$. By
minimality of $S$, $T$ belongs to $\cS_{k-1}$. Let $\pi_T$ be the
coordinate projection of $S$ onto $T$. The frontier of $S$ inside
$\biguplus\cS_k$ is the closure of $T$ in $\biguplus\cS_k$, hence the function $\tau$
which coincides with the identity map on $\biguplus\cS_{k-1}$ and with $\pi_T$ on
$S$ is continuous. It is then a retraction $\biguplus\cS_k$ onto $\biguplus\cS_{k-1}$,
which finishes the proof. 
\end{proof}

The Splitting Theorem~\ref{th:split-def} is a strengthening of the
next lemma using retractions. 

\begin{lemma}\label{le:split-level}
  Let $X\subseteq K^m$ be a relatively open semi-algebraic set without
  isolated points and $n\geq1$ an integer. Then there exists a partition
  of $X$ in semi-algebraic sets $X_k$ for $1\leq k\leq n$ such that $\partial X_k=\partial
  X$ for every $k$.
\end{lemma}

We are going to prove Lemma~\ref{le:split-level} by using a
triangulation $(\cU,\varphi)$ of $(X,\partial X)$ and applying
Proposition~\ref{pr:split-p-adic} to $\varphi^{-1}(X)$. In order to ensure
that this set is still relatively open, we first reduce to the case
where $X$ is bounded by means of the following construction. 

Let $\tilde K=K\cup\{\infty\}$ and for every $I\subseteq\{1,\dots,m\}$ let $K_I^m=\tilde K_I^m\cap
K^m$ where
\begin{displaymath}
\tilde K_I^m = \{x\in\tilde K^m\tq x_k\in R\iff k\in I\}. 
\end{displaymath}
Let $R_I^m=\{x\in R^m\tq \forall k\notin I,\ x_k\neq0\}$, and for every $x\in R_I^m$ let
$\psi_I(x)=(y_k)_{1\leq k\leq m}$ be defined by $y_k=x_k$ if $k\in I$, and
$y_k=1/(\pi x_k)$ otherwise. Clearly $\psi_I$ is semi-algebraic
homeomorphism from $R_I^m$ to $K_I^m$ which extends uniquely to a
homeomorphism $\tilde\psi_I$ from $R^m$ to $\tilde K_I^m$.

\begin{proof}
Note first given a partition of $X$ in finitely many
semi-algebraic pieces $U_1,\dots,U_r$ which are clopen in $X$, it
suffices to prove the result separately for each $U_j$. Indeed, each
$U_j$ will then be relatively open with $\partial U_j\subseteq \partial X$ (because $U_j$ is
clopen in $X$), and $\bigcup_{j\leq r}\partial U_j =\partial X$ (because $\partial
U_j=\overline{U}_j\setminus X$ and $\bigcup_{j\leq r}\overline{U}_j=\overline{X}$).
So, if a partition of each $U_j$ in semi-algebraic pieces
$(U_{j,k})_{1\leq k\leq n}$ is found such that $\partial U_{j,k}=\partial U_j$ for every
$k$, then the union $X_k$ of $U_{j,k}$ for $1\leq j\leq r$ defines a
partition of $X$ in semi-algebraic pieces and we have $\partial X_k=\bigcup_{j\leq
r}\partial U_{j,k}$ (same argument as above) hence $\partial X_k=\bigcup_{j\leq r}\partial U_j=\partial
X$. 

Now, as $I$ ranges over the subsets of $\{1,\dots,m\}$, the sets $X\cap K_I^m$
form a partition of $X$ in semi-algebraic sets clopen in $X$. By the
argument above we can deal separately with each of these sets, hence
we can reduce to the case where $X\subseteq K_I^m$ for some $I$.

Let $Y=\psi_I^{-1}(X)$ and $\hat X$ be the closure of $X$ in $\hat K_I^m$.
Note that $\hat\psi_I(\overline{Y})=\hat X$. The fact that
$\overline{X}\setminus X$ is closed in $K^m$, hence in $K_I^m$, implies that
$\hat X\setminus X$ is closed in $\hat K_I^m$. It follows that its image under
$\hat \psi_I^{-1}$, which is precisely $\overline{Y}\setminus Y$, is closed in
$R^m$, hence in $K^m$. Thus $Y$ is relatively open. It then suffices
to prove the result for $Y$, that is we can assume that $X=Y$ is
bounded. Of course we can assume as well that $\partial X$ is non-empty
(otherwise $X_1=X$ and $X_k=\emptyset$ for $2\leq k\leq n$ is obviously a solution). 

$\TR m$ gives a triangulation $(\cU,\varphi)$ of $(X,\partial X)$. $\cU$ is the
disjoint union of finitely many simplicial complexes $\cU_j$ in
$D^MR^{q_j}$ for $1\leq j\leq r$. Let $U_j=\varphi(\bigcup\cU_j)\cap X$ for every $j$,
this defines a partition of $X$ in semi-algebraic sets clopen in $X$.
By using again the initial remark in this proof, it suffices to check the
result for each $U_j$ separately. So we can assume that $\cU$ itself
is a simplicial complex in $D^MR^q$ for some $q$. 

By construction $\overline{X}$ is semi-algebraic, closed and bounded,
and $\varphi^{-1}$ is semi-algebraic and continuous, so
$\varphi^{-1}(\overline{Y})=\overline{\varphi^{-1}(Y)}$ for every
semi-algebraic\footnote{For every continuous map $f:X\subseteq K^q\to K^r$
  and every $Y\subseteq X$, if $X$ is closed then
  $f(\overline{Y})\subseteq\overline{f(Y)}$. The reverse inclusion holds 
  if $X$ is compact, or if $f$, $Y$, $X$ are semi-algebraic and $X$ is
  closed and bounded (see
  Theorem~\ref{th:extr-val}).\label{fn:image-compact}}
 $Y\subseteq\overline{X}$. Let $A=\varphi^{-1}(X)$, we have
$\overline{A}=\varphi(\overline{X})$ hence $\overline{A}\setminus A=\varphi(\overline{X}\setminus
X)$ is closed, that is $A$ is relatively open.
Proposition~\ref{pr:split-p-adic} then applies to $A$ and gives a
partition of $A$ in semi-algebraic sets $A_1,\dots,A_n$ such that $\partial A_k=\partial
A$ for every $k$. 

For $1\leq k\leq n$ let $X_k=\varphi(A_k)$. These semi-algebraic sets form a
partition of $X$, because $A_1,\dots,A_n$ form a partition of $A$.
Moreover, since $\bigcup\cU$ is semi-algebraic, closed and bounded,
we have $\varphi(\overline{B})=\overline{\varphi(B)}$ for every
semi-algebraic\footnote{See footnote~\ref{fn:image-compact}} set
$B$ contained in $\bigcup\cU$. It follows that for $1\leq k\leq n$ we have
$\partial X_k=\varphi(\partial A_k)=\varphi(\partial A)=\partial X$, which proves the result.
\end{proof}

\begin{theorem}\label{th:split-def}
  Let $X$ be a relatively open non-empty semi-algebraic subset of
  $K^m$ without isolated points, and $Y_1,\cdots,Y_s$ a collection of
  closed semi-algebraic subsets of $\partial X$ such that $Y_1\cup\cdots\cup Y_s=\partial X$.
  Then there is a partition of $X$ in non-empty semi-algebraic sets
  $X_1,\dots,X_s$ such that $\partial X_i= Y_i$ for $1\leq i\leq s$.
\end{theorem}

\begin{proof} 
$X$ is non-empty and has no isolated point, hence is infinite. The
result is obvious for $s=0$ (there is nothing to prove) and $s=1$
(take $X_1=X$). By induction it suffices to prove it for $s=2$.
Indeed, if $s\geq 3$ and the result is proved for $s-1$, then the result
for $s=2$ applied to $X$ with $Z_1=Y_1\cup\cdots\cup Y_{s-1}$ and $Z_2=B_s$ gives
a partition in two pieces $X'_1$, $X'_2$ such that $\partial X'_l=Z_l$ for
$l=1,2$, and the induction hypothesis applied to $X'_1$ with
$Y_1,\dots,Y_{s-1}$ gives a partition of $X'_1$ in pieces $X_1,\dots,X_{s-1}$ such
that $\partial X_i=Y_i$ for $1\leq i\leq s$. The conclusion follows, by taking
$X_s=X'_2$. So from now on we assume that $s=2$. 

It suffices to prove the weaker result that a partition $(X'_1,X'_2)$
exists with all the required properties for $(X_1,X_2)$ except possibly
the condition that they are non-empty. Indeed, if such a
partition is found and for example $X'_2=\emptyset$ then necessarily $Y_2=\partial
X'_2=\emptyset$. In that case pick any $x\in X$, and choose a clopen neighbourhood $V$
of $x$ such that $V\cap \partial X$ is empty (this is possible because $X$ is
relatively open). Then $X_1=X\setminus V$ and $X_2=X\cap V$ give the
conclusion. 

Let $\rho:\overline{X}\to \partial X$ be a continuous retraction of $\overline{X}$
onto $\partial X$ given by Theorem~\ref{th:retraction}. Let $V\subseteq\partial X$ be any
semi-algebraic set open in $\partial X$, $Z$ its closure and $A=\rho^{-1}(Z)\cap X$. We
are claiming that $\partial A = Z$. Note that $A$ is closed in $X$ by
continuity of $\rho$, because $A$ is the inverse image of the closed set
$Z$ by $\rho_{|X}$. So it suffices to prove that $\overline{A}\cap\partial X=Z$, or
equivalently that $\overline{A}\cap\partial X$ contains $V$ and is contained in
$Z$. For the first inclusion let $y$ be any element of $V$, and $W$
any neighbourhood of $y$. We have to prove that $W\cap A\neq\emptyset$. By
continuity of $\rho$ at $y=\rho(y)$ there is a neighbourhood $U$ of $y$ such
that $U\cap\overline{X}$ is contained in $\rho^{-1}(W\cap V)$. In particular
\begin{displaymath}
  U\cap W\cap X\subseteq U\cap\overline{X}\subseteq \rho^{-1}(W\cap V) \subseteq \rho^{-1}(V)=A
\end{displaymath}
so $U\cap W\cap A=U\cap W\cap X$. 
On the other hand, $U\cap W\cap X\neq\emptyset$ because $U\cap W$ is a neighbourhood of $y$
and $y\in V\subseteq \overline{X}$. {\it A fortiori} $W\cap A$ is non-empty. This
proves that $y\in\overline{A}$, hence that $V\subseteq\overline{A}\cap\partial X$.
Conversely, if $y'$ is any element of $\partial X\setminus Z$, there is a
neighbourhood $W'$ of $y'$ such that $W'\cap \partial X$ is disjoint from $Z$.
By continuity of $\rho$, $\rho^{-1}(W)$ is then a neighbourhood of $y'$ in
$\overline{X}$. It is disjoint from $\rho^{-1}(Z)=A$ hence
$y'\notin\overline{A}$. So $\overline{A}$ is disjoint from $\partial X\setminus Z$. That
is $\overline{A}\cap\partial X\subseteq Z$, which proves our claim.

Let $Z_1=\overline{Y_1\setminus Y_2}$ and $Z_2=\overline{Y_2\setminus Y_1}$. For
$k=1,2$ let $A_k=\rho^{-1}(Z_k)$. Let $Z_0$ be the closure of $\partial X\setminus(Z_1\cup
Z_2)$ and $A_0=\rho^{-1}(Z_0)$. The above claim gives that $\partial A_k= Z_k$
for $0\leq k\leq 2$. Let $B_0$ be the set of non-isolated points of $A_0$.
Clearly $\partial B_0 = \partial A_0 = Z_0$ since $A_0\setminus B_0$ is finite. In
particular $B_0$ is relatively open, and Lemma~\ref{le:split-level}
gives two semi-algebraic sets $B_1$, $B_2$ partitioning $B_0$ such
that $\partial B_1=\partial B_2=Z_0$. So if we set $X_1=A_1\cup B_1$ and $X_2=A_2\cup
B_2\cup(A_0\setminus B_0)$ we get the conclusion. 
\end{proof}

\begin{theorem}\label{th:piec-str-cont}
  Let $f:X\subseteq K^m\to K$ be a semi-algebraic function with bounded graph
  (that is $f$ is a bounded function on a bounded domain). If
  it has finitely many limit values at every point of $\overline{X}$
  then $f$ is piecewise largely continuous.
\end{theorem}

Note that the counterpart of Theorem~\ref{th:piec-str-cont} for
real-closed fields holds. Indeed, by triangulation we can reduce
to the case of a continuous function $f$ on a simplex $S\subseteq K^m$. The
assumption that $f$ has finitely many limit values at every point of
$\overline{S}$ then implies directly that $f$ is largely continuous.
Indeed, this follows easily from the fact that over real-closed fields
the direct image by a continuous semi-algebraic map of any
semi-algebraically connected set (such as $S\cap B$ with $B$ a ball
centered at any point of $\overline{S}$) is again semi-algebraically
connected.

On the contrary, $p$\--adic simplexes are not at all
semi-algebraically connected and it can happen that a function
satisfying all these assumptions on a $p$\--adic simplex is not
largely continuous. For example on the simplex $S=D^1R^*$ the
semi-algebraic function $f$ defined by $f(x)=0$ if $v(x)\in2\cZ$ and
$f(x)=1$ otherwise is a continuous, bounded function having two
distinct limit values at $0$. Thus $f$ is not largely continuous. It
is obviously piecewise largely continuous, though.

\begin{proof}
Every semi-algebraic function is piecewise continuous (see for example
\cite{mour-2009}). So, replacing $f$ by its restriction to the pieces
of an appropriate partition of $X$ if necessary, we can assume that
$f$ is continuous. Removing $X\cap\partial X$ if necessary (using a
straightforward induction on $\dim X$ and the fact that $\dim \partial X<\dim
X$) we can even assume that $X$ is relatively open. The proof
then goes by induction on the lexicographically ordered tuples
$(e,e')$ where $e=\dim X$ and $e'=\dim\partial X$. If $\partial X$ is empty,
that is $X$ is closed, then $f$ is largely continuous and the result
is obvious. So let us assume that $e'\geq0$ (hence $e\geq 1$)  and the
result is proved for smaller tuples $(e,e')$. 

Let $D=(\partial X\times K)\cap\overline{\Gr f}$. The projection of $D$ onto $\partial X$
has finite fibers hence $D$ is a union of cells of type~$0$. The
number of these cells, say $N$, then bounds the cardinality of these
fibers, that is the number of limit values of $f$ at every point of $\partial
X$. For every $a\in\partial X$ let $D_a=\{t\in K\tq (a,t)\in D\}$. We first show
that $\widehat{D}=\partial X$, that is $D_a\neq\emptyset$ for every $a\in \partial X$. For every
$\varepsilon\in R^*$ let $C_\varepsilon=(B(a,\varepsilon)\times K)\cap\overline{\Gr f}$. This is a uniformly
semi-algebraic family of closed and bounded semi-algebraic subsets of
$K^n$. Each of them is non-empty because $C_\varepsilon$ contains $(x,f(x))$
for any $x$ in $B(a,\varepsilon)\cap X$ (which is non-empty since $a\in\partial X$).
Obviously $C_{\varepsilon_1}\subseteq C_{\varepsilon_2}$ whenever $|\varepsilon_1|\leq|\varepsilon_2|$, so $\bigcap_{\varepsilon\in
R^*}C_\varepsilon$ is non-empty by Theorem~\ref{th:comp-inter}. This last set
is equal to $D_a$, which proves our claim. 

For $1\leq i\leq N$ let $W_i$ be the set of $a\in \partial X$ such that $D_a$ has
exactly $i$ elements. These sets $W_i$ form a partition of $D$ in
semi-algebraic pieces. By Theorem~\ref{th:skol} (and a straightforward
induction) there are semi-algebraic functions $f_{i,j}:W_i\to K$ such
that $D_a=\{f_{i,j}(a)\}_{1\leq j\leq i}$ for every $a\in D_a$. Since $\dim \partial
X<\dim X$, by the induction hypothesis these functions $f_{i,j}$ are
piecewise largely continuous. This gives a partition of $\partial X$ in
semi-algebraic pieces $V_k$ for $1\leq k\leq r$, and a family of largely
continuous semi-algebraic functions $g_{k,l}:V_k\to K$ for $1\leq l\leq s_k$
such that $V_k\subseteq W_{s_k}$ and $D$ is the union of the graphs of all
these functions $g_{k,l}$. 

Theorem~\ref{th:split-def} applied to $X$ and the sets
$\overline{V_k}$ for $1\leq k\leq r$ gives a partition of $X$ in
semi-algebraic pieces $X_k$ such that $\partial X_k=\overline{V_k}$. It
suffices to prove that the restrictions of $f$ to each $X_k$ is
piecewise largely continuous. So we can assume that $r=1$ and $X=X_1$.
That is, we have a semi-algebraic set $V=V_1$ dense in $\partial X$ and
largely continuous functions $g_l=g_{1,l}:V\to K$ for $1\leq l\leq s=s_1$ such
that $D_a=\{g_{l}(a)\}_{1\leq l\leq s}$ has $s$ elements for every $a\in D_a$.
Replacing $V$ by $V\setminus \partial V$ if necessary we can assume that $V$ is
relatively open.

Let $\rho:\overline{X}\to\overline{V}$ be a continuous retraction given by
Theorem~\ref{th:retraction}. For $1\leq l\leq s$ let 
\begin{displaymath}
  U_l=\big\{x\in X\tq \forall k\neq l,\ 
  \big|f(x)-\bar g_l\big(\rho(x)\big)\big| <
  \big|f(x)-\bar g_k\big(\rho(x)\big)\big| \big\}.
\end{displaymath}
Each $U_l$ is open in $X$ by continuity of $f$, $\rho$ and the $\bar g_k$'s.
Their complements $X'=X\setminus\bigcup_{l=1}^s U_l$ are closed in $X$, hence $\partial X'\subseteq \partial
X$. Moreover, for every $a\in V$, the limit values of $f$ at
$a$ being by construction the pairwise distinct $g_l(a)$ for $1\leq l\leq s$, there exists
$\varepsilon\in R^*$ such that every point of $B(a,\varepsilon)\cap X$ belongs to one
of the $U_l$'s. In other words $B(a,\varepsilon)\cap X'=\emptyset$ hence $a$ does not
belong to the closure of $X'$. So $\partial X\subseteq \partial X\setminus V=\partial V$, in particular
$\dim \partial X'< \dim V=\dim \partial X$ hence the induction hypothesis applies to
the restriction of $f$ to $X'$. 

It only remains to check that the restrictions of $f$ to each $U_l$
are piecewise largely continuous. We are claiming that $f$ has only
one limit value at every point $a$ of $\overline{U}_l\setminus \partial V$. Note that
$\overline{U}_l$ is the disjoint union of $\overline{U}_l\cap X$ and
$\overline{U}_l\cap \partial X$, and that $\partial X=V\cup \partial V$. Obviously, if
$a\in\overline{U}_l\cap X$ then by continuity of $f$, $f(x)$ tends to
$f(a)$ as $x$ tends to $a$ in $U_l$. Now if $a\in (\overline{U}_l\setminus \partial V)\setminus
X$ then $a\in V$, $\rho(a)=a$ and $\bar g_k(\rho(a))=g_k(a)$ for every $k$.
Hence by definition of $U_l$, $f(x)$ is closer to $g_l(a)$ than to
every other $g_k(a)$, so $g_l(a)$ is the only possible limit value of
$f(x)$ as $x$ tends to $a$ in $U_l$, which proves our claim. So the
semi-algebraic function $g$ which coincides with $f$ on
$\overline{U}_l\cap X$ and with $g_l$ on $V$ is continuous. The frontier
of its domain is contained in $\overline{U_l}\cap\partial X\subseteq\partial X=V\cup\partial V$ and is
disjoint from $V$, hence is contained in $\partial V$. By the induction
hypothesis, $g$ is then piecewise largely continuous, hence so is
$f_{|U_l}$ since $f$ and $g$ coincide on $U_l$.
\end{proof}

\section{Largely continuous cell decomposition}
\label{se:cell-dec}

This section recalls the main theorem of \cite{dene-1984} in order to
emphasize some details which appear only in its proof. These details
are important for us because they ensure that the functions defining
the cells involved in the conclusions inherit certain properties,
defined below, from the functions in the assumptions. Using them we
are going to derive from $\TR m$ a new preparation theorem for
semi-algebraic functions ``up to a small deformation''
(Theorem~\ref{th:large-cont-prep}). The point is that after such a
deformation, we get a Cell Preparation Theorem involving only cells
defined by largely continuous functions. 

In order to do so, it is crucial for us to control the boundary of any
cell $C$ we are dealing with. Ideally, we would like it to decompose
naturally in cells defined by functions obtained for the functions
defining $C$ by passing to the limits, just as it is done for the
faces of discrete polytopes (Item~\ref{it:pres-face} of
Proposition~\ref{pr:pres-face}). With this aim in mind, we now
introduce a sharper notion of cell mod $\GG$, for any clopen
semi-algebraic subgroup $\GG$ of $K^\times$ with finite index. 

A {\df presented cell $A$ mod $\GG$} in $K^{m+1}$ is a tuple
$(c_A,\nu_A,\mu_A,G_A)$ with $c_A$ a semi-algebraic function on a
non-empty domain $X\subseteq K^m$ with values in $K$ (called the {\df center}
of $A$), $\nu_A$ and $\mu_A$ either semi-algebraic functions on $X$ with
values in $K^\times$ or constant functions on $X$ with values $0$ or $\infty$
(called the {\df bounds} of $A$), and $G_A$ an element of $K/G$
(called the {\df coset} of $A$), having the property that for every
$x\in X$ there is $t\in K$ such that:
\begin{equation}
  |\nu_A(x)|\leq|t-c_A(x)|\leq|\mu_A(x)|\quad\mbox{and}\quad t-c_A(x)\in G_A
  \label{eq:def-cell}
\end{equation}
We say that $A$ is {\df largely continuous} if its center and bounds
are. In any case the set of tuples $(x,t)\in X\times K$ satisfying
(\ref{eq:def-cell}) is a cell, in the general sense given in the
introduction. When we want to distinguish this set from the presented
cell $A$ we call it the {\df cellular set underlying $A$}.
Nevertheless, abusing the notation, we will also denote it $A$ most
often. The conditions enumerated above (\ref{eq:def-cell}) ensure that
the domain $X$ of $c_A$, $\mu_A$, $\nu_A$ is exactly the socle
$\widehat{A}$ of $A$. When two presented cells $A$ and $B$ have the
same underlying cellular set we write it $A\simeq B$.

From now onwards we will use the word ``{\df cell}'' mostly for
\emph{presented cells} but also very often for the \emph{underlying
cellular sets}, the difference being clear from the context. For
instance we will freely talk of disjoint (presented) cells, of bounded
(presented) cells, of (presented) cells partitioning some set and so
on, meaning that the corresponding cellular sets have these
properties. Also for any $Z\subseteq\widehat{A}$ we will write $A\cap(Z\times K)$ both
for this (cellular) set and for the presented cell
$(c_{A|Z},\nu_{A|Z},\mu_{A|Z},G_A)$. The latter will also be denoted
$(c_A,\nu_A,\mu_A,G_A)_{|Z}$. Similarly $\Gr c_A$ both denotes the graph
of $c_A$ and the presented cell $(c_A,0,0,\{0\})$.

A presented cell $A$ is of {\df type} $0$ if $G_A=\{0\}$, of type $1$
otherwise. The type of $A$ is denoted $\Tp A$. We say that $A$ is
{\df well presented} if either $v\nu_A-v\mu_A$ is unbounded or $\nu_A=\mu_A$.
We call $A$ a {\df fitting cell} if it has {\df fitting bounds}, that
is, for every $x\in\widehat{A}$:
\begin{displaymath}
  |\mu_A(x)|=\sup\{|t-c_A(x)|\tq (x,t)\in A\} 
\end{displaymath}
\begin{displaymath}
  |\nu_A(x)|=\inf\{|t-c_A(x)|\tq (x,t)\in A\}
\end{displaymath}

Sometimes it will be convenient to write $G_A=\lambda_A\GG$ for some $\lambda_A\in
G_A$. We will always do this uniformly, so that $\lambda_A=\lambda_B$ whenever
$G_A=G_B$. To that end a set $\Lambda_\GG$ of representatives of $K/\GG$ is
fixed once and for all, and when we consider a presented cell $A$ mod
$\GG$ it is understood that $\lambda_A$ is the unique element of
$G_A\cap\Lambda_\GG$. In addition, we require from this set of representatives
that every $\lambda\in\Lambda_\GG$ has the smallest possible positive valuation. In
particular if $\GG=\PN_N^\times$ or $Q_{N,M}^\times$ and $A$ is a cell mod $\GG$
of type~$1$ then $0\leq v\lambda_A<N$.
\\

For every family $\cA$ of presented cells in $K^{m+1}$ we
let\footnote{Here the letters $\CB$ stand for ``center and
  boundaries''.}
$\CB(\cA)$ denote the family of all the functions $c_A$, $\mu_A$, $\nu_A$
for $A\in\cA$. Given another family $\cD$ of presented cells in
$K^{m+1}$ we say that:
\begin{enumerate}
  \item
    $\cD$ belongs to $\Vect\cA$ if for every $D\in \cD$, $c_D,\nu_D$ are
    $K$\--linear combinations of functions $f_{|\widehat{D}}$ for
    $f\in\CB(\{A\in\cA\tq \widehat{D}\subseteq\widehat{A}\})$, and either $\mu_D$ is
    such a linear combination as well or $\mu_D=\infty$.
  \item \label{it:D1-n}
    $\cD$ belongs to $\Alg_n\cA$ if $\cD$ is finer than $\cA$ and for
    every $A\in\cA$, every $D\in\cD$ contained in $A$ and every $(x,t)\in D$:
    \begin{enumerate}
      \item\label{it:D1-n-good}
        either $t-c_A(x)=\cU_n(x,t)(t-c_D(x))$;
      \item \label{it:D1-n-bad}
        or $t-c_A(x)=\cU_n(x,t)h_{D,A}(x)$ where $h_{D,A}:\widehat{D}\to
        R$ is the product of (finitely many) linear combinations of
        functions ${c_B}_{|\widehat{D}}$ such that $B\in\cA$ and
        $\widehat{D}\subseteq\widehat{B}$.
    \end{enumerate}
\end{enumerate}

These somewhat cumbersome definitions help us to express Denef's Cell
Decomposition Theorem in a slightly more precise way than in
\cite{dene-1984}.

\begin{theorem}[Denef]\label{th:D1}
  Given a semi-algebraic subgroup $\GG$ of $K^\times$ with finite index,
  let $\cA$ be a finite family of presented cells mod $\GG$ in
  $K^{m+1}$. Then for every positive integer $n$ there exists a finite
  family $\cD$ of fitting cells mod $\GG$ refining $\cA$ such that
  $\widehat{\cD}$ is a partition of $\bigcup\widehat{\cA}$ and $\cD$
  belongs to $\Vect\cA$ and to $\Alg_n\cA$.
\end{theorem}

This is essentially theorem~7.3 of \cite{dene-1984}. Indeed, for any
given integer $N$, if $n$ is large enough then $1+\pi^nR\subseteq \PN_N\cap R^\times$.
Hence $\cU_n(x,t)$ in conditions~(\ref{it:D1-n-good}),
(\ref{it:D1-n-bad}) of the definition of $\Alg_n\cA$ can be written
$u(x,t)^N$ with $u$ a semi-algebraic function from $A$ to $R^\times$
(thanks to Theorem~\ref{th:skol}). This is how the above result is
stated in \cite{dene-1984} with $\GG=K^\times$. Our slightly more precise
form, as well as the additional properties involving $\Vect\cA$ and
$\Alg_n\cA$, appear only in the proof of theorem~7.3 in
\cite{dene-1984} (still with $\GG=K^\times$). The generalization to fitting
cells mod an arbitrary clopen semi-algebraic group $\GG$ with finite
index in $K^\times$ is straightforward\footnote{Here is a sketchy proof.
  For each $A\in\cA$ let $B_A$ be the cell mod $K^\times$ with the same
  center of bounds as $A$. Denef's construction applied to the family
  $\cB$ of all these cells $B_A$ gives a family $\cC$ of cells mod
  $K^\times$ refining $\cB$. Each $C$ in $\cC$ is the union of a finite
  family $\cD_C$ of cells mod $\GG$ with the same center and bounds as
  $C$, each of which is clopen in $C$ (because $\GG$ is clopen in
  $K^\times$ with finite index). For each $A\in\cA$ let $\cD_A$ be the family
  all the cells in $\bigcup\{D_C\tq C\in\cC\}$ contained in $A$. The family
$\cD=\bigcup\cD_A\tq A\in\cA\}$ gives the conclusion.}.

Given a polynomial function $f$, we say that a function
$h:X\subseteq K^m\to K$ belongs to $\coalg(f)$ if there exists a finite
partition of $X$ into definable pieces $H$, on each of which the degree
in $t$ of $f(x,t)$ is constant, say $e_H$, and such that the following
holds. If $e_H\leq 0$ then $h(x)$ is identically equal to $0$ on $H$.
Otherwise there is  a family $(\xi_1,\dots,\xi_{r_H})$ of $K$\--linearly
independent elements in an algebraic closure of $K$ and a family of
definable functions $b_{i,j}:H\to K$ for $1\leq i\leq e_H$ and $1\leq j\leq r_H$,
and $a_{e_H}:H\to K^*$ such that for every $x$ in $H$
\begin{displaymath}
  f(x,T)=a_{e_H}(x)\prod_{1\leq i\leq e_H}
         \bigg(T-\sum_{1\leq j\leq r_H}b_{i,j}(x)\xi_j\bigg)
\end{displaymath}
and
\begin{displaymath}
  h(x)=\sum_{1\leq i\leq e_H}\sum_{1\leq j\leq r_H}\alpha_{i,j}b_{i,j}(x)
\end{displaymath}
with the $\alpha_{i,j}$'s in $K$. If $\cF$ is any family of
polynomial functions we let $\coalg(\cF)$ denote the set of
linear combinations of functions in $\coalg(f)$ for $f$ in $\cF$.

\begin{theorem}[Denef]\label{th:D2}
  Let $\cF\subseteq K[X,T]$ be a finite family of polynomials, with $X$ an
  $m$\--tuple of variables and $T$ one more variable. Let
  $N\geq1$ be an integer and $\cA$ a family of boolean combinations of
  subsets of the form $f^{-1}(\PN_N)$ with $f\in\cF$. For every integer
  $n\geq 1$ there is a finite family of fitting cells mod $\PN_N^\times$ 
  refining $\cA$, with center and bounds in $\coalg(\cF)$, and for
  every such cell $H$ a positive integer $\alpha_{f,H}$ and a semi-algebraic
  function $h_{f,H}:\widehat{H}\to K$ such that for every $(x,t)\in H$:
  \begin{displaymath}
    f(x,t)=\cU_n(x,t)h_{f,H}(x)(t-c_H(x))^{\alpha_{f,H}}. 
  \end{displaymath}
\end{theorem}

\begin{proof}
  W.l.o.g. we can assume that every $f$ in $\cF$ is non constant and
  that $n$ is large enough so that $1+\pi^nR\subseteq \PN_N^\times$. Theorem~7.3 in
  \cite{dene-1984} gives a finite family of cells $B$ mod $K^\times$
  partitioning $K^m$, and for each of them a positive integer
  $\alpha_{f,B}$ and semi-algebraic functions $u_{f,B}:B\to R^\times$ and
  $h_{f,B}:\widehat{B}\to K$ such that:
  \begin{equation}\label{eq:D2}
    \forall(x,t)\in B,\ f(x,t)=u_{f,B}(x,t)^Nh_{f,B}(x)(t-c_B(x))^{\alpha_{f,B}} 
  \end{equation}
  Moreover the functions $u_{f,B}^N$ constructed in the proofs of
  lemma~7.2 and theorem~7.3 in \cite{dene-1984} are precisely of the
  form $1+\pi^n\omega_{f,B}$ for some semi-algebraic function $\omega_{f,B}$ on
  $B$, and the functions $c_B$, $\mu_B$, $\nu_B$ constructed there belong
  to $\coalg(\cF)$. Refining the socle of $B$ if necessary we can
  ensure that $h_{f,B}(x)\PN_N^\times$ is constant as $(x,t)$ ranges over
  $B$. On the other hand $B$ splits into finitely many cells mod
  $\PN_N^\times$, with the same center and bounds as $B$, because $\PN_N^\times$
  has finite index in $K^\times$. On each of these cells $H$,
  $f(x,t)\PN_N^\times$ is constant by (\ref{eq:D2}). Hence $H$ is either
  contained or disjoint from $A$, for every $A\in\cA$. So the family of
  all these cells $H$ which are contained in $\bigcup\cA$ gives the
  conclusion. 
\end{proof}

Using that every semi-algebraic function is piecewise continuous, the
cells mod $\PN_N^\times$ given by Theorem~\ref{th:D2} can easily be chosen
with continuous center and bounds. However it is not possible to
ensure that they are largely continuous (think of the case
where $\cA$ consists of a single semi-algebraic set which is itself
the graph of a semi-algebraic function which is not largely
continuous). Our aim, in the remainder of this section, is to find a
work-around. We are going to prove that it can be done, not exactly
for $\theta$ but for a function $\theta\circ u_\eta$ where $\eta\in K^m$ can be chosen
arbitrarily small and $u_\eta$ is the linear automorphism of $K^{m+1}$
defined by:
\begin{equation}\label{eq:def-u-eta}
  \forall(x,t)\in K^m\times K,\quad u_\eta(x,t)=(x+t\eta,t).
\end{equation}

\begin{remark}\label{re:gd-approx}
  The smaller $\eta$ is, the closer $u_\eta$ is to the identity map since
  $\|\eta\|$ is also the norm (in the usual sense for linear maps) of
  $u_\eta-\Id$. So the functions $\theta\circ u_\eta$ can be considered as
  ``arbitrarily small deformations'' of $\theta$. 
\end{remark}

In \cite{drie-1998} a good direction for a subset $S$ of $K^{m+1}$ is
defined as a non-zero vector $x=(x_1,\dots,x_{m+1})\in K^{m+1}$ such that
every line directed by $x$ has finite intersection with $S$. 
It is more convenient to identify such collinear vectors
hence we redefine {\df good directions} for $S$ as the points
$x=[x_1,\dots,x_{m+1}]$ in the projective space $\PP^m(K)$ such that every
affine line in $K^{m+1}$ directed by $x$ has finite intersection with
$S$. 

Analogously we call $x\in\PP^m(K)$ a {\df geometrically good direction}
for a family $\cF$ of polynomials in $K[X,T]$ if for
every algebraic extension $F$ of $K$ and every $f\in\cF$, $x$ is a good
direction for the zero set of $f$ in $F^{m+1}$. 

\begin{remark}\label{re:gd-pol}
  With the above notation, $[\eta,1]$ is a good direction for $S$ if and
  only if the projection of $u_\eta^{-1}(S)$ onto $K^m$ has finite
  fibers. Indeed for every $a\in K^m$ and every $t\in K$ we have:
  \begin{displaymath}
    (a,0)+t(\eta,1)\in S\iff (a+t\eta,t)\in S \iff (a,t)\in u_\eta^{-1}(S)
  \end{displaymath}
  Therefore $[\eta,1]$ is a geometrically good direction for $\cF$ if and
  only if for every algebraic extension $F$ of $K$ and every $f\in\cF$,
  the projection onto $F^m$ of the zero set of $f\circ u_\eta$ in $F^{m+1}$
  has finite fibers.
\end{remark}

\begin{lemma}[Good Direction]\label{le:gd-open}
  For every finite family $\cF$ of non-zero polynomials in $K[X,T]$,
  the set of geometrically good directions for $\cF$ contains a
  non-empty Zariski open subset of $\PP^m(K)$. In particular, for
  every non-zero $\varepsilon\in R$ there is $\eta\in R^m$ such that $\|\eta\|\leq|\varepsilon|$ and
  $[\eta,1]$ is a good direction for $\cF$. 
\end{lemma}

\begin{proof}
Let $p_\cF$ be the product of the polynomials in $\cF$, and $d$ its
total degree. Then $p_\cF$ can be written as $p_\cF=p_\cF^\circ-q_\cF$
with $p_\cF^\circ$ a non zero homogeneous polynomial of degree $d$ and
$q_\cF$ a polynomial of total degree $<d$. 

Let $b\in K^{m+1}$ be non-zero and $x$ the corresponding point in
$P^m(K)$. It is not a geometrically good direction for $\cF$ if and
only if for some algebraic extension $F$ of $K$ and some $a\in F^m$ the
line $a+F.b$ is contained in the zero set of $p_\cF$ in $F^{m+1}$,
that is $p_\cF(a+tb)=0$ for every $t\in F$ or equivalently
$p_\cF^\circ(a+Tb)=q_\cF(a+Tb)$. This implies that the degree in $T$ of
$p_\cF^\circ(a+Tb)$ is $<d$. In particular the coefficient of $T^d$ in
$p_\cF^\circ(a+Tb)$ is zero. A straightforward computation shows that this
coefficient is just $p_\cF^\circ(b)$. 

So every element in $\PP^m(K)$ which is outside the zero set of
$p_\cF^\circ$ is a geometrically good direction for $\cF$. This proves the
main point. Now if $K^m$ is identified with its image in $\PP^m(K)$ by
the mapping $a\mapsto[a,1]$ then every ball in $K^m$ is Zariski dense in
$\PP^m(K)$, so the last claim of the lemma holds. 
\end{proof}

\begin{lemma}\label{le:gd-strong}
  Assume $\TR m$. Let $\eta\in K^m$ be such that $[\eta,1]$ is a geometrically
  good direction for $\cF$. Let $u_\eta$ be as in (\ref{eq:def-u-eta})
  and $\cF_\eta=\{f\circ u_\eta\tq f\in\cF\}$. Then every function in $\coalg(\cF_\eta)$
  whose graph is bounded is piecewise largely continuous. 
\end{lemma}

\begin{proof}
The functions in $\coalg(\cF_\eta)$ are linear combinations of functions
in $\coalg(f_\eta)$ for $f\in\cF$, hence it suffices to fix any $f$ in
$\cF$ and prove the result for $\coalg(f_\eta)$. Let $d$ be the degree in
$T$ of $f$, and $F$ a Galois extension of $K$ in which every
polynomial in $K[T]$ of degree $\leq d$ 
factors. Given a basis $\cB=(\xi_1,\dots,\xi_r)$ of $F$ over $K$, for each
integer $e\leq d$ let $a_e\in K[X]$ be the coefficient of $T^e$ in $f_\eta$,
let $A_e\subseteq K^m$ be the set of elements $x\in K^m$ such that $f_\eta(x,T)$
has degree $e$ in $T$, and choose a family of semi-algebraic functions
$b_{i,j}:A_e\to K$ such that for every $x\in A_e$
\begin{equation}\label{eq:skol-facto2}
  f_\eta(x,T)=a_e(x)\prod_{i\leq e}\Big(T-\sum_{j\leq r}b_{i,j}(x)\xi_j\Big).
\end{equation}

Let $Z_F(f_\eta)$ denote the zero set of $f_\eta$ in $F$, and $\sigma_1,\dots,\sigma_r$ be
the list of $K$\--automorphisms of $F$. Fix an integer $i\leq e$, and for
every $x\in A_e$ let
\begin{displaymath}
  \lambda_i(x)=\sum_{j\leq r}b_{i,j}(x)\xi_j.
\end{displaymath}
For every $k\leq r$ we have
\begin{displaymath}
  \sigma_k(\lambda_i(x))=\sum_{j\leq r}b_{i,j}(x)\sigma_k(\xi_j).
\end{displaymath}
Inverting the matrix $(\sigma_k(\xi_j))_{j\leq r,k\leq r}$ gives for every $j\leq r$ the
function $b_{i,j}$ as a linear combination of $\sigma_k\circ\lambda_i$ for $k\leq r$. By
construction $\Gr \sigma_k\circ\lambda_i$ is contained in $Z_F(f_\eta)$. This set is
closed, hence $\overline{\Gr \sigma_k\circ\lambda_i}$ is contained in $Z_F(f_\eta)$ too. 

The projection of $Z_F(f_\eta)$ onto $F^m$ has finite fibers since $\eta$ is
a good direction for $\cF$ (see Remark~\ref{re:gd-pol}). So the same
holds for the closure of the graph of $\sigma_k\circ\lambda_i$. This means that
each $\sigma_k\circ\lambda_i$ has finitely many different limit values at every point of
$\overline{A_e}$. Obviously each $b_{i,j}$ inherits this property,
hence so does every $h\in\coalg f_\eta$. If moreover the graph of $h$ is
bounded, it then follows from Theorem~\ref{th:piec-str-cont} (using
$\TR m$) that $h$ is piecewise largely continuous. 
\end{proof}

Now we can turn to the ``largely continuous cell preparation 
up to small deformation'' which was the aim of this section. We obtain
it by combining the above construction based on good directions and the
classical cell preparation theorem for semi-algebraic functions from
Denef (Corollary~6.5 in \cite{dene-1984}) revisited by Cluckers
(Lemma~4 in \cite{cluc-2001}). 

\begin{theorem}\label{th:large-cont-prep}
  Assume $\TR m$. Let $(\theta_i:A_i\subseteq K^{m+1}\mapsto K)_{i\in I}$ be a finite family of
  semi-algebraic functions whose domains $A_i$ are bounded. Then for
  some integer $e\geq1$ and all integers $n,N\geq1$ there exists a tuple
  $\eta\in K^m$, an integer $M_0>2v(e)$, an integer $N_0$ divisible by
  $eN$, and a finite family $\cD$ of largely continuous fitting cells
  mod $Q_{N_0,M_0}^\times$, such that $\cD$ refines $\{u_\eta^{-1}(A_i)\tq i\in
  I\}$ and such that for every $i\in I$, every $D\in\cD$ contained in
  $u_\eta^{-1}(A_i)$ and every $(x,t)\in D$
  \begin{displaymath}
     \theta_i\circ u_\eta(x,t)=\cU_{e,n}(x,t)h_{i,D}(x)
    \big[\lambda_D^{-1}\big(t-c_D(x)\big)\big]^\frac{\alpha_{i,D}}{e} 
  \end{displaymath}
  where $u_\eta$ is as in (\ref{eq:def-u-eta}), $h_{i,D}:\widehat{D}\to K$
  is a semi-algebraic function and $\alpha_{i,D}\in\ZZ$. 
  
  Moreover the set of $\eta\in K^m$ having this property is Zariski dense
  (in particular $\eta$ can be chosen arbitrarily small), and the
  integers $e$, $M$ can be chosen arbitrarily large (in the sense of
  footnote~\ref{ft:arbit-large}).
\end{theorem}

\begin{remark}\label{re:zero-infty}
  The above expression of $\theta_i\circ u_\eta$ is well defined because $e$
  divides $N_0$, $M_0>2v(e)$ and $\lambda_D^{-1}(t-c_D(x))$ belongs to
  $Q_{N_0,M_0}$ for every $(x,t)\in D$ (see the definition of $x\mapsto
  x^\frac{1}{e}$ on $Q_{N_0,M_0}$ after Lemma~\ref{le:Hensel-DP}). Of
  course if $D$ is of type $0$, then $\lambda_D=t-c_D(x)=0$ and we use our
  conventions that $0^{-1}=\infty$ and $\infty.0=1$.
\end{remark}

If we were only interested in the existence of such a preparation
theorem with largely continuous cells for $\theta_i\circ u_\eta$, the integer $N$
would be of no use and could be taken equal to $1$. However it will be
convenient to allow different values of $N$ when
we will use Theorem~\ref{th:large-cont-prep} in the proof the
Triangulation Theorem.

\begin{proof}
Let $e_*,M_*\geq1$ be arbitrary integers.
Corollary~6.5 in \cite{dene-1984} applied to each $\theta_i$ gives an
integer $e_i\geq 1$ and a family $\cA_i$ of semi-algebraic sets
partitioning $A_i$ such that for every every $A$ in $\cA_i$ and every
$(x,t)$ in $A$:
\begin{equation}\label{eq:theta-e-u}
  \theta_i^{e_i}(x,t)=u_{i,A}(x,t)\frac{f_{i,A}(x,t)}{g_{i,A}(x,t)}
\end{equation}
where $u_{i,A}$ is a semi-algebraic function from $A$ to $R^\times$ and
$f_{i,A}$, $g_{i,A}$ are polynomial functions such that
$g_{i,A}(x,t)\neq0$ on $A$. Replacing if necessary each $e_i$ by a common
multiple $e$ of them and of $e_*$, we can assume that $e_i=e$ for
every $i$ and $e$ is divisible by $e_*$. Let $\cA$ be a 
refinement of $\bigcup_{i\in I}\cA_i$.

Fix any two integers $n,N\geq1$ and any integer $n_0$ such
that $n_0\geq n+v(e)$ and $n_0>2v(e)$. Since $D^{n_0}R^\times$ is a subgroup
of $R^\times$ with finite index, every $A\in\cA$ splits into finitely many
semi-algebraic pieces on each of which $u_{i,A}$ is constant modulo
$D^{n_0}R^\times$ (for every $i\in I$ such that $A\subseteq A_i$). Thus, refining
$\cA$ if necessary, (\ref{eq:theta-e-u}) can be replaced, for every
$A$ in $\cA$ contained in $A_i$ and every $(x,t)$ in $A$, by
\begin{equation}
  \theta_i^{e}(x,t)=\cU_{n_0}(x,t)\tilde u_{i,A}\frac{f_{i,A}(x,t)}{g_{i,A}(x,t)}
  \label{eq:theta-e-U}
\end{equation}
with $\tilde u_{i,A}\in R^\times$. 

Each $A$ in $\cA$ is semi-algebraic. So there is
a finite family $\cB$ of semi-algebraic sets refining $\cA$, an
integer $N_0\geq1$ and a finite list $\cF$ of non-zero polynomials in
$m+1$ variables such that every element of $\cB$ is a boolean
combinations of sets $f^{-1}(\PN_{N_0})$ with $f\in\cF$. By
Remark~\ref{re:N-mac}, $N_0$ can be chosen divisible by $eN$.
Expanding $\cF$ if necessary, we can assume that all the polynomials
$f_{i,A}$ and $g_{i,A}$ in (\ref{eq:theta-e-U}) also belong to $\cF$,
except those which are equal to the zero polynomial. 

Lemma~\ref{le:gd-open} gives $\eta\in K^{m+1}$ such that $[\eta,1]$ is a
geometrically good direction for $\cF_\eta$, where $\cF_\eta=\{f\circ u_\eta\tq
f\in\cF\}$. Note that every set in $\cA_\eta=\{u_\eta^{-1}(A)\tq A\in\cA\}$ is a
boolean combination of sets $\smash{f_\eta^{-1}(\PN_{N_0})}$ with $f_\eta\in\cF_\eta$.
Denef's Theorem~\ref{th:D2} applied to $\cF_\eta$ gives a finite family
$\cC$ of fitting cells mod $\PN_{N_0}^\times$ which refines $\cA_\eta$ and
whose center and bounds belong to $\coalg\cF_\eta$, such that for every
$f\in\cF$, every $C\in\cC$ and every $(x,t)\in C$
\begin{equation}
  f_\eta(x,t)=\cU_{n_0}(x,t)h_{f,C}(x)\big(t-c_C(x)\big)^{\alpha_{f,C}}
  \label{eq:f-eta}
\end{equation}
where $h_{f,C}:\widehat{C}\to K$ is a semi-algebraic function and
$\alpha_{f,C}$ is a positive integer. We removed the zero polynomial from
$\cF$, but obviously (\ref{eq:f-eta}) holds for $f=0$ as well, by
taking $h_{f,C}=0$ in that case. Each $A_i$ is bounded hence so is their
union $\bigcup\cA$ as well as $\bigcup\cA_\eta$. So the center and bounds of every
cell in $\cC$ must be bounded functions with bounded domain. By
Lemma~\ref{le:gd-strong} (assuming $\TR m$) these functions are
piecewise largely continuous. Refining the socle of $\cC$ if
necessary, and $\cC$ accordingly, we can then reduce to the case where
every cell in $\cC$ is largely continuous. Note that $\cU_{n_0}\circ
u_\eta=\cU_{n_0}$, so by combining (\ref{eq:theta-e-U}) and
(\ref{eq:f-eta}) we get that for every $i\in I$, every $C\in\cC$ contained
in $u_{\eta}^{-1}(A_i)$ and every $(x,t)\in C$
\begin{equation}
  \theta_{i,\eta}(x,t)^{e}=\cU_{n_0}(x,t)h_{i,C}(x)\big(t-c_C(x)\big)^{\alpha_{i,C}}
  \label{eq:theta-C}
\end{equation}
where $\theta_{i,\eta}=\theta_i\circ u_\eta$, $h_{i,C}:\widehat{C}\to K$ is a
semi-algebraic function and $\alpha_{i,C}\in\ZZ$. For any
integer $M_0>2v(e)$, $Q_{N_0,M_0}^\times$ is a subgroup with finite index in
$\PN_{N_0}^\times$ hence every such cell $C$ mod $\PN_{N_0}^\times$ splits into
finitely many cells $D$ mod $Q_{N_0,M_0}^\times$ with the same center, bounds
and type as $C$. The integer $M_0$ can be chosen arbitrarily large,
in particular greater than $M_*$. Let $\cD$ be the family of all these
cells $D$. From (\ref{eq:theta-C}) and Lemma~\ref{le:Hensel-DP} we
derive that for every $i\in I$, every $D\in\cD$ contained in
$u_{\eta}^{-1}(A_i)$ and every $(x,t)\in D$
\begin{equation}
  \theta_{i,\eta}(x,t)^{e}=\cU_{n_0}(x,t)\tilde h_{i,D}(x)
  \Big(\big[\lambda_D^{-1}\big(t-c_D(x)\big)\big]^\frac{\alpha_{i,D}}{e}\Big)^{e}
  \label{eq:theta-D}
\end{equation}
where $\tilde h_{i,D}=h_{i,C}$ and $\alpha_{i,D}=\alpha_{i,C}$ with $C$ the unique
cell in $\cC$ containing $D$. The factor $\cU_{n_0}$ in
(\ref{eq:theta-D}) can be written $\cU_{n_0-v(e)}^{e}$ by
Remark~\ref{re:racine-de-U}. Thus (\ref{eq:theta-D}) implies that
$\smash{\tilde h_{i,D}}$ takes values in $\PN_e$. So by Theorem~\ref{th:skol}
there is a semi-algebraic function $h_{i,D}$ such that
$\tilde h_{i,D}=h_{i,D}^{e}$. As a consequence, from (\ref{eq:theta-D}) it
follows that there is a semi-algebraic function $\chi_{i,D}$ with values
in $\UU_e$ such that for every $(x,t)\in D$
\begin{equation}
  \theta_{i,\eta}(x,t)=\chi_{i,D}(x,t)\cU_{n_0-v(e)}(x,t)h_{i,D}(x)
  \big[\lambda_D^{-1}\big(t-c_D(x)\big)\big]^\frac{\alpha_{i,D}}{e}
  \label{eq:theta-chi}
\end{equation}
By construction $n_0-v(e)\geq n$ hence the factor $\cU_{n_0-v(e)}$
can {\it a fortiori}\, be replaced by $\cU_n$. Then $\chi_{i,D}\cU_n$
(which is just $\cU_{e,n}$) replaces $\chi_{i,D}\cU_{n_0-v(e)}$ in
(\ref{eq:theta-chi}), which proves the result. 
\end{proof}

\section{Cellular complexes}
\label{se:cont-comp}

For this and the next section, let $\GG$ be a fixed semi-algebraic
clopen subgroup of $K^\times$ with finite index. Then $v\GG$ is a subgroup
of $\cZ$ with finite index, hence $v\GG=N_0\cZ$ for some integer
$N_0\geq1$. Our aim in these two sections is to prove that every finite
family of bounded largely continuous fitting cells mod $\GG$, such as
the one given by Theorem~\ref{th:large-cont-prep}, can be refined in a
complex of cells mod $\GG$ satisfying certain restrictive assumptions
defined below. 

\paragraph{Notation.} 
For every largely continuous fitting cell $A$ mod $\GG$ in $K^{m+1}$
with socle $X$, recall that $A=(c_A,\mu_A,\nu_A,G_A)$ is a presented cell.
For every semi-algebraic set $Y$ contained in $\overline{X}$, $(\bar
c_A,\bar \nu_A,\bar \mu_A,G_A)_{|Y}$ is then also a largely continuous
presented cell mod $\GG$, provided the restrictions to $Y$ of $\bar\nu_A$ and
$\bar\mu_A$ either take values in $K^\times$ or are constant, and the
underlying set of tuples $(y,t)\in Y\times K$ defined by
\begin{equation}
  |\bar\mu_A(y)|\leq|t-\bar c_A(x)|\leq|\bar\nu_A(x)|\mbox{ \ and \ }
  t-\bar c_A(x)\in G_A
  \label{eq:partial-A}
\end{equation}
is non-empty. Similarly, the sets $\partial^0_YA$ and $\partial^1_YA$ defined below
are (if non-empty) largely continuous fitting cells mod $\GG$
contained in $\overline{A}\cap(Y\times K)$.
\begin{itemize}
  \item
    $\partial_Y^0 A=(\bar c_A,0,0,\{0\})_{|Y}$ if $\bar\nu_A=0$ on $Y$, $\partial_Y^0A=\emptyset$
    otherwise;
  \item 
    $\partial_Y^1 A=(\bar c_A,\bar \nu_A,\bar \mu_A,G_A)_{|Y}$ if $\bar\mu_A\neq0$ on
    $Y$, $\partial_Y^1A=\emptyset$ otherwise.
\end{itemize}
If non empty the underlying set of $\partial_Y^0 A$ is the graph of the
restriction of $\bar c_A$ to $Y$, while the underlying set of $\partial_Y^A$
is the set of $(x,t)\in Y\times K$ satisfying (\ref{eq:partial-A}).
For example, when $\bar\nu_A=0\neq\bar\mu_A$ on $Y$ and $\nu_A\neq0$ on $X$, we can
intuitively represent these sets as follows. 

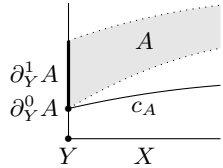
\begin{center}
  \begin{tikzpicture}
    \small
    \def\fonctioncA{plot[domain=0:2] (\x,{.3+\x*(6-\x)/26})}
    \newcommand{\fonctioncB}[2]{plot[domain=0:2] (\x,{#1+\x*(6-\x)/#2})}
    \newcommand{\cellule}[3]{
        \fill[color=gray!#1] 
          plot[domain=#2:#3] (\x,{.3+\x*(6-\x)/10}) --
          plot[domain=#3:#2] (\x,{1.2+\x*(6-\x)/17}) -- cycle;
      }

    \cellule{20}{0}{2};

    \draw[thin] (0,-.1) node[below]{$Y$} -- node[below]{$X$} (2,-.1);
    \draw[thin] (0,-.1) -- (0,1.7);
    \draw (0,-.1) node{\tiny$\bullet$};

    \draw \fonctioncA;
    \draw (1,.3) node{$c_A$};

    \draw[dotted] \fonctioncB{.3}{10};
    \draw[dotted] \fonctioncB{1.2}{17};

    \draw 
      (1,1.2) node{$A$};

    \draw[very thick] 
      (0,.3) node{\tiny$\bullet$} node[left]{$\partial_Y^0A$} -- node[left]{$\partial_Y^1A$} (0,1.2);

  \end{tikzpicture}
\end{center}

Provided that on $Y$, $\bar\nu_A$ and $\bar\mu_A$ either take values in
$K^\times$ or are constant, $\partial^0_YA$ and $\partial^1_YA$ are (if non-empty) largely
continuous fitting cells mod $\GG$ contained in $\overline{A}\cap(Y\times K)$.

\begin{remark}\label{re:dd-partition}
  If $\cX$ is a partition of $\overline{X}$, the family of non-empty
  $\partial_Y^iA$ for $i\in\{0,1\}$ and $Y\in\cX$ form a partition of
  $\overline{A}$. 
\end{remark}

Given two cells $A$, $B$ in $K^{m+1}$ and an
integer $n\geq1$, we write $B\lhd^n A$ if $B\subseteq A$ and if there exists
$\alpha\in\{0,1\}$ and a semi-algebraic function $h_{B,A}:\widehat{B}\to K$ such
that for every $(x,t)$ in $B$:
\begin{displaymath}
  t-c_A(x)=\cU_n(x,t)h_{B,A}(x)^\alpha\big(t-c_B(x)\big)^{1-\alpha}  
\end{displaymath}
We call $h_{B,A}$ a {\df $\lhd^n$\--transition} for $(B,A)$.
If $\cA$, $\cB$ are families of cells in $K^{m+1}$ we write
$\cB\lhd^n\cA$ if $B\lhd^nA$ for every $B\in\cB$ and $A\in\cA$ such that $B$
meets $A$. A {\df $\lhd^n$\--system} for $(\cB,\cA)$ is
then the data of one $\lhd^n$\--transition for each possible
$(B,A)$ in $\cB\times\cA$.

\begin{remark}\label{re:alg-n-comp}
  For any two finite families $\cA$, $\cB$ of cells mod $\GG$,
  if $\cB$ refines $\cA$ and belongs to $\Alg_n\cA$ then $\cB\lhd^n\cA$. 
\end{remark}

A {\df closed $\lhd^n$\--complex} of cells mod $\GG$ is a finite family $\cA$ of
largely continuous fitting cells mod $\GG$ such that $\bigcup\cA$ is
closed, the socle of $\cA$ is a complex of sets and for every
$A,B\in\cA$ if $B$ meets $\overline{A}$ then for some
$i\in\{0,1\}$, $\partial_Y^iA$ is a cell\footnote{The condition $\partial^iYA$ is a cell
  means that on $Y$, $\bar\mu_A$ and $\bar\nu_A$ either take values in $K^\times$
  or are constant.\label{ft:derAi-cell}} and $B\lhd^n\partial^i_YA$, with $Y=\widehat{B}$.
If moreover $B=\partial^i_YA$ we call $\cA$ a {\df closed cellular complex}
mod $\GG$. As the terminology suggests, we are going to prove that
closed $\lhd^n$- and cellular complexes are complexes of sets in the
general sense of Section~\ref{se:notation} (see
Proposition~\ref{pr:cell-comp}). Any subset of a closed
$\lhd^n$\--complex (resp. closed cellular complex) is a {\df
$\lhd^n$\--complex} (resp. a {\df cellular complex}. As usually
we call them {\df monoplexes} if they form a tree with respect to the
specialization order.

When $\cA$ is a $\lhd^n$\--complex of cells mod $\GG$, for all $Y\in\widehat{\cA}$
and for all cells $A$, $B$ in $\cA$ such that $B$ meets $\overline{A}$,
there is an integer $\alpha\in\{0,1\}$ and a semi-algebraic function
$h_{B,A}:\widehat{B}\to K$ such that for every $(x,t)$ in $B$:
\begin{displaymath}
  t-\bar c_A(x)=\cU_n(x,t)h_{B,A}(x)^\alpha\big(t-c_B(x)\big)^{1-\alpha}.  
\end{displaymath}
An {\df inner $\lhd^n$\--system} for $\cA$ is the data of one
function $h_{B,A}$ as above for every possible $A,B\in\cA$.

\begin{proposition}\label{pr:cell-comp}
  Let $\cA$ be a closed $\lhd^n$\--complex of cells mod $\GG$. Then $\cA$ is a
  closed complex of sets. Moreover, for every $A,B\in\cA$ and every
  $Z\in\widehat{\cA}$ if $B$ meets $\partial_Z^0A$ then $B=\partial_Z^0A=\Gr\bar
  c_{A|Z}$.
\end{proposition}

\begin{proof}
  By assumption the socle of every cell $A$ in $\cA$ is relatively
  open and pure dimensional. Thanks to the restrictions we made on the
  bounds in our definition of presented cells, it follows that $A$ is
  also relatively open and pure dimensional. 

  In order to show that $\cA$ is a partition, let $A$, $B$ be two cells
  in $\cA$ which are not disjoint and let $X=\widehat{A}$. Both
  $\widehat{B}$ and $X$ belong to $\widehat{\cA}$ and are not
  disjoint, hence $\widehat{B}=X$. Since $B$ meets $A\subseteq\overline{A}$,
  by assumption $B$ is contained in $\partial^i_XA$ with $i=\Tp B$. But then
  $\partial^i_XA$ meets $A$, hence obviously is equal to $A$. So $B\subseteq A$, and
  equality holds by symmetry. 

  Now let $A$ be any cell in $\cA$ and $X=\widehat{A}$. Since $\cA$
  is a closed complex, every point of $\overline{A}$ belongs to a unique
  $B\in\cA$. Since $B$ meets $\overline{A}$, by assumption $B\subseteq\partial^i_YA$
  with $Y=\widehat{B}$ and $i=\Tp B$. In particular $B\subseteq\overline{A}$,
  which proves that $\overline{A}$ is a union of cells in $\cA$ (hence
  so is $\partial A$ since $\cA$ is a partition and $\partial A$ is disjoint from
  $A$). This proves that $\cA$ is a closed complex of sets.

  The last point follows. Indeed, if $B$ meets $\partial_Z^0A\subseteq\overline{A}$
  then it is contained in $\partial_Y^iA$ for some $i\in\{0,1\}$, with
  $Y=\widehat{B}$. In particular $\partial_Z^0A$ meets $\partial_Y^iA$. They are two
  pieces of a partition of $\overline{A}$ (see
  Remark~\ref{re:dd-partition}) hence $\partial_Y^0A=\partial_Z^iA$. Therefore $Y=Z$
  and $i=0$, so $B\subseteq\partial_Z^0A$. That is, $B$ is of type $0$ and $c_B=\bar
  c_A$ on $\widehat{B}=Z$, so $B=\Gr\bar c_{A|Z}=\partial_Z^0A$.
\end{proof}

\begin{proposition}\label{pr:n-complex}
  Let $\cA$ be a finite family of largely continuous fitting cells mod
  $\GG$ and $n\geq1$ an integer. There exists a $\lhd^n$--complex $\cD$ of
  cells mod $\GG$ refining $\cA$ such that $\cD\lhd^n\cA$. 
\end{proposition}

In the next section we will prove that one can even require that $\cD$
is a cellular monoplex mod $G$.

\begin{proof}
The proof goes by induction on $d=\dim\bigcup\widehat{\cA}$. If a
$\lhd^n$\--complex $\cD_1$ is found which proves the result for a family
$\cA_1$ of cells mod $\GG$ containing $\cA_1$ then obviously the
family $\cD$ of cells in $\cD_1$ contained in $\bigcup\cA$ proves the
result for $\cA$. Thus, enlarging $\cA$ if necessary, we can assume
that $\bigcup\cA$ and $\bigcup\widehat{\cA}$ are closed. By Denef's
Theorem~\ref{th:D1} and Remark~\ref{re:alg-n-comp} there is a finite
family $\cB$ of largely continuous fitting cells mod $\GG$ refining
$\cA$ such that $\cB\lhd^n\cA$. Replacing $\cA$ by this refinement if
necessary we can also assume that $\cA$ is a partition. 

If $\cB$ is any vertical refinement of $\cA$ then obviously
$\cB\lhd^n\cA$. Thus, by taking if necessary a finite partition $\cX$
refining $\widehat{\cA}$ and replacing $\cA$ by the corresponding
vertical refinement (that is the family of all cells $A\cap(X\times K)$
with $A\in\cA$ and $X\in\cX$ contained in $\widehat{A}$), we can assume
that $\widehat{\cA}=\cX$ is a partition. By the same argument we can
assume as well that for every $A\in\cA$ and every $X\in\cX$ contained in $\partial X$,
the restrictions of $\bar\mu_A$ and $\bar\nu_A$ to $X$ take values in
$K^\times$ or are constant, hence $\partial_X^0A$ (resp. $\partial_X^1A$) is a cell
with socle $X$ whenever it is non-empty. By Remark~\ref{re:comp-pure}
we can even assume that it is a complex of pure dimensional sets. Let
$\cX_d$ be the family of $X\in\cX$ with dimension $d$. Note that every
$X\in\cX_d$ is open in $\bigcup\cX$ because $\cX$ is a complex and
$\dim\bigcup\cX=d$.

For every $X\in\cX_d$ let $\cA_X$ be the family of cells in $\cA$ with
socle $X$. For every cell $A\in\cA_X$ of type $1$ such that $\nu_A=0$,
$\Gr c_A$ is contained in $\overline{A}$ hence in $\bigcup\cA_X$ since
$\bigcup\cA$ is closed and $\widehat{\cA}$ is a partition. It may happen
that $\Gr c_A$ does not belong to $\cA$. With
Proposition~\ref{pr:cell-comp} in view we have to remedy this.
Every point $(x,c_A(x))$ in $\Gr c_A$ belongs to some cell $B$ in
$\cA_X$. This cell must be of type $0$ otherwise the fiber $B_x=\{t\in
K\tq (x,t)\in B\}$ would be open, hence it would contain a neighbourhood
$V$ of $c_A(x)$ and so $\{x\}\times V$ would be contained in $B$ and meet
$A$, which implies that $B\subseteq A$ since $\cA_X$ is partition, in
contradiction with the fact that $B$ meets $\Gr c_A$. So there is a finite
partition $\cY_A$ of $X$ in semi-algebraic pieces $Y$ on each of which
there is a unique cell $B\in\cA_X$ of type $0$ whose center coincides
with $c_A$ on $Y$. Repeating the same argument for every $A\in\cA_X$ and
every $X\in\cX_d$ gives a finite partition $\cY$ of $\bigcup\cX_d$ finer then
every such $\cY_A$. Let $\cX'$ be a complex of pure dimensional
semi-algebraic sets refining $\cX\cup\cY$. Replacing if necessary $\cA$
by the vertical refinement defined by $\cX'$, we can then assume from
now on that for every $X$ in $\cX_d$ and every $A\in\cA$ with socle $X$,
if $\nu_A=0$ then $\Gr c_A$ belongs to $\cA$. 

Let $\cA_d=\{A\in\cA\tq\widehat{A}\in\cX_d\}$ and $\cB$ be the union of
$\cA\setminus\cA_d$ and of the family of non-empty $\partial^i_YA$ for $i\in\{0,1\}$,
$A\in\cA_d$ and $Y\in\cX\setminus\cX_d$. Clearly $\dim\bigcup\widehat{B}<d$ so the induction
hypothesis gives a $\lhd^n$--complex $\cC$ of cells mod $\GG$ refining
$\cB$ such that $\cC\lhd^n\cB$. {\it A fortiori} $\cC\lhd^n(\cA\setminus\cA_d)$
because the latter is contained in $\cB$. So if we let $\cD=\cC\cup\cA_d$,
then $\cD$ refines $\cA$ and $\cD\lhd^n\cA$. It only remains to check that $\cD$
is a $\lhd^n$\--complex, and first that $\widehat{\cD}$ is a complex of
sets.

Note that $\widehat{\cA}_d=\cX_d$ hence 
$\widehat{\cD}=\widehat{\cC}\cup\widehat{\cA}_d=\widehat{\cC}\cup\cX_d$ 
is a partition,
and every set in $\widehat{\cD}$ is pure dimensional and relatively open (by
induction hypothesis for $\widehat{\cC}$ and by construction for
$\cX_d$). For every $X\in\widehat{\cD}$, we have to prove that $\partial X$ is
a union of sets in $\widehat{\cC}\cup\cX_d$. If $X\in\widehat{\cC}$ this is
clear because $\widehat{\cC}$ is a complex. Otherwise $X\in\cX_d$ hence
$\partial X$ is a union of sets in $\cX$ (because $\cX$ is a complex).
All these sets have dimension $<d=\dim X$ hence belong to
$\cX\setminus\cX_d$. But $\cC$ refines $\cB$, which contains
$\cA\setminus\cA_d$, whose socle is $\cX\setminus\cX_d$, hence $\widehat{\cC}$ refines
$\cX\setminus\cX_d$. Thus $\partial X$ is also the union of sets in $\widehat{\cC}$,
hence of $\widehat{\cD}$.

Now let $D,E\in\cD$ be such that $E$ meets $\overline{D}$, let
$X=\widehat{D}$ and $Y=\widehat{E}$. By construction $\partial_Y^0D$ and
$\partial_Y^1D$ are cells (if non-empty) and cover $\overline{D}\cap(Z\times K)$. So
there is $i\in\{0,1\}$ such that $\partial^i_Y D$ is a cell which meets $E$. We
have to prove that $E\lhd^n\partial_X^iD$. Note that $Y$ meets the socle of
$\overline{D}$, which is contained in $\overline{X}$, hence $Y=X$ or
$Y\subseteq\partial X$ because $\widehat{\cD}$ is a complex. So, if $\dim X<d$ then
also $\dim Y<d$ hence $D,E\in\cC$. In that case $E\lhd^n\partial_X^jD$ because
$\cC$ is a $\lhd^n$--complex. Thus we can assume that $\dim X=d$, that is
$D\in\cA_d$. We know that $Y=X$ or $Y\subseteq\partial X$. In the first case $Y=X$
hence $\partial_X^i D\in\cA_d\subseteq\cD$ by construction, so $E=\partial_X^i D$ because
$\cD$ is a partition. In the second case $Y\in\widehat{\cC}$ hence
$E\in\cC$. Now $Y$ is contained in some $Z\in\cX\setminus\cX_d$ because
$\widehat{\cC}$ refines $\cX\setminus\cX_d$, and $E$ meets $\partial_Z^iD$. By
construction $\partial_Z^iD$ belongs to $\cB$. Since $\cC\lhd^n\cB$ it follows
that $E\lhd^n\partial_Z^iD$ hence {\it a fortiori} $E\lhd^n\partial_Y^iD$ because $E\subseteq Y\times
K$ and $\partial_Y^iD=\partial_Z^iD\cap(Y\times K)$. 
\end{proof}

Before entering in more complicated constructions, let us mention here two
elementary properties of fitting cells which will be of some use later. 

\begin{proposition}\label{pr:fit-cell}
  Let $A\subseteq K^{m+1}$ be a cell mod $\GG$ of type $1$. Then:
  \begin{itemize}
    \item
      $\mu_A$ is a fitting bound if and only if  $\mu_A=\infty$ or
      $v\mu_A(\widehat{A})\subseteq vG_A$. 
    \item 
      $\nu_A$ is a fitting bound if and only if $\nu_A=0$ or
      $v\nu_A(\widehat{A})\subseteq vG_A$).
  \end{itemize}
\end{proposition}

\begin{proof}
  The case where $\mu_A=\infty$ being trivial, we can omit it. If $\mu_A\neq\infty$ is a
  fitting bound then obviously $v\mu_A(\widehat{A})\subseteq vG_A$ because
  $v(t-c_A(x))\in vG_A$ for every $(x,t)\in A$. Conversely assume that
  $v\mu_A(\widehat{A})\subseteq vG_A$. Let $x$ be any element of
  $\widehat{A}$. We have to prove that $|\mu_A(x)|=\max\{|d|\tq d\in D_x\}$
  where $D_x=\{t-c_A(x)\tq (x,t)\in A\}$. $D_x$ is bounded since $\mu_A\neq\infty$,
  hence by Corollary~\ref{co:max} it contains an element $d$ of
  maximal norm. By construction $|d|\leq|\mu_A(x)|$. Assume for a
  contradiction that $|d|<|\mu_A(x)|$, that is $v(d/\mu_A(x))>0$. By
  construction $v(d)$ and $v\mu_A(x)$ belong to $vG_A=v\lambda_A+v\GG$ hence
  $v(d/\mu_A(x))\in v\GG=N_0\cZ$. Thus $v(d/\mu_A(x))\geq N_0$, that is
  $|d|\leq|\pi^{N_0}\mu_A(x)|$. Pick any $g\in\GG$ such that $v(g)=N_0$ and
  let $t'=c_A(x)+d/g$. We have $t'-c_A(x)=d/g\in G_A$,
  $|\nu_A(x)|\leq|d|\leq|d/g|$ and $|d/g|\leq|\mu_A(x)|$, hence $(x,t')\in A$. So
  $t'-c_A(x)\in D_x$ and $|d|<|t'-c_A(x)|$, a contradiction. 
  The proof for $\nu_A$ is similar and left to the reader. 
\end{proof}

\begin{proposition}\label{pr:mu-bounded}
  For every fitting cell $A$ mod $Q_{N_0,M_0}$ in $K^{m+1}$, if $A\subseteq
  R^{m+1}$ then $v\mu_A\geq-M_0$.
\end{proposition}

Since $A\subseteq R^{m+1}$, one may naively expect that $|\mu_A|\leq1$, that is
$v\mu_A\geq0$. The presented cell
$A=(-\pi^{-M_0},\pi^{-M_0},\pi^{-M_0},Q_{N_0,M_0})$ is a counterexample in
$K$: it is contained in $R$ (it is actually equal to $R$) and
$v\mu_A=-M_0<0$. 

\begin{proof}
  Assume the contrary, that is $v\mu_A(x)<-M_0$ for some
  $x\in\widehat{A}$. Since $A$ is a fitting cell there is $t\in K$ such
  that $(x,t)\in A$ and $v(t-c_A(x))=v\mu_A(x)$. Since $A\subseteq R^{m+1}$, $t\in R$ hence
  $v(t-c_A(x))<0=v(t)$ implies that $vc_A(x)=v(t-c_A(x))=v\mu_A(x)$. So
  there are $a\in R$ and $g\in Q_{N,M}$ such that $c_A(x)=a\pi^{M_0+1}$ and
  $t-c_A(x)=\lambda_A g$. In particular $v(\lambda_Ag)=v(t-c_A(x))=v\mu_A(x)<-M_0$
  so $\pi^{M_0}\lambda_Ag\notin R$. Now let $t'=t+\pi^{M_0}\lambda_Ag$, then $t'\notin R$ since $t\in R$
  and $\pi^{M_0}\lambda_Ag\notin R$. On the other hand $1+\pi^{M_0}\in Q_{N_0,M_0}$ and
  \begin{displaymath}
    t'-c_A(x) = t-c_A(x)+\pi^{M_0}\lambda_Ag = \lambda_Ag + +\pi^{M_0}\lambda_Ag = 
    \lambda_A(1+\pi^{M_0})g.
  \end{displaymath}
  So $t'-c_A(x)\in \lambda_A Q_{N_0,M_0}$ and
  $v(t'-c_A(x))=v(\lambda_A(1+\pi^{M_0})g)=v(\lambda_Ag)=v\mu_A(x)$. Thus $(x,t')\in A$,
  a contradiction since $t'\notin R$ and $A\subseteq R^{m+1}$.
\end{proof}

\section{Cellular monoplexes}
\label{se:cont-mono}

We keep as in Section~\ref{se:cont-comp} a semi-algebraic clopen
subgroup $\GG$ of $K^\times$ with finite index, and $N_0\geq1$ an integer such
that $v\GG=N_0\cZ$. Lemma~\ref{le:pretriang} below (together with
Lemma~\ref{le:lift-triang}) is the technical heart of this paper. 
This section is entirely devoted to its proof. 

\begin{lemma}\label{le:pretriang}
  Assume $\TR m$. Let $\cA$ be a finite set of bounded, largely
  continuous, fitting cells mod $\GG$ in $K^{m+1}$. Let $\cF_0$ be a
  finite family of definable functions with domains in
  $\widehat{\cA}$. Let $n,N\geq1$ be a pair of integers. For some
  integers $e$, $M>2v(e)$ which can be made arbitrarily large (in the
  sense of footnote~\ref{ft:arbit-large}), there is a tuple
  $(\cV,\varphi,\cD,\cF_\cD)$ such that: 
  \begin{itemize}
    \item
      $\cD$ is a cellular monoplex mod $\GG$ refining $\cA$ such that
      $\cD\lhd^n\cA$. 
    \item
      $\cF_\cD$ is a  $\lhd^n$\--system for $(\cD,\cA)$.
    \item
      $(\cV,\varphi)$ is a triangulation of\/\footnote{Recall that $\CB(\cD)$
        denotes the family of center and bounds of the cells
      in $\cD$.} $\cF_0\cup\cF_\cD\cup\CB(\cD)$ with parameters $(n,N,e,M)$,
      such that $\widehat{\cD}=\varphi(\cV)$. 
  \end{itemize}
\end{lemma}

Note that, in order to obtain this result, it does not suffice to find
a continuous monoplex $\cD$ of well presented cells mod $\GG$ refining
$\cA$ such that $\cD\lhd^n\cA$, and then to select an arbitrary 
$\lhd^n$\--system $\cF_\cD$ for $(\cD,\cA)$ and to apply $\TR m$ to
$\cF_0\cup\cF_\cD\cup\CB(\cD)$. Indeed, this will give a triangulation
$(\cV,\varphi)$ of $\cF_0\cup\cF_\cD\cup\CB(\cD)$. But $\varphi(\cV)$ will then be a
refinement of $\widehat{\cD}$, not $\widehat{\cD}$ itself. It is then
tempting to vertically refine $\cD$, that is to replace $\cD$ by the
family $\cE$ of cells $D\cap(\varphi(V)\times K)$ for $D\in\cD$ and $V\in\cV$ such that
$\varphi(V)\subseteq\widehat{D}$. This ensures that $\widehat{\cE}=\varphi(\cV)$ and $\cE$ is a
cellular complex such that $\cE\lhd^n\cA$. But $\cE$ is no
longer a monoplex.

In order to break this vicious circle we have to build $\cV$, $\cD$
and $\cF_\cD$ simultaneously. The remainder of this section is
devoted to this construction. It is divided in three parts:
(\ref{sub:prepar}) preparation,
(\ref{sub:vert-refin}) vertical refinement,
(\ref{sub:hori-refin}) horizontal refinement.

\subsection{Preparation}
\label{sub:prepar}

Given a family $\cX$ of subsets of $K^m$, we let $\cF_{0|\cX}$ denote
the family of all the restrictions $f_{|X}$ with $f\in\cF_0$ and $X\in\cX$
contained in the domain of $f$. By the same argument as in the
beginning of the proof of Proposition~\ref{pr:n-complex}, we can
assume that $\bigcup\cA$ is closed. Finally, replacing if necessary $\cA$
by a refinement $\cD$ given by Proposition~\ref{pr:n-complex} and
$\cF_0$ by $\cF_{0|\cX}$ with $\cX=\widehat{\cD}$, we are reduced to
the case where $\cA$ is a closed $\lhd^n$--complex of bounded cells mod $\GG$.
Enlarging $\cF_0$ if necessary, we can, and will, assume that it
contains $\CB(\cA)$ and an inner $\lhd^n$\--system for $\cA$. For some
integers $e$, $M>2v(e)$ which can be made arbitrarily large, $\TR m$
gives a triangulation $(\cS,\varphi)$ of $\cF_0$ with parameters
$(n,N,e,M)$. For every $A\in\cA$ we let $S_A=\varphi^{-1}(\widehat{A})$. 

Since $\bigcup\cA$ is bounded and closed in $K^{m+1}$, its image
$\bigcup\widehat{\cA}$ by the coordinate projection is closed in $K^m$ by
Theorem~\ref{th:extr-val}. Now $\varphi$ is a homeomorphism from $\biguplus\cS$ to
$\bigcup\widehat{\cA}$, hence $\cS$ is closed by
Remark~\ref{re:closed-triang}. 

Let $\cA'$ be the family of cells $A\cap(\varphi(S)\times K)$ for $A\in\cA$ and
$S\in\cS$ such that $\varphi(S)\subseteq\widehat{A}$, and let $\cF_0'=\cF_{0|\varphi(\cS)}$.
Since every cell in $\cA'$ has the same center and bounds as the
unique cell in $\cA$ which contains it, clearly $\cA'$ is still a
closed $\lhd^n$--complex, $\cF'_0$ contains $\CB(\cA')$ and an inner
$\lhd^n$\--system for $\cA'$, and $(\cS,\varphi)$ is still a triangulation of
$\cF'_0$. Thus, replacing $(\cA,\cF_0)$ by $(\cA',\cF'_0)$ if
necessary, we can assume that $\varphi(\cS)=\widehat{\cA}$, that is
$S_A\in\cS$ for every $A\in\cA$. 

A {\df preparation} for $(\cS,\varphi,\cA,\cF_0)$ is a tuple
$(\cT,\cB,\cF_\cB)$ such that:
\begin{description}
  \item[(P1)] $\cT$ is a simplicial subcomplex of $\cS$. We let
    $\cS_{|\cT}=\{S\in\cS\tq S\subseteq\biguplus\cT\}$, and $\cA_{|\cT}$ be the family of
    cells $A\in\cA$ such that $S_A\in\cS_{|\cT}$. Note that:
    \begin{itemize}
      \item
        $\cT$ is closed because $\biguplus\cT$ is closed in $\biguplus\cS$.
      \item 
        By Remark~\ref{re:closed-triang} it follows that the image by
        $\varphi$ of $\bigcup\cT$, that is the socle of $\cA_{|\cT}$, is closed
        too. 
      \item 
        Hence $\cA_{|\cT}$ is closed because $\bigcup\cA_{|\cT}$ is the
        inverse image of its socle by the (continuous) coordinate
        projection of $\bigcup\cA$ onto $\bigcup\widehat{\cA}$. 
    \end{itemize}
  \item[(P2)]
    $\cB$ is a cellular monoplex 
    mod $\GG$ refining $\cA_{|\cT}$ such that $\varphi(\cT)=\widehat{\cB}$.
    For every $B\in\cB$ we let $T_B=\varphi^{-1}(\widehat{B})$. Note that
    $\cB$ is closed because $\bigcup\cB=\bigcup\cA_{|\cT}$. 
  \item[(P3)] 
    $\cB\lhd^n\cA_{|\cT}$ and $\cF_\cB$ is a  $\lhd^n$\--system for
    $(\cB,\cA_{|\cT})$.
  \item[(P4)]
    $\cT$ together with the restriction of $\varphi$ to $\biguplus\cT$, which we
    will denote $\varphi_{|\cT}$, is a triangulation of $\cF_\cB\cup\CB(\cB)$
    with parameters $(n,N,e,M)$. Note that, since $\cT$ refines
    $\cS_{|\cT}$ and $(\cS,\varphi)$ is a triangulation of $\cF_0$,
    $(\cT,\varphi_{|\cT})$ is also a triangulation of $\cF_{0|\cX}$ with
    $\cX=\varphi(\cT)$.
\end{description}

\begin{remark}\label{re:pretri-prep}
  Obviously $(\emptyset,\emptyset,\emptyset)$ is preparation for $(\cS,\varphi,\cA,\cF_0)$. Given an
  arbitrary preparation $(\cT,\cB,\cF_\cB)$ for $(\cS,\varphi,\cA,\cF_0)$
  such that $\bigcup\cT\neq\bigcup\cS$, and $S$ a minimal element in
  $\cS\setminus\cS_{|\cT}$, it suffices to build from it a preparation
  $(\cU,\cC,\cF_\cC)$ such that $\biguplus\cU=\biguplus\cT\cup S$. Indeed, $\cS_{|\cU}$
  contains one more element of $\cS$ than $\cS_{|\cT}$ thus, starting
  from $(\emptyset,\emptyset,\emptyset)$ and repeating the process inductively we will finally
  get a preparation $(\cV,\cD,\cF_\cD)$ such that $\biguplus\cV=\biguplus\cS$, hence
  $\cA_{|\cV}=\cA$.  (P4) then implies that $(\cV,\varphi)$ is a
  triangulation of $\cF_0\cup\cF_\cB\cup\CB(\cB)$ with parameters
  $(n,N,e,M)$. So the tuple $(\cV,\varphi,\cD,\cF_\cD)$ satisfies the
  conclusion of Lemma~\ref{le:pretriang}, which finishes the proof. 
\end{remark}

So from now on, let $(\cT,\cB,\cF_\cB)$ be a given preparation for
$(\cS,\varphi,\cA,\cF_0)$ such that $\biguplus\cT\neq\biguplus\cS$. Let $S$ be a minimal element
in $\cS\setminus\cS_{|\cT}$ and $\cA_S=\{A\in\cA\tq\widehat{A}=\varphi(S)\}$. The
minimality of $S$ ensures that every proper face of $S$ belongs to
$\cS_{|\cT}$, hence $\biguplus\cS\cup S$ and $\bigcup(\cA_{|\cT}\cup\cA_S)$ are closed.

\begin{claim}\label{cl:A0-A1}
  Let $A$ be a cell of type $1$ in $\cA_S$, $T$ a simplex in $\cT$
  contained in $\overline{S}$,
  and $Y=\varphi(T)$. If\/ $\bar\nu_A=0$ on $Y$ then $\Gr\bar c_{A|Y}=\partial_Y^0A$
  belongs to $\cB$. If moreover $\bar\mu_A\neq0$ on $Y$ then $\partial_Y^1A$ is
  covered by the cells in $\cB$ that it meets, and among them there is
  a unique cell $B_T^1$ whose closure meets $\partial_Y^0A$. More
  precisely:
  \begin{displaymath}
    B_T^1=\big(\bar c_{A|Y},0,\mu_{B_T^1},G_A\big) 
  \end{displaymath}
  and either $|\mu_{B_T^1}|=|\bar\mu_A|$ on $Y$, or
  $|\mu_{B_T^1}|\leq|\pi^{N_0}\bar\mu_A|$ on $Y$. In particular the closure of
  $B_T^1$ contains $\partial_Y^0A$.
\end{claim}

\begin{proof}
Note first that for every $i\in\{0,1\}$, $\partial_Y^iA$ is contained in
$\overline{A}$ hence in $\bigcup\cA$ since it is closed by assumption.
Every cell in $\cA$ which meets $\partial_Y^iA$ is contained in it since
$\cA$ is a $\lhd^n$\--complex, and belongs to $\cA_{|\cT}$ (otherwise its socle
would not meet $Y$ since $Y\subseteq\bigcup\varphi(\cT)$). Since $\cB$ refines
$\cA_{|\cT}$ it follows that $\partial_Y^iA$ is the union of the cells $B$ in
$\cB$ which it contains. 

In particular, if $\bar\nu_A=0$ on $A$ then $\partial_Y^0A\neq\emptyset$ hence it contains
a cell $B\in\cB$. Necessarily $B$ is of type $0$ since so is $\partial_Y^0A$,
and thus $B=\partial_Y^0A$ since they have the same socle $Y$. This proves
the first point. 

For the second point, since $\bar\nu_A=0\neq\bar\mu_A$ on $Y$ both $\partial_Y^0A$
and $\partial_Y^1A$ are non-empty. Now $\partial_Y^0A$ is contained in the closure
of $\partial_Y^1A$, which is the union of the closure of the cells in $\cB$
contained in $\partial_Y^1A$. Hence necessarily the closure of at least one
of them, say $B$, meets $\partial_Y^0A$. 

$\widehat{B}$ meets $Y$ and both of them belong to $\widehat{\cB}$ so
$\widehat{B}=Y$. Since $\overline{B}\cap(Y\times K)$ meets $\partial_Y^0A$ and
$B\subseteq\partial_Y^1A$ is disjoint from $\partial_Y^0A$, $B$ must be of type $1$ with
$\nu_B=0$ because otherwise $B$ would be closed in $Y\times K$. It follows
that $\overline{B}\cap(Y\times B)$ is the union of $B$ and $\partial_Y^0B$, and the
latter meets $\partial_Y^0A$. By the first point $\partial_Y^0A\in \cB$. By
Proposition~\ref{pr:cell-comp} applied to $\cB$, $\partial_Y^0B\in\cB$.
Thus $\partial_Y^0B=\partial_Y^0A$, in particular they have the same center so
$c_B=\bar c_{A|Y}$. Pick any $(x,t)\in B$, so that $t-c_B(x)\in G_B$. $B$
is contained in $\partial_Y^1A$ hence $t-\bar c_A(x)\in G_A$. Since
$c_B(x)=\bar c_A(x)$ it follows that $G_B\cap G_A\neq\emptyset$ hence $G_A=G_B$. 

This proves that $B=(\bar c_{A|Y},0,\mu_B,G_A)$. The uniqueness of $B$
follows. Indeed if $B'$ is any cell in $\cB$ contained in $\partial_Y^1A$
whose closure meets $\partial_Y^0A$, the same argument shows that $B'=(\bar
c_{A|Y},0,\mu_{B'},G_A)$. This implies that for any $t'\in K$ such that
$t'-\bar c_A(x)$ is small enough and belongs to $G_A$, the point
$(x,t')$ will belong both to $B$ and $B'$, so $B=B'$. 

If $|\mu_B|=|\bar\mu_{A|Y}|$ we are done, so let us assume the contrary.
Then $|\mu_B(x)|\neq|\bar\mu_A(x)|$ for some $x\in Y$. $B$ is a fitting cell so
let $t\in K$ be such that $(x,t)\in B$ and $|t-c_B(x)| = |\mu_B(x)|$. We
have $|t-\bar c_A(x)| \leq |\bar\mu_A(x)|$ because $(x,t)\in\partial_Y^1A$, so
$|\mu_B(x)|<|\bar\mu_A(x)|$. We are going to show that $|\mu_B|<|\bar\mu_A|$ on
$Y$. Since $\partial_Y^1A$ is a fitting cell it follows that $\partial_Y^1A$ is not
contained in $B$, so there is at least one other cell $C$ in $\cB$
contained in $\partial_Y^1A$. Now $\widehat{C}$ is contained in $Y$ and both
of them belong to $\widehat{\cB}$ so $\widehat{B}=Y$. For each $y$ in
$Y$ fix $t_y$ in $K$ such that $(y,t_y)\in C$. Since $C$ is contained in
$\partial_Y^1A$ we have:
\begin{displaymath}
  0\leq|t_y-\bar c_A(y)|\leq|\bar\mu_A(y)|\mbox{ and }t_y-\bar c_A(y)\in G_A 
\end{displaymath}
Necessarily $|\mu_B(y)|<|t_y-\bar c_A(y)|$ because otherwise $(y,t_y)$
would belong both to $C$ and $B$, a contradiction. Hence {\it a
fortiori} $|\mu_B(y)|<|\bar\mu_A(y)|$. By Proposition~\ref{pr:fit-cell}
this implies that $|\mu_B(y)|\leq|\pi^{N_0}\bar\mu_A(y)|$ because $B$ and $A$
are fitting cells mod $\GG$, $v\GG=N_0\cZ$ and $G_B=G_A$. So
$|\mu_B|\leq|\pi^{N_0}\bar\mu_{A|Y}|$ in that case, which proves our claim. 
\end{proof}

We can now begin our construction of a preparation $(\cU,\cC,\cF_\cC)$
for $(\cS,\varphi,\cA,\cF_0)$ such that $\biguplus\cU=\biguplus\cT\cup S$. We are going to
refine $\cA_S$ twice. First ``vertically'', according to the image by
$\varphi$ of a certain partition of $S$ which, together with $\cT$, forms a
simplicial subcomplex $\cU$ of $\cS$ refining $\cS_{|\cT}\cup\{S\}$
(Claim~\ref{cl:U-comp}). Then ``horizontally'' by enlarging the cells
in $\cB$ contained in the closure of $\bigcup\cA_S$ in such a way that the
family of these new cells, together with $\cB$, forms a cellular
monoplex $\cC$ mod $\GG$ refining $\cA_{|\cT}\cup\cA_S=\cA_{|\cU}$ such
that $\cC\lhd^n\cA_{|\cU}$. The point of the construction is to ensure
that $\cC$ comes with a  $\lhd^n$\--system $\cF_\cC$ for
$(\cC,\cA_{|\cU})$ such that $(\cU,\cC,\cF_\cC)$ is a preparation for
$(\cS,\varphi,\cA,\cF_0)$. 

\subsection{Vertical refinement} 
\label{sub:vert-refin}

Let $\cT_S$ be the list of proper faces of $S$. We first deal with the case
where $S$ is not closed, that is $\cT_S\neq\emptyset$. For every $A$ in $\cA_S$
let:
\begin{displaymath}
  (c_A^\circ,\nu_A^\circ,\mu_A^\circ)=(c_A\circ\varphi,\nu_A\circ\varphi,\mu_A\circ\varphi)_{|S}
\end{displaymath}

For every $T\in\cT_S$ and every $A\in\cA_S$ let $\Phi_{T,A}(t,\varepsilon)$ be the formula
saying that $(t,\varepsilon)\in T\times R^*$ and that one of the following conditions hold,
with $n_1=\max(n,1+2vN)$:
\begin{description}
  \item[(A1)$_{t,\varepsilon}$:]
    $\bar\nu^\circ_A(t)\neq 0$ and for every $s\in S$ such that
    $\|s-t\|\leq|\varepsilon|$:
    \begin{displaymath}
      |c^\circ_A(s)-\bar c^\circ_A(t)|\leq|\pi^{n_1}\bar \nu^\circ_A(t)|
    \end{displaymath}
    and $|\nu^\circ_A(s)|=|\bar \nu^\circ_A(t)|$ and $|\mu^\circ_A(s)|=|\bar \mu^\circ_A(t)|$
  \item[(A2)$_{t,\varepsilon}$:]
    $\bar\nu^\circ_A(t)= 0$, $\bar\mu^\circ_A(t)\neq 0$ and for every $s\in S$ such that
    $\|s-t\|\leq|\varepsilon|$:
    \begin{displaymath}
      |c^\circ_A(s)-\bar c^\circ_A(t)|\leq|\pi^{n_1-N_0}\mu^\circ_B(t)|
    \end{displaymath}
    and $|\nu^\circ_A(s)|\leq|\mu^\circ_B(t)|\leq|\bar \mu^\circ_A(t)|=|\mu^\circ_A(s)|$ where $B$
    is the cell $B_{T}^1$ given by Claim~\ref{cl:A0-A1}. 
  \item[(A3)$_t$:]
    $\bar\nu^\circ_A(t)=\bar\mu^\circ_A(t)=0$.
\end{description}
Let $\Phi(t,\varepsilon)$ be the conjunction of the (finitely many)
$\Phi_{T,A}(t,\varepsilon)$'s as $T$ ranges over $\cT_S$ and $A$ over $\cA_S$.
Finally let $\Psi(t,\varepsilon)$ be the formula saying that that $|\varepsilon|$ is maximal
among the elements $\varepsilon'$ in $R^*$ such that $K\models\Phi(t,\varepsilon')$. Obviously
$\Psi(t,\varepsilon)$ implies $\Phi(t,\varepsilon)$. 

By continuity of the center and bounds of $A$, for every $t\in T$ there
exists $\varepsilon_{t,T,A}\in R^*$ such that $K\models\Phi_{T,A}(t,\varepsilon_{t,T,A})$. Hence for
every $t\in\partial S$ there is $\varepsilon\in R^*$ such that $K\models\Phi(t,\varepsilon)$ (every $\varepsilon\in R^*$
such that $|\varepsilon|\leq|\varepsilon_{t,T,A}|$ for every $(T,A)\in\cT_S\times\cA_S$ is a
solution). For every $t\in\partial S$, the set $E_t$ of elements $\varepsilon$ of $R^*$
such that $K\models \Phi(t,\varepsilon)$ is semi-algebraic, bounded and non-empty. So by
Corollary~\ref{co:max} there is $\varepsilon_t\in E_t$ such that $|\varepsilon_t|$ is
maximal in $|E_t|$, that is $K\models \Psi(t,\varepsilon_t)$. Theorem~\ref{th:skol} then
gives a semi-algebraic function $\varepsilon:\partial S\to R^*$ such that $K\models\Psi(t,\varepsilon(t))$
for every $t\in \partial S$, hence {\it a fortiori}:
\begin{equation}
  \forall t\in\partial S,\ K\models\Phi(t,\varepsilon(t)).
  \label{eq:phi-eps-t}
\end{equation}

\begin{claim}\label{cl:eps-loc-cst}
  Let $\varepsilon:\partial S\to R^*$ be the semi-algebraic function defined above. Then
  the restriction of $|\varepsilon|$ to every proper face $T$ of $S$ is
  continuous.
\end{claim}

\begin{proof} 
Note first that if $K\models\Phi_{T,A}(t,\varepsilon')$ for some $t\in T$ and $\varepsilon'\in R^*$
then $K\models\Phi_{T,A}(t',\varepsilon')$ for every $t'\in T\cap B$ where $B=B(t,\varepsilon')$ is the
ball with center $t$ and radius $\varepsilon'$.

Indeed, assume for example that $\bar\nu^\circ_A(t)\neq0$, hence (A1)$_{t,\varepsilon'}$
holds. It claims that for every $s\in S\cap B$
\begin{equation}
  |c_A^\circ(s)-\bar c_A^\circ(t)|\leq|\pi^{n_1}\bar\nu_A(t)|
  \label{eq:eps-lc-1}
\end{equation}
and $|\nu_A(s)|=|\bar\nu_A^\circ(t)|$ and $|\mu_A(s)|=|\bar\mu_A^\circ(t)|$. Now
$T\cap B\subseteq\overline{S\cap B}$ hence, as $s$ tends in $S\cap B$ to any given
$t'\in T\cap B$ we get
\begin{equation}
  |c_A^\circ(t')-\bar c_A^\circ(t)|\leq|\pi^{n_1}\bar\nu_A(t)|
  \label{eq:eps-lc-2}
\end{equation}
and $|\nu_A(t')|=|\bar\nu_A^\circ(t)|$ and $|\mu_A(t')|=|\bar\mu_A^\circ(t)|$. 
By combining  (\ref{eq:eps-lc-1}) and (\ref{eq:eps-lc-2}) with the
triangle inequality we obtain that for every $s\in S\cap B$
\begin{displaymath}
  |c_A^\circ(s)-\bar c_A^\circ(t')|\leq|\pi^{n_1}\bar\nu_A(t)|=|\pi^{n_1}\bar\nu_A(t')|
\end{displaymath}
and $|\mu_A(s)|=|\bar\mu_A^\circ(t')|$, that is (A1)$_{t',\varepsilon'}$. 

Assume now that $\bar\mu^\circ_A(t)=0$, hence $\bar\nu_A^\circ(t)=0$, that is
(A3)$_t$ holds.  Then $\varphi(T)\in\widehat{A}$ and $\bar\mu_A(\varphi(t))=0$
imply that $\bar\mu_A=0$ on $\varphi(T)$ because $\cA$ is a closed
$\lhd^n$--complex (see footnote~\ref{ft:derAi-cell}). So
$\bar\mu_A^\circ=\bar\nu_A^\circ=0$ on $T$, and (A3)$_{t'}$ follows. 

The intermediate case (A2)$_{t,\varepsilon}$ where $\bar\nu_A^\circ(t)=0$ and
$\bar\mu_A^\circ(t)\neq0$ is similar, and left to the reader. 

Now it follows that if $K\models\Psi(t,\varepsilon')$ and $\|t'-t\|\leq|\varepsilon'|$, then $K\models\Psi(t,\varepsilon'')$ if
and only if $|\varepsilon'|=|\varepsilon''|$. So $|\varepsilon(t)|=|\varepsilon(t')|$ for every $t,t'\in T$ such that
$\|t-t'\|\leq|\varepsilon(t)|$. Thus $|\varepsilon|$ is locally constant, hence continuous
on $T$. 
\end{proof}

Theorem~\ref{th:mono-div} applies to $S$, $\cT_S$ and the function
$\varepsilon$. It gives a partition $\cU_S$ of $S$ such that $\cU_S\cup\cT_S$ is a
simplicial complex, for each $T\in\cT_S$ there is a unique $U\in\cU_S$
with facet $T$, and for every $u\in U$:
\begin{equation}\label{eq:constr-Uk}
  \left\| u - \pi_U(u) \right\| \leq \left|\varepsilon(\pi_U(u))\right| 
\end{equation}
where $\pi_U$ is the coordinate projection of $U$ onto $T$ (see
Remark~\ref{re:face-proj}). On $\varphi(U)$ let $\sigma_U=\varphi\circ\pi_U\circ\varphi^{-1}$. This is
a continuous retraction of $\varphi(U)$ onto $\varphi(T)$. 

For every $U\in\cU_S$ and every $A\in\cA_S$ let
\begin{equation}\label{eq:constr-Ak}
  A_U=A\cap\big(\varphi(U)\times K\big).
\end{equation}
Let $T_U=\emptyset$ if $U$ is closed, and $T_U\in\cT_S$ be the facet of $U$
otherwise. Finally let $\cU=\cU_S\cup\cT_S$.

\begin{claim}\label{cl:U-comp}
  With the notation above, $\cU$ is a simplicial subcomplex of $\cS$
  refining $\cS_{|\cT}\cup\{S\}$ and containing $\cT$. For every $U\subseteq S$ in
  $\cU$ and every $A\in\cA_S$, $A_U$ is a largely continuous fitting
  cell mod $\GG$. Moreover if $U$ is not closed then:
  \begin{enumerate}
    \item
      If $|\bar\nu_A|\neq 0$ on $\varphi(T_U)$, then for every
      $x\in\widehat{A}_U$:
      \begin{equation}\label{eq:c1}
        \big|c_A(x)-\bar c_A(\sigma_U(x))\big| \leq
        \big|\pi^{n_1}\bar\nu_A(\sigma_U(x))\big| 
      \end{equation}
      \begin{equation}\label{eq:munu1}
        |\nu_A(x)| = |\bar\nu_A(\sigma_U(x))|
        \mbox{ \ and \ }
        |\mu_A(x)| = |\bar\mu_A(\sigma_U(x))|
      \end{equation}
    \item
      If $|\bar\nu_A|=0<|\bar\mu_A|$ on $\varphi(T_U)$, then for every
      $x\in\widehat{A}_U$:
      \begin{equation}\label{eq:c2}
        \big|c_A(x)-\bar c_A(\sigma_U(x))\big| \leq
        \big|\pi^{n_1-N_0}\mu_B(\sigma_U(x))\big|
      \end{equation}
      \begin{equation}\label{eq:munu2}
        |\nu_A(x)|\leq|\mu_B(\sigma_U(x))|\leq|\bar \mu_A(\sigma_U(x))|=|\mu_A(x)|
      \end{equation}
      where $B$ is the cell $B_{T}^1$ given by Claim~\ref{cl:A0-A1}. 
  \end{enumerate}
\end{claim}

\begin{proof}
By construction $\cU$ is clearly a simplicial complex refining
$\cT\cup\{S\}$, hence refining $\cS_{|\cT}\cup\{S\}$ since $\cT$ refines
$\cS_{|\cT}$. For every $U\subseteq S$ in $\cU$ and every $A\in\cA_S$,
$A_U$ is a largely continuous fitting cell mod $\GG$ by
(\ref{eq:constr-Ak}), because so is $A$.
If moreover $U$ is not closed let $T=T_U\in\cT_S$
be its facet, let $x$ be any element of $\widehat{A}_{U}=\varphi(U)$,
$s=\varphi^{-1}(x)\in U$ and $t=\pi_U(s)\in T$, where $\pi_U$ is the coordinate
projection of $U$ onto $T$ (see Remark~\ref{re:face-proj}). Note that
$\sigma_U(x)=\varphi\circ\pi_U(s)=\varphi(t)$ hence $\bar c_A(\sigma_U(x))=\bar c_A^\circ(t)$, and
similarly for $\bar \nu_A(\sigma_U(x))$ and $\bar \mu_A(\sigma_U(x))$. By
(\ref{eq:phi-eps-t}) we have $K\models\Phi(t,\varepsilon(t))$. 

If $|\bar\nu_A|\neq 0$ on $\varphi(T)$ then $\bar\nu_A^\circ(t)=\bar\nu_A(\sigma_U(x))\neq0$
hence $\Phi(t,\varepsilon(t))$ says that (A1)$_{t,\varepsilon(t)}$ holds for $t$. By
(\ref{eq:constr-Uk}), $\|s-t\|\leq|\varepsilon(t)|$ so (\ref{eq:c1}) and
(\ref{eq:munu1}) follow from (A1)$_{t,\varepsilon(t)}$. Similarly, if
$|\bar\nu_A|=0<|\bar\mu_A|$ on $\varphi(T)$ then (\ref{eq:c2}) and
(\ref{eq:munu2}) follow from (A2)$_{t,\varepsilon(t)}$. 
\end{proof}

This finishes the construction of the vertical refinement of $\cA_S$
if $S$ is not closed. When $S$ is closed we simply take
$\cU=\cS_{|\cT}\cup\{S\}$. Claim~\ref{cl:U-comp} holds
in this case too, for the trivial reason that there is no non-closed
$U\subseteq S$ in $\cU$. 

\begin{remark}\label{re:cAk}
  For every $U\in\cU_S$, if $\nu_{A_U}=0$ then $\Gr c_{A_U}=B_U$
  for some $B\in\cA_S$. Indeed $\nu_{A|\varphi(U)}=\nu_{A_U}=0$ implies that
  $\nu_A=0$ (thanks to our definition of presented cells) hence $\Gr
  c_A=\partial^0_{\varphi(S)}A$ belongs to $\cA$: it is contained in
  $\overline{A}$, hence in $\bigcup\cA$ since the latter is closed, in
  particular it meets at least one cell $B$ in $\cA$, and the last
  point of Proposition~\ref{pr:cell-comp} then gives that $B=\Gr c_A$.
  Thus $B=\Gr c_A\in\cA_S$, and clearly $\Gr c_{A_U}=B_U$.
\end{remark}

\subsection{Horizontal refinement}
\label{sub:hori-refin}

For every $A\in\cA_S$ we are going to construct for each $U\in\cU_S$  a
partition $\cE_{A,U}$ of $A_U$, and for each $E$ in $\cE_{A,U}$ a
semi-algebraic function $h_{E,A_U}:\varphi(U)\to K$ such that:
\begin{description}
  \item[(Pres)] 
    $\widehat{E}=\varphi(U)=\widehat{A_U}$ and $E$ is a largely continuous
    fitting cell mod $\GG$.
  \item[(Fron)] 
    One of the following holds:
    \begin{description}
      \item[\boldmath($\partial1$)] 
        $\partial E=\emptyset$.
      \item[\boldmath($\partial2$)] 
        $\partial E=\overline{\Gr c_E}$ and $\Gr c_E\in\cE_{C,U}$ for some $C\in\cA_S$.
      \item[\boldmath($\partial3$)] 
        $\partial E=\overline{B}$ for some $B\in\cB$, in which case $U$ is not
        closed,
        $\widehat{B}=\varphi(T_U)$ and:
        \begin{displaymath}
          (c_B,\nu_B,\mu_B)=(\bar c_E,\bar \nu_E,\bar \mu_E)_{|\varphi(T_U)}. 
        \end{displaymath}
    \end{description}
  \item[(Out)] 
    $E\lhd^n A_U$ and $h_{E,A_U}$ is a $\lhd^n$\--transition for $(E,A_U)$. 
  \item[(Mon)] 
    $c_E\circ\varphi_{|U}$, $\mu_E\circ\varphi_{|U}$, $\nu_E\circ\varphi_{|U}$ and
    $h_{E,A_U}\circ\varphi_{|U}$ are $N$\--monomial mod $U_{e,n}$. 
\end{description}

This last construction will finish the proof of
Lemma~\ref{le:pretriang}. Indeed, assuming that it is done, let
$\cC$ be the union of $\cB$ and all the cells $E$ in $\cE_{A,U}$ for
$A\in\cA_S$ and $U\in\cU_S$. Let $\cF_\cC$ be the union of the family of
the corresponding functions $h_{E,A_U}$ and of $\cF_\cB$. By
Claim~\ref{cl:U-comp}, $\cU$ is a simplicial subcomplex of $\cS$ such
that $\biguplus\cU=\biguplus\cT\cup S$. The assumption (P2) for $\cB$, together with
(Pres) and (Fron), give that $\cC$ is a cellular monoplex mod $\GG$
refining $\cA_{|\cT}\cup\cA_S$ and that $\varphi(\cU)=\widehat{\cC}$. The
assumption (P3) for $\cB$ and $\cF_B$, together with (Out) above, give
that $\cC\lhd^n\cA_{|\cU}$ and $\cF_\cC$ is a  $\lhd^n$\--system for
$(\cC,\cA_{|\cU})$. Finally the assumption (P4) for $(\cT,\varphi_{|\cT})$
together with (Mon) ensure that $(\cU,\varphi_{|\cU})$ is a triangulation of
$\cF_\cC\cup\CB(\cC)$ with parameters $(n,N,e,M)$. So $(\cU,\cC,\cF_\cC)$
is a preparation of $(\cS,\varphi,\cA,\cF_0)$, and since $\biguplus\cU=\biguplus\cT\cup S$ we
conclude by Remark~\ref{re:pretri-prep}. 
\\

So let $A\in\cA_S$ and $U\in\cU_S$ be fixed once and for all in the
remainder. 

\begin{remark}\label{re:cAk-mon}
  Recall that $(\cS,\varphi)$ is a triangulation of $\cF_0$, and $\cF_0$
  contains $\CB(\cA)$. In particular $c_A\circ\varphi_{|S}$ is $N$\--monomial
  mod $U_{e,n}$ hence {\it a fortiori} so is $c_A\circ\varphi_{|U}$. By
  (\ref{eq:constr-Ak}) $c_{A_U}=c_{A|\varphi(U)}$ hence
  $c_{A_U}\circ\varphi_{|U}=c_A\circ\varphi_{|U}$. Thus $c_{A_U}\circ\varphi_{|U}$ is
  $N$\--monomial mod $U_{e,n}$, and so are $\mu_{A_U}\circ\varphi_{|U}$ and
  $\nu_{A_U}\circ\varphi_{|U}$ by the same argument. 
\end{remark}

Let us first assume that $U$ is closed. We distinguish two
elementary cases.

\paragraph{Case 1.1:} $\mu_{A_U}=0$ or $\nu_{A_U}\neq0$. 
\\
Then $A_U$ is closed. We let $\cE_{A,U}=\{A_U\}$ and
$h_{A_U,A_U}=1$. (Pres), ($\partial1$), (Out) and (Mon) are obvious (using
Remark~\ref{re:cAk-mon} for the latter). 

\paragraph{Case 1.2:} $0=|\nu_{A_U}|<|\mu_{A_U}|$. 
\\
We let $\cE_{A,U}=\{A_U\}$ and $h_{A_U,A_U}=1$. Again (Pres), (Out) and
(Mon) are obvious (same as Case~1.1). Moreover $\partial A_U=\Gr c_{A_U}=\Gr
c_A$, which belongs to $\cA_S$ by Remark~\ref{re:cAk}. If we let
$B=\Gr c_A$, we have $B_U=\Gr c_{A_U}$ and $\mu_{B_U}=0$ hence $\cE_{B,U}=\{B_U\}$
by the previous case. So $\partial A_U=\Gr c_A$ and $\Gr c_A\in\cE_{B,U}$,
which finishes the proof of the statement that ($\partial2$) holds.
\\

These cases being solved, we assume in the remainder that $U$ is not
closed. Recall that $T_U$ is then the facet of $U$ and belongs to
$\cT$. By construction $\partial \varphi(U)=\varphi(\partial U)= \varphi(\overline{T_U})=
\overline{\varphi(T_U)}$. For the convenience of the reader, each of the
following cases is illustrated by a geometric representation of its
conditions (almost like if we were dealing with a cell $A$ over a real
closed field, except that the vertical intervals representing the
fibres of $A$ over $\widehat{A}$ can be clopen). In these figures $A_U$ is
represented by a gray area in $K^2$, its bounds by dotted lines, its socle
$\varphi(U)$ by the horizontal axe, $\partial\varphi(U)=\varphi(T_U)$ by a dot on the left
bound of $\varphi(U)$, and $\overline{A}\cap(\varphi(T_U)\times K)$ by a thick line or
dot on the vertical axe above $\varphi(T_U)$.

\paragraph{Case~2.1:} $|\bar \mu_{A_U}|=0$ on $\varphi(T_U)$.

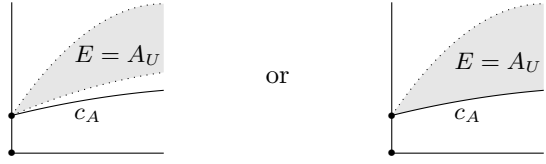
\begin{center}
  \begin{tikzpicture}
    \small
    \def\fonctionNUa{plot[domain=0:2] (\x,{.5+\x*(6-\x)/24})}
    \def\fonctionNUb{plot[domain=0:2] (\x,{.5+\x*(6-\x)/14})}
    \def\fonctionMU{plot[domain=0:2] (\x,{.5+\x*(4-\x)/2.7})}
    \newcommand{\cellulea}[3]{
        \fill[color=gray!#1] 
          plot[domain=#2:#3] (\x,{.5+\x*(6-\x)/24}) --
          plot[domain=#3:#2] (\x,{.5+\x*(4-\x)/2.7}) -- cycle;
      }
    \newcommand{\celluleb}[3]{
        \fill[color=gray!#1] 
          plot[domain=#2:#3] (\x,{.5+\x*(6-\x)/14}) --
          plot[domain=#3:#2] (\x,{.5+\x*(4-\x)/2.7}) -- cycle;
      }

    \draw[thin] (0,0) -- (2,0);
    \draw[thin] (0,0) -- (0,2);
    \draw (0,0) node{\tiny$\bullet$};

    \celluleb{20}{0}{2};
    \draw \fonctionNUa;
    \draw[dotted] \fonctionNUb;
    \draw[dotted] \fonctionMU;
    \draw (0,.5) node{\tiny$\bullet$};

    \draw (1,.5) node{$c_A$};
    \draw (1.4,1.3) node{$E=A_U$};

    \draw (3.5,1) node{{\normalsize or}};

    \begin{scope}[xshift=5cm]

      \draw[thin] (0,0) -- (2,0);
      \draw[thin] (0,0) -- (0,2);
      \draw (0,0) node{\tiny$\bullet$};

      \cellulea{20}{0}{2};
      \draw \fonctionNUa;
      \draw[dotted] \fonctionMU;
      \draw (0,.5) node{\tiny$\bullet$};

      \draw (1,.5) node{$c_A$};
      \draw (1.4,1.2) node{$E=A_U$};

    \end{scope}

  \end{tikzpicture}
\end{center}

We let $\cE_{A,U}=\{A_U\}$ and $h_{A_U,A_U}=1$. (Pres), (Out) and
(Mon) are obvious as in the previous cases.

\begin{itemize}
  \item {\it Sub-case 2.1.a:}
    $\nu_{A_U}\neq0$ or $\mu_{A_U}=\nu_{A_U}=0$. Then $\partial A_U$ is the closure of
    $\Gr\bar c_{A_U|\varphi(T_U)}=\Gr\bar c_{A|\varphi(T_U)}$. The latter belongs
    to $\cB$ by Claim~\ref{cl:A0-A1}, which proves ($\partial3$). 
  \item {\it Sub-case 2.1.b:}
    $\nu_{A_U}=0\neq\mu_{A_U}$. Then $\partial A_U$ is the closure of $\Gr c_{A_U}$.
    By Remark~\ref{re:cAk}, there is a cell $C\in\cA_S$ such that
    $C_U=\Gr c_{A_U}$. Then $\mu_{C_U}=0$ (because $A_U$ is a fitting
    cell) hence $\cE_{C,U}=\{C_U\}$ by the previous sub-case. So $\partial
    A_U=\overline{\Gr c_{A_U}}$ and $\Gr c_{A_U}\in\cE_{C,U}$, which
    proves that ($\partial2$) holds.
\end{itemize}

\paragraph{Case~2.2:} $0<|\bar \nu_A|$ on $\varphi(T_U)$.

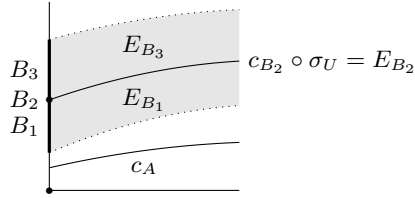
\begin{center}
  \begin{tikzpicture}
    \small
    \def\fonctioncA{plot[domain=0:2.5] (\x,{.3+\x*(6-\x)/26})}
    \newcommand{\fonctioncB}[2]{plot[domain=0:2.5] (\x,{#1+\x*(6-\x)/#2})}
    \newcommand{\cellule}[3]{
        \fill[color=gray!#1] 
          plot[domain=#2:#3] (\x,{.5+\x*(6-\x)/14}) --
          plot[domain=#3:#2] (\x,{2+\x*(6-\x)/22}) -- cycle;
      }

    \cellule{20}{0}{2.5};

    \draw[thin] (0,0) -- (2.5,0);
    \draw[thin] (0,0) -- (0,2.5);
    \draw (0,0) node{\tiny$\bullet$};

    \draw \fonctioncA;
    \draw (1.25,.3) node{$c_A$};

    \draw[dotted] \fonctioncB{.5}{14};
    \draw \fonctioncB{1.2}{17};
    \draw[dotted] \fonctioncB{2}{22};

    \draw 
      (1.25,1.2) node{$E_{B_1}$}
      (1.25,1.9) node{$E_{B_3}$}
      (2.5,1.7) node[right]{$c_{B_2}\circ\sigma_U=E_{B_2}$};

    \draw[very thick] 
      (0,.5) -- node[left]{$B_1$} 
      (0,1.2) node{\tiny$\bullet$} node[left]{$B_2$}
      -- node[left]{$B_3$} (0,2);

  \end{tikzpicture}
\end{center}

In this case, by Claim~\ref{cl:A0-A1}, $\partial_{\varphi(T_U)}^1A=\bar A\cap(\varphi(T_U)\times K)$
is the union of the cells $B\in\cB$ which it contains. For every such
$B$, $\widehat{B}=\varphi(T_U)$ (because $\widehat{\cB}=\varphi(\cT)$ and
$\widehat{B}\subseteq\varphi(T_U)$) and we let: 
\begin{displaymath}
  E_B=(c_B\circ\sigma_U,\nu_B\circ\sigma_U,\mu_B\circ\sigma_U,G_B) 
\end{displaymath}
These $E_B$'s form a family $\cE_{A,U}$ of two by two disjoint largely
continuous cells because the various cells $B$ involved are so and:
\begin{equation}
  (x,t)\in E_B\iff x\in\widehat{A_U}\mbox{ and }(\sigma_U(x),t)\in B.
  \label{eq:fiber-B}
\end{equation}
Each $E_B$ has socle $\varphi(U)=\widehat{A_U}$ and for every $x\in\varphi(U)$,
$\sigma_U(x)$ belongs to $\varphi(T_U)=\widehat{B}$. If $B$ is of type~$0$, then so is $E_B$
and $\mu_B(\sigma_U(x))=0$ (because $B$ is a fitting cell of type $0$) hence
$E_B$ is a fitting cell. If $B$ is of type~$1$, then $\mu_B(\sigma_U(x))\in vG_B$ by
Proposition~\ref{pr:fit-cell} (because $B$ is a fitting cell of
type~$1$). That is $\mu_{E_B}(x)\in G_{E_B}$ hence $\mu_{E_B}$ is a fitting
bound by Proposition~\ref{pr:fit-cell}. Similarly $\nu_{E_B}$ is a
fitting bound, so $E_B$ is a fitting cell. This proves (Pres), and one
can easily derive from (\ref{eq:fiber-B}) that $\partial E_B=\overline{B}$ so
that ($\partial 3$) holds. Note also that $c_{E_B}\circ\varphi_{|U}$ is
$N$\--monomial mod $U_{e,n}$ because so is $c_B\circ\varphi_{|T_U}$ and
$c_{E_B}\circ\varphi_{|U}= c_B\circ\sigma_U\circ\varphi_{|U}= c_B\circ\varphi\circ\pi_U= c_B\circ\varphi_{|T_U}$. The same
reasoning applies to $\nu_{E_B}$ and $\mu_{E_B}$. So the next claim
finishes to prove that $\cE_{A,U}$ is a partition of $A_U$ and that
(out), (Mon) hold.

\begin{claim}\label{cl:rec3}
  $E_B\lhd^n A_U$ and there is a semi-algebraic $\lhd^n$\--transition
  $h_{E_B,A_U}$ for $(E_B,A_U)$ such that
  $h_{E_B,A_U}\circ\varphi_{|U}$ is $N$\--monomial mod $U_{e,n}$.
\end{claim}

\begin{proof}
For every $(x,t)$ in $E_B$, let us prove that $(x,t)$ belongs to $A_U$.
Since $x\in\widehat{A_U}$ it suffices to prove that $(x,t)\in A$. 
By construction $(\sigma_U(x),t)$ belongs to $B$ hence to $\partial_{\varphi(T_U)}^1 A$ so:
\begin{equation}\label{eq:EBdansA}
  \big|\bar\nu_A(\sigma_U(x))\big|
    \leq \big|t-\bar c_A(\sigma_U(x))\big|
    \leq \big|\bar\mu_A(\sigma_U(x))\big|
  \mbox{ and }
  t-\bar c_A(\sigma_U(x))\in G_A
\end{equation}
By (\ref{eq:munu1}) $|\nu_A(x)| = |\bar\nu_A(\sigma_U(x))|$ and
$|\mu_A(x)| = |\bar\mu_A(\sigma_U(x))|$. Moreover by (\ref{eq:c1}):
\begin{eqnarray}
  \big|\big(t-c_A(x)\big) - \big(t-\bar c_A(\sigma_U(x))\big)\big| 
  &=& \big|c_A(x) - \bar c_A(\sigma_U(x)\big| \nonumber\\
  &\leq& \big|\pi^{n_1}\bar\nu_A(\sigma_U(x))\big| \label{eq:star}\\
  &<& \big|t-\bar c_A(\sigma_U(x))\big| \nonumber
\end{eqnarray}
Thus $|t-c_A(x)|=|t-\bar c_A(\sigma_U(x))|$ and by (\ref{eq:EBdansA}):
\begin{displaymath}
  |\nu_A(x))|\leq|t-c_A(x)|\leq|\mu_A(x)|
\end{displaymath}
Moreover by $(\ref{eq:star})$:
\begin{equation}\label{eq:tcn}
  \left|\frac{t-c_A(x)}{t-\bar c_A(\sigma_U(x))}-1\right|
  \leq \left|\pi^{n_1}\frac{\bar\nu_A(\sigma_U(x))}{t-\bar c_A(\sigma_U(x)}\right|
  \leq \big|\pi^{n_1}\big|
\end{equation}
Recall that $n_1=\max(n,1+2vN)$, in particular $n_1>2vN$ hence
$1+\pi^{n_1}R\subseteq \PN_N$ by Hensel's lemma. Since $t-\bar c_A(\sigma_U(x))\in G_A$
by (\ref{eq:EBdansA}) and $G_A\in K^\times/\PN_{N}^\times$, it follows that
$t-c_A(x)\in G_A$. So $(x,t)\in A$ which proves that $E_B\subseteq A$. 

It remains to check that $E_B\lhd^n A_U$, and to find a
$\lhd^n$\--transition for $(E_B,A_U)$. For every $(x,t)\in E_B$ let:
\begin{displaymath}
  \omega_B(x,t)=\pi^{-n}\bigg(\frac{t-c_A(x)}{t-\bar c_A(\sigma_U(x))}-1\bigg)
\end{displaymath}
By (\ref{eq:tcn}) $\omega_B$ takes values in $\pi^{n_1-n}R$ 
hence in $R$ since $n_1\geq n$, thus for every $(x,t)\in E_B$:
\begin{equation}
    t-c_A(x)=\cU_n(x,t)\big(t-\bar c_A(\sigma_U(x))\big) 
  \label{eq:cA-cbarA}
\end{equation}
with $\cU_n=1+\pi^n\omega_B$ in this case. We have $B\subseteq\partial^1_{\varphi(T_U)}A$ and by
(P3) $\cB\lhd^n\cA_{|\cT}$. Since $\cA$ is a closed $\lhd^n$\--complex this
implies that for some $A'\in\cA$ we have $B\lhd^n A'\lhd^n\partial^1_{\varphi(T_U)}A$. Let
$h_0\in\cF_0$ be a $\lhd^n$\--transition function for $(A',\partial^1_{\varphi(T_U)}A)$, and
$h_1\in\cF_B$ a $\lhd^n$\--transition function for $(B,A')$. Then for some
$\alpha_0,\alpha_1\in\{0,1\}$ and every $(x',t')$ in $B$ we have
\begin{displaymath}
  t'-\bar c_A(x')=\cU_n(x',t')h_0^{\alpha_0}(x')\big(t'-c_{A'}(x')\big)^{1-\alpha_0}
\end{displaymath}
and 
\begin{displaymath}
    t'-c_{A'}(x')=\cU_n(x',t')h_1^{\alpha_1}(x')\big(t'-c_B(x')\big)^{1-\alpha_1} 
\end{displaymath}
hence $t'-\bar c_A(x')=\cU_n(x',t')h(x')^\alpha(t'-c_B(x'))^{1-\alpha}$ with
$h=h_0^{1-\alpha_0}h_1^{(1-\alpha_0)\alpha_1}$ and $\alpha=\alpha_0+\alpha_1-\alpha_0\alpha_1$. So $h$ is a
$\lhd^n$\--transition function for $(B,\partial^1_{\varphi(T_U)}A)$. Moreover
$h_0\circ\varphi_{|T_U}$ and $h_1\circ\varphi_{|T_U}$ are $N$\--monomial mod $U_{e,n}$ by
(P4), hence so is $h\circ\varphi_{|T_U}$. For every $(x,t)$ in $E_B$,
$(\sigma_U(x),t)\in B$ so
\begin{displaymath}
  t-\bar c_A\big(\sigma_U(x)\big)
  = \cU_n(x,t)h\big(\sigma_U(x)\big)^\alpha\big[t-c_B\big(\sigma_U(x)\big)\big]^{1-\alpha}.
\end{displaymath}
Combining this with (\ref{eq:cA-cbarA}) and the definition of
$c_{E_B}=c_B\circ\sigma_U$ we get
\begin{displaymath}
  t-c_A(x)= \cU_n(x,t)h\big(\sigma_U(x)\big)^\alpha\big[t-c_{E_B}(x)\big]^{1-\alpha}.
\end{displaymath}
So $E_B\lhd^n A_U$ and $h\circ\sigma_U$ is a $\lhd^n$\--transition for $(E_B,A_U)$.
Moreover $h\circ\sigma_U\circ\varphi_{|U}=h\circ\varphi\circ\pi_U$ by definition of $\sigma_U$. The coordinate
projection $\pi_U$ of $U$ onto $T_U$ is obviously $1$\--monomial, and
$h\circ\varphi_{|T_U}$ is $N$\--monomial mod $U_{e,n}$ by construction. So
$h\circ\sigma_U\circ\varphi$ is also $N$\--monomial mod $U_{e,n}$ and we can take
$h_{E_B,A_U}=h\circ\sigma_U$. 
\end{proof}

\paragraph{Case~2.3:} $0=|\bar \nu_A|<|\bar\mu_A|$ on $\varphi(T_U)$ and $\nu_A\neq0$.

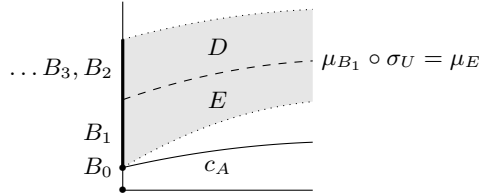
\begin{center}
  \begin{tikzpicture}
    \small
    \def\fonctioncA{plot[domain=0:2.5] (\x,{.3+\x*(6-\x)/26})}
    \newcommand{\fonctioncB}[2]{plot[domain=0:2.5] (\x,{#1+\x*(6-\x)/#2})}
    \newcommand{\cellule}[3]{
        \fill[color=gray!#1] 
          plot[domain=#2:#3] (\x,{.3+\x*(6-\x)/10}) --
          plot[domain=#3:#2] (\x,{2+\x*(6-\x)/22}) -- cycle;
      }

    \cellule{20}{0}{2.5};

    \draw[thin] (0,0) -- (2.5,0);
    \draw[thin] (0,0) -- (0,2.5);
    \draw (0,0) node{\tiny$\bullet$};

    \draw \fonctioncA;
    \draw (1.25,.3) node{$c_A$};

    \draw[dotted] \fonctioncB{.3}{10};
    \draw[dashed] \fonctioncB{1.2}{17};
    \draw[dotted] \fonctioncB{2}{22};

    \draw 
      (1.25,1.2) node{$E$}
      (1.25,1.9) node{$D$}
      (2.5,1.7) node[right]{$\mu_{B_1}\circ\sigma_U=\mu_E$};

    \draw[very thick] 
      (0,.3) node{\tiny$\bullet$} node[left]{$B_0$} -- node[left]{$B_1$} 
      (0,1.2) -- node[left]{$\dots B_3,B_2$} (0,2);

  \end{tikzpicture}
\end{center}

Let $B_0=B^0_{T_U}$ and $B_1=B^1_{T_U}$ the two cells in $\cB$ given
by claim~\ref{cl:A0-A1}. Let:
\begin{displaymath}
  E = (c_A,\nu_A,\mu_{B_1}\circ\sigma_U,G_A)_{|\varphi(U)}
\end{displaymath}
If $|\mu_{B_1}|=|\bar\mu_A|$ on $\varphi(T_U)$ then $|\mu_{B_1}\circ\sigma_U|=|\mu_A|$ on
$\varphi(U)$ by (\ref{eq:munu2}). Thus $E$ and $A_U$ have the same
underlying set. In this case we let $\cE_{A,U}=\{E\}$ and properties
(Pres), (Mon), ($\partial3$) are trivially true. So is (out), using
Remark~\ref{re:cAk-mon} for $c_A\circ\varphi_{|U}$, $\nu_A\circ\varphi_{|U}$, and (P4) for
$\mu_{B_1}\circ\sigma_U\circ\varphi_{|U}=\mu_{B_1}\circ\varphi_{|T_U}$. 

Otherwise $|\mu_{B_1}|<|\bar\mu_A|$ on $\varphi(T_U)$ by Claim~\ref{cl:A0-A1}
and we let:
\begin{displaymath}
  D = (c_A,\pi^{-N_0}\mu_{B_1}\circ\sigma_U,\mu_A,G_A)_{|\varphi(U)} 
\end{displaymath}
$|\mu_{B_1}|\leq|\pi^{N_0}\bar\mu_A|$ on $\varphi(T_U)$ by Claim~\ref{cl:A0-A1},
$|\bar \mu_A\circ\sigma_U|=|\mu_A|$ on $\varphi(U)$ by (\ref{eq:munu2}), so
$|\nu_D|=|\pi^{-N_0}\mu_{B_1}\circ\sigma_U|\leq|\bar\mu_A\circ\sigma_U|\leq|\mu_A|=|\mu_D|$ on $\varphi(U)$.
Moreover $A$ is a fitting cell hence for every $x\in\varphi(U)$ there is $t\in
K$ such that $(x,t)\in A$ and $|t-c_A(x)|=|\mu_A(x)|$, so $(x,t)\in D$. Thus
$D$ is indeed a cell, with socle $\varphi(U)$. It is actually a largely
continuous cell, and $\mu_D=\mu_A$ is a fitting bound. Let us check that
$\nu_D=\pi^{-N_0}\mu_{B_1}\circ\sigma_U$ is a fitting bound too. $B_1$ is a fitting
cell of type~$1$ with socle $\varphi(T_U)$ hence $\mu_{B_1}(\varphi(T_U))\subseteq vG_{B_1}$
by Proposition~\ref{pr:fit-cell}. But $G_{B_1}=G_A$ by
Claim~\ref{cl:A0-A1}, $G_A=G_D$ and $\varphi(T_U)=\sigma_U(\varphi(U))$ by
construction, and $N_0\in v\GG$ so $\nu_D(\varphi(U))\subseteq vG_D$. Thus $\nu_D$ is
indeed a fitting bound by Proposition~\ref{pr:fit-cell}.

Clearly $A_U$ is the disjoint union of $E$ and $D$. Moreover the cells
in $\cB$ contained in $\bar D\cap(\varphi(T_U)\times K)$ are exactly those contained
in $\bar A\cap(\varphi(T_U)\times K)$ except $B_0$ and $B_1$. Thus the construction
that we have done for $A_U$ in case~2.2 applies to $D$ because
$\bar\nu_{D}=\pi^{-N_0}\mu_{B_1}\neq0$ on
$\varphi(T_U)$
and because the analogues of conditions (\ref{eq:c1}) and (\ref{eq:munu1})
that we used for $A_U$ in case~2.2 hold for $D$ in the present case.
Indeed by (\ref{eq:c2}) we have
\begin{displaymath}
  \big|c_{A_U}(x)-\bar c_{A_U}(\sigma_U(x))\big|
  \leq \big|\pi^{n_1-N_0}\bar\mu_{B_1}(\sigma_U(x))\big|.
\end{displaymath}
This is just condition (\ref{eq:c1}) for $D$ since $c_D=c_{A_U}$ and
$\nu_D=\pi^{-N_0}\mu_{B_1}$. Moreover condition (\ref{eq:munu1}) for $D$ is: 
\begin{displaymath}
  |\nu_{D}|=|\bar\nu_{D}\circ\sigma_U|
  \mbox{ \ and \ }
  |\mu_{D}|=|\bar\mu_{D}\circ\sigma_U|
\end{displaymath}
The first equality is true by definition of $\nu_D$ as
$\pi^{-N_0}\mu_{B_1}\circ\sigma_U$. The second one is true because $\mu_D=\mu_A$ and
because of (\ref{eq:munu2}).
  
So the construction of Case~2.2 gives a partition $\cE'$ of $D$ and
for each $E'\in\cE'$ a semi-algebraic function\footnote{Case~2.2 applied
  to $D$ actually gives for each $E'\in\cE'$ a $\lhd^n$\--transition
  $h_{E',D}$ for $(E',D)$. But $D\subseteq A_U$ and $c_D=c_{A_u}$ so
  $h_{E',D}$ is also a $\lhd^n$\--transition for $(E',A_U)$ and we
  can set $h_{E',A_U}=h_{E',D}$.\label{ft:h-D-AU}} 
$h_{E',A_U}:\varphi(U)\to K$ satisfying conditions (Pres), ($\partial3$), (out) and
(Mon). Since $E$ also has these properties (with $h_{E,A_U}=1$ since
$c_E=c_A$ on $\varphi(U)$) we can take $\cE_{A,U}=\{E\}\cup\cE'$.

\paragraph{Case~2.4:} $\bar\mu_A\neq0$ on $\varphi(T_U)$ and $\nu_A=0$.

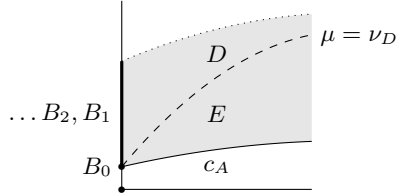
\begin{center}
  \begin{tikzpicture}
    \small
    \def\fonctioncA{plot[domain=0:2.5] (\x,{.3+\x*(6-\x)/26})}
    \newcommand{\fonctioncB}[2]{plot[domain=0:2.5] (\x,{#1+\x*(6-\x)/#2})}
    \newcommand{\cellule}[3]{
        \fill[color=gray!#1] 
          plot[domain=#2:#3] (\x,{.3+\x*(6-\x)/26}) --
          plot[domain=#3:#2] (\x,{1.7+\x*(6-\x)/14}) -- cycle;
      }

    \cellule{20}{0}{2.5};

    \draw[thin] (0,0) -- (2.5,0);
    \draw[thin] (0,0) -- (0,2.5);
    \draw (0,0) node{\tiny$\bullet$};

    \draw \fonctioncA;
    \draw (1.25,.3) node{$c_A$};

    \draw[dashed] \fonctioncB{.3}{5};
    \draw[dotted] \fonctioncB{1.7}{14};

    \draw 
      (1.25,1) node{$E$}
      (1.25,1.8) node{$D$}
      (2.5,2) node[right]{$\mu=\nu_D$};

    \draw[very thick] 
      (0,.3) node{\tiny$\bullet$} node[left]{$B_0$} 
      -- node[left]{$\dots B_2,B_1$} (0,1.7);

  \end{tikzpicture}
\end{center}

Let again $B_1=B_{T_U}^1$ be the cell given by claim~\ref{cl:A0-A1}.
We are going to split $A_U$ in two cells $E$ and $D$ to which previous
cases apply. In order to do so, choose any $i\in\Supp U \setminus \Supp
T_U$. For every $u\in U$ let $\xi_i(u)=u_i$, the $i$\--th coordinate of
$u$. Clearly $\xi_i$ is largely continuous and $\bar\xi_i=0$ on $\partial
U=\bar T_U$. So the function:
\begin{displaymath}
  \mu=(\xi_i\circ\varphi^{-1})^N.(\mu_{B_1}\circ\sigma_U)
\end{displaymath}
is largely continuous on $\widehat{A_U}=\varphi(U)$ and $\bar\mu=0$ on
$\varphi(T_U)$, hence also on $\overline{\varphi(T_U)}=\partial\varphi(U)$. Note
that $\mu\circ\varphi_{|U}$ is $N$\--monomial mod $U_{e,n}$. Let:
\begin{displaymath}
  E = (c_A,0,\pi^{N_0}\mu,G_A)_{|\varphi(U)} 
\end{displaymath}
\begin{displaymath}
  D = (c_A,\mu,\mu_A,G_A)_{|\varphi(U)}
\end{displaymath}
$E$ and $D$ are largely continuous fitting cells mod $\GG$ which define
a partition of $A_U$. (Here we use that $A$ is a fitting cell: for
every $x\in\varphi(U)$ there is $t\in K$ such that $(x,t)\in A$ and
$|t-c_A(x)|=|\mu_A(x)|$ so $(x,t)\in D$, which proves that $D$ is really a
cell. That $D$, $E$ are fitting cells and $A_U=E\cup D$ then follows from
Proposition~\ref{pr:fit-cell}.) In particular $E$ satisfies condition
(Pres). Since $\nu_E=0$ and $\bar\mu_E=\pi^{N_0}\bar\mu=0$ on $\partial\varphi(U)$, we have
$\partial E=\overline{\Gr c_E}$. By Remark~\ref{re:cAk}, $\Gr c_{A_U}=C_U$
for some $C\in\cA_S$, and by Sub-case~2.1.1 applied to $C_U$, $\Gr
c_{A_U}\in\cE_{C,U}$. This proves ($\partial2$) for $E$ since $c_E=c_{A_U}$. 
Let $h_{E,A_U}=1$, this is a $\lhd^n$\--transition for $(E,A_U)$
since they have the same center, so $E$ satisfies (out). It also
satisfies (Mon), thanks to Remark~\ref{re:cAk-mon} for
$c_E=c_{A_U}$ and because $\mu_E\circ\varphi_{|U}=\pi^{N_0}\mu\circ\varphi_{|U}$ is
$N$\--monomial mod $U_{e,n}$.

Case~2.3 applies to $D$ because $\nu_D=\mu\neq 0$, $|\bar\nu_D|=|\bar\mu|=0$ on $\varphi(T_U)$ and
$|\bar\mu_D|=|\bar\mu_A|\neq0$ on $\varphi(T_U)$, and because the analogues of 
conditions (\ref{eq:c2}) and (\ref{eq:munu2}) that we used for $A_U$
in case~2.3 hold for $D$ in the present case. Indeed (\ref{eq:c2}) holds
for $D$ because it holds for $A_U$, and because $D$ and $A_U$
have the same center. Condition (\ref{eq:munu2}) for $D$ is:
\begin{displaymath}
  |\nu_D|\leq|\mu_{B_1}\circ\sigma_U|\leq|\bar\mu_D\circ\sigma_U|=|\mu_D|
\end{displaymath}
The first inequality is true because $|\nu_D|=|\mu|\leq|\mu_{B_1}\circ\sigma_U|$ by
construction, the second one is true by claim~\ref{cl:A0-A1} and
because $\mu_D=\mu_A$, and the last equality is true because it is true
for $A_U$ by (\ref{eq:munu2}) and because $\mu_D=\mu_{A_U}=\mu_{A|\varphi(U)}$. 

So the construction of case~2.3 gives a partition $\cE'$ of $D$ and
for each $E'\in\cE'$ a semi-algebraic function\footnote{Same remark as
  in footnote~\ref{ft:h-D-AU}.} $h_{E',A_U}:\varphi(U)\to K$
satisfying conditions (Pres), (Fron), (out) and (Mon). Since $E$ also
has these properties we can take $\cE_{A,U}=\{E\}\cup\cE'$.

\section{Cartesian morphisms}
\label{se:cart-map}

Let $\cA$ be a cellular monoplex mod $\GG$ such that $\bigcup\cA$ is a
closed subset of $R^{m+1}$. Let $(\cU,\psi)$ be a triangulation of
$\CB(\cA)$ with parameters $(n,N,e,M)$ such that for every $A\in\cA$,
$\psi^{-1}(A)\in\cU$ (we will denote it $U_A$). Note that this is
essentially the data given by the conclusion of
Lemma~\ref{le:pretriang}. The aim of this section is to build a
triangulation $(\cS,\varphi)$ of $\cA$ with the same parameters $(n,N,e,M)$,
together with a continuous projection $\Phi:\biguplus\cS\to\biguplus\cU$ such that the following diagram is
commutative.
\begin{displaymath}
  \xymatrix{
    \bigcup\cA \ar@{->>}[d]    & \biguplus\cS \ar@{-->}[l]_\varphi \ar@{-->>}[d]^\Phi \\
    \bigcup\widehat{\cA} & \biguplus\cU \ar[l]_\psi
  }
\end{displaymath}

We will make the assumption that $\GG=Q_{N,M'}$ with $M'=M+v(N)$ and
$M>v(N)$. In addition we temporarily assume that $\cA$ is a rooted
tree, and $\cU$ a simplicial complex in $D^MR^{q_1}$ for some $q_1$.
We keep these data and assumptions until the end of this section,
where we finally state our result in a more precise and slightly more
general form. 

The construction is done below through a series of claims, which are
connected in the following way. The idea is to prepare the
construction of $\cS$, $\varphi$, $\Phi$ by building first the tree $\cH$ of
supports\footnote{See Remark~\ref{re:well-dispatched}.} of $S$ for
$S\in\cS$, together with an epimorphism of trees from $\cH$ to $\cU$. In
order to do so, we construct a pair of trees of finite subsets of
$\NN^*$ ordered by inclusion, $\cH=(H(A))_{A\in\cA}$ and
$\cP=(P(A))_{A\in\cA}$, which come naturally with increasing
maps\footnote{$\cA$, $\widehat{\cA}$ and
  $\cU$ are ordered by specialisation, while $\cH$ and $\cP$ are
ordered by inclusion.} making the following diagram commutative (see
Claim~\ref{cl:dispatching} and the comments after).
\begin{equation}
  \xymatrixrowsep{.5em}
  \xymatrix{
    \cA \ar@{->>}[dd] & \cH \ar@{->}[l]_\sim 
                            \ar@{-->>}[dd] 
                            \ar@{->>}[rd] &       \\
                      &                   & \cP \ar@{->>}[ld] \\                  
    \widehat{\cA}     & \cU \ar[l]_\sim      &
  }
  \label{eq:diag-A-H-P-U}
\end{equation}
For each $A\in\cA$, a simplex $S_A$ will then be constructed inside
$F_{H(A)}(D^MR^{q_2})$ (for some $q_2\in\NN^*$ large enough), together
with a semi-algebraic isomorphism $\varphi_A$ and a semi-algebraic
projection $\Phi_A$ defined by means of these maps from $\cH$
to $\cA$ and from $\cH$ to $\cU$. This will ensure not only that the
following diagram is commutative (Claim~\ref{cl:image-SA})
\begin{displaymath}
  \xymatrix{
    A \ar@{->>}[d]    & S_A \ar@{-->}[l]_{\varphi_A} \ar@{-->>}[d]^{\Phi_A} \\
    \widehat{A} & U_A \ar[l]_{\psi_{|U_A}}
  }
\end{displaymath}
but also that $\cS=(S_A)_{A\in\cA}$ is a simplicial complex
(Claim~\ref{cl:SA-simplex}) and that the resulting maps $\varphi$, $\Phi$
defined by glueing all the local maps $\varphi_A$, $\Phi_A$ are continuous on
$\bigcup\cS$ (Claims~\ref{cl:cart-cont} and \ref{cl:phi-homeo}).

\begin{claim}\label{cl:faces-UA}
  The faces of $U_A$ are exactly the sets $U_B$ with $B\leq A$ in $\cA$.
\end{claim}

\begin{proof}
  Let $B\leq A$ in $\cA$, $Y=\widehat{B}$ and $i=\Tp B$. Then, with the
  notation of Section~\ref{se:cont-comp},  $B=\partial_Y^iA$ because $\cA$
  is a cellular complex. Since $A$ is bounded, the socle of
  $\overline{A}$ is closed hence $Y$ must be contained in it. Since
  $\psi^{-1}(Y)=U_B$, it follows that $U_B$ is a face of $U_A$.
  Conversely for every face $V$ of $U_A$, the set $B=\partial_{\psi(V)}^0A$
  (resp. $B=\partial_{\psi(V)}^1A$) is non-empty if $\bar\nu_{A|Y}=0$ (resp.
  $\bar\mu_{A|Y}\neq0$) hence belongs to $\cA$. One of these two cases
  necessarily happens (because $|\bar\nu_A|\leq|\bar\mu_A|$ on $Y$), which
  gives $B\in\cA$ such that $U_B=V$. 
\end{proof}

\begin{claim}\label{cl:face-stricte}
  Given any two cells $B\leq A$ in $\cA$, $B<A$ if and only if
  either $U_B<U_A$ or $\Tp B<\Tp A$. In particular if $B$ is the
  predecessor of $A$ in $\cA$ then either $U_B$ is the facet of $U_A$,
  or $U_B=U_A$, in which case $\Tp B=0$ and $\Tp A=1$
\end{claim}

\begin{proof}
  Recall that $B=\partial_Y^jA$ with  $Y=\widehat{B}=\psi(U_B)$ and $j=\Tp B$.
  In particular $A=\partial_X^iA$ with $X=\widehat{A}$ and $i=\Tp A$. Thus
  $B\neq A$ if and only if $U_B\neq U_A$ or $\Tp B\neq\Tp A$. Since $U_B\leq U_A$
  by the previous claim, and obviously $\Tp B\leq \Tp A$ (otherwise
  $\partial_Y^jA=\emptyset$) this proves the equivalence. In particular if $U_B=U_A$
  then $\Tp B<\Tp A$ hence $\Tp B=0$ and $\Tp A=1$. 

  If $B$ is the predecessor of $A$ in $\cA$ and $U_B\neq U_A$, then
  $U_B<U_A$ by Claim~\ref{cl:faces-UA}. Let $V$ be the facet of $U_A$.
  Then $U_B\leq V<A$ hence $B\leq\partial^j_{\psi(V)}A<A$. On the other hand $B$ is
  the predecessor of $A$ in $\cA$, hence $B=\partial^j_{\psi(V)}A$. So
  $\widehat{B}=\psi(V)$ and finally $U_B=V$. 
\end{proof}

Given a strictly increasing map $\sigma:I\to J$ with $I\subseteq[\![1,r]\!]$ and
$J\subseteq[\![1,s]\!]$, we let $[\sigma]:K^s\to K^r$ be defined by $[\sigma](y)=u$ where
$u_i=y_{\sigma(i)}$ if $i\in I$, and $u_i=0$ otherwise. We say that a
function $f:S\subseteq K^r\to K^s$ is a {\df Cartesian map} if for every
$I\subseteq[\![1,r]\!]$ the restriction of $f$ to $S\cap F_I(K^r)$ is of that form,
that is if there is $J\subseteq[\![1,s]\!]$ and a strictly increasing map
$\sigma:I\to J$ such that $f(y)=[\sigma](y)$ for every $y\in S$ with support $I$. If
$X$ is the disjoint union of finitely many sets $X_k\subseteq K^{r_k}$ for
various $k$, then a Cartesian map on $X$ is simply the data of a
Cartesian map on each $X_k$. A {\df Cartesian morphism} is a
continuous Cartesian map. 

\begin{claim}\label{cl:dispatching}
  There exists a pair of functions $H$, $P$ from $\cA$ to
  $\cP(\NN^*)$ such that $H$ is strictly increasing and for every $B\leq A$
  and every $C$ in $\cA$:
  \begin{description}
  \item[(C0)] 
    If $\Tp A=0$ then $H(A)=P(A)$.
  \item[(C1)] 
    If $\Tp A=1$ then $H(A)=P(A)\cup\{r_A\}$ for some $r_A>\max P(A)$.
  \item[(C2)]
    $\Card P(A)=\Card(\Supp U_A)$.
  \item[(C3)] 
    $P(B)=H(B)\cap P(A)$ (in particular $P$ is increasing and $P(B)\subseteq
    H(B)$).
  \item[(C4)] 
    If $\sigma_A:\Supp U_A\to P(A)$ denotes the increasing bijection given by
    (C2) then $\sigma_A(\Supp U_B)=P(B)$.
  \item[(C5)] 
    If $P(C)\subseteq P(A)$ then $U_C\leq U_A$.
  \end{description}
\end{claim}

According to this claim, $H:\cA\to\cH$ is an increasing bijection and
$P:\cA\to\cP$ an increasing surjection. Thus $P\circ H^{-1}:\cH\to\cP$ is an
increasing surjection. $H^{-1}$ and $P\circ H^{-1}$ are respectively the
maps $\cH\to\cA$ and $\cH\to\cP$ in the diagram (\ref{eq:diag-A-H-P-U}) at
the beginning of this section. The maps $A\mapsto\widehat{A}$ and $U_A\mapsto
A$ are $\cA\to\widehat{\cA}$ and $\cU\to\widehat{\cA}$ respectively. The
last\footnote{The dashed map from $\cH$ to $\cU$ is just the
compositum of $\cH\to\cP\to\cU$.} map, from $\cP$ to $\cU$, is $P(A)\mapsto
U_A$. This is a well defined increasing map by (C5), and obviously a
surjective one. The commutativity of the diagram follows by
construction.

\begin{remark}\label{re:sigma}
  Since $\sigma_A$ and $\sigma_B$ are strictly increasing, (C4) implies that
  $\sigma_A(i)=\sigma_B(i)$ for every $i\in\Supp U_B$.
\end{remark}

\begin{proof}
The construction goes by induction in $\Card\cA$. For the root $A$ of
$\cA$ we let $P(A)=\Supp U_A$, $H(A)=P(A)$ if $\Tp A=0$, and
$H(A)=P(A)\cup\{q_1+1\}$ if $\Tp A=1$ (recall that $\cU$ is a simplicial
complex in $D^MR^{q_1}$). If $\cA=\{A\}$ we are done. Otherwise let $A$
be a maximal element of $\cA$ and apply the induction hypothesis to
$\cA\setminus\{A\}$. This defines $P(A')$, $H(A')$ for every $A'\in\cA\setminus\{A\}$ so
that $H$ is strictly increasing on $\cA\setminus\{A\}$ and properties (C0) to
(C4) hold for every $B'\leq A'$ in $\cA\setminus\{A\}$. 

Let $B$ be the predecessor of $A$ in $\cA$ and $k=\Card(\Supp
U_A\setminus\Supp U_B)+1$. For every $A'\in\cA\setminus\{A\}$ let $P_k(A')=\{ki\}_{i\in
P(A')}$ and $H_k(A')=\{ki\}_{i\in H(A')}$. Clearly $P_k$ and $H_k$ inherit
all the properties of $P$ and $H$. Thus, replacing if necessary $P$
and $H$ by $P_k$ and $H_k$ we can assume that $H(A')\subseteq k\NN^*$ for every
$A'\in\cA\setminus\{A\}$. 

Let $q'$ be the maximum of the integers in all these sets $H(A')$. We
have to define $P(A)$ and $H(A)$ so that the resulting maps $P$, $H$
satisfy: (C0) to (C5) for every $B'\leq A'$ and every $C'$ in $\cA$; $H(B')\subseteq H(A')$ and
$H(B')\neq H(A')$ if $B'\neq A'$. By the induction hypothesis it suffices to
check these properties when $A'=A$, $B'=B$ and $C'\in\cA'$. 

We are going to build $\sigma_A$ first, and then let $P(A)=\sigma_A(\Supp U_A)$.
Let $j_1<\cdots<j_r$ be an enumeration of $\Supp U_B$. Let $j_0=0$ and
$j_{r+1}=q'+1$. For every $i\in\Supp U_A$ there is a unique
$l\in[\![0,r]\!]$ such that $j_l\leq i< j_{l+1}$. We then let
$\sigma_A(i)=\sigma_B(j_l)+i-j_l$ (if $j_l=j_0=0$ we let $\sigma_B(j_l)=0$ in this
definition). Note that $j_l+k\leq j_{l+1}$ and $\sigma_B(j_{l+1})\in P(B)\subseteq k\NN^*$
hence
\begin{displaymath}
  \sigma_A(j_l)\leq\sigma_A(i)< \sigma_B(j_l)+j_{l+1}-j_l\leq\sigma_B(j_l)+k\leq\sigma_B(j_{l+1}).
\end{displaymath}
It follows immediately that $\sigma_A$ is strictly increasing. Let
$P(A)=\sigma_A(\Supp U_A)$, by construction (C2) and (C4) hold, $P(A)\cap
k\NN^*=P(B)$ and $P(B)$ is strictly contained in $P(A)$ except if $\Supp
U_A=\Supp U_B$. Note also that in any case $q'+k\notin H(B)\cup P(A)$.
Finally (C5) holds because: 
\begin{itemize}
  \item
    If $P(C')\subseteq P(A)$ then $P(C')\subseteq P(B)$ by construction (because
    $C'\in\cA'$ hence $P(C')\subseteq k\NN^*$). So $U_{C'}\leq U_B$ by the induction
    hypothesis, and since $U_B\leq U_A$ we get $U_{C'}\leq U_A$.
  \item 
    If $P(A)\subseteq P(C')$ then in particular $P(A)\subseteq k\NN^*$, hence by
    construction $P(A)=P(B)=P(C')$ and $\Supp U_A=\Supp U_B$. This
    last point implies that $\dim U_A=\dim U_B$, hence $U_A=U_B$ since
    $U_B\leq U_A$. On the other hand $P(B)=P(C')$ implies that
    $U_B=U_{C'}$ by the induction hypothesis. So altogether $U_A= U_{C'}$
    and {\it a fortiori} $U_A\leq U_{C'}$.
\end{itemize}
It remains to define $H(A)$ and to check (C1) and (C3). We distinguish
four cases, according to the types of $A$ and $B$, and apply
Claim~\ref{cl:face-stricte} to each of them.

\paragraph{Case~1:} $\Tp A=0$, hence $\Tp B=0$ and $U_B$ is the facet
of $A$. In particular $\Supp U_B$ is strictly contained in $\Supp
U_A$, hence so is $P(B)$ in $P(A)$. By the induction hypothesis (C0),
$H(B)=P(B)$. Let $H(A)=P(A)$, then $H(B)\subseteq H(A)$, $H(B)\neq H(A)$ and
(C0), (C3) are obvious. 

\paragraph{Case 2:} $\Tp A=1$, $\Tp B=0$ and $U_B=U_A$. Then
$P(B)=P(A)$ by construction, and $P(B)=H(B)$ by the induction hypothesis
(C0). Let $H(A)=H(B)\cup\{q'+k\}$, then $H(B)\subseteq H(A)$, $H(B)\neq H(A)$ and (C1)
are obvious because $q'+k\notin H(B)$, and $H(B)\cap P(A)=P(B)\cap P(A)=P(B)$
which proves (C3). 

\paragraph{Case 3:} $\Tp A=1$, $\Tp B=0$ and $U_B$ is the facet of
$U_A$. We let $H(A)=P(A)\cup\{q'+1\}$. By the induction hypothesis (C0)
$H(B)=P(B)$. By construction $P(B)\subseteq P(A)\subseteq H(A)$. So $H(B)\subseteq H(A)$,
$H(B)\neq H(A)$ and (C1) are obvious because $q'+k\notin H(B)\cup P(A)$. As in
Case~2, $H(B)\cap P(A)=P(B)\cap P(A)=P(B)$ which proves (C3).

\paragraph{Case 4:} $\Tp A=\Tp B=1$  and $U_B$ is the facet of
$U_A$. By the induction hypothesis (C1), $P(B)$ is strictly contained in
$P(A)$. Let $H(A)=P(A)\cup H(B)$. Then  $H(B)\subseteq H(A)$, $H(B)\neq H(A)$
because $H(A)\setminus H(B)= P(A)\setminus k\NN^*=P(A)\setminus P(B)\neq\emptyset$, $H(A)\cap P(B)=(P(A)\cap
P(B))\cup(H(B)\cap P(B))=P(B)\cup P(B)=P(B)$ which proves (C3), and (C1)
follows because then $H(A)\setminus P(A)=H(B)\setminus P(A)=H(B)\setminus P(B)$ is a singleton
by the induction hypothesis (C1).
\end{proof}

With the notation of Claim~\ref{cl:dispatching}, let $q_2$ be
the maximal element of $\bigcup_{A\in\cA}H(A)$ and $\cS^\dag=
\{F_{H(A)}(D^MR^{q_2})\}_{A\in\cA}$. For every $A\in\cA$ let
$\Phi_A=[\sigma_A]:F_{H(A)}(D^MR^{q_2})\to D^MR^{q_1}$. Finally let 
$\Phi:\bigcup\cS^\dag\to D^MR^{q_1}$ be the resulting Cartesian map.

\begin{claim}\label{cl:cart-cont}
  $\Phi$ is continuous, hence a Cartesian morphism. 
\end{claim}

\begin{proof}
  We have to show that for every $T\leq S$ in $\cS^\dag$ and every
  $z\in T$, $\Phi(y)$ tends to $\Phi(z)$ as $y$ tends to
  $z$ in $S$.  By construction there are $A$, $B$ in $\cA$ such that
  $H(A)=\Supp S$, $H(B)=\Supp T$, $\Phi(y)=[\sigma_A](y)$ and $\Phi(z)=[\sigma_B](z)$.
  Since $[\sigma_A]$ is obviously continuous, it tends to $[\sigma_A](z)$ so we
  have to prove that $[\sigma_A](z)=[\sigma_B](z)$. Let $u=[\sigma_A](z)$ and
  $u'=[\sigma_B](z)$. Recall that $u,u'\in D^MR^{q_1}$ and for every
  $i\in[\![1,q_1]\!]$, $u_i=z_{\sigma_A(i)}$ if $i\in\Supp U_A$, $u_i=0$ otherwise,
  $u'_i=z_{\sigma_B(i)}$ if $i\in\Supp U_B$, and $u'_i=0$ otherwise. 
  
  Since $T\leq S$ we have $\Supp T\leq\Supp S$, that is $H(B)\leq H(A)$, hence
  $B\leq A$ since $H$ is strictly increasing. In particular $\Supp
  U_B\subseteq\Supp U_A$ hence for every $i\in[\![1,q_1]\!]$, we have
  $u_i=u'_i=0$ if $i\notin\Supp U_A$, and by Remark~\ref{re:sigma}
  $z_{\sigma_A}(i)=z_{\sigma_B}(i)$ if $i\in \Supp U_B$, that is $u_i=u'_i$ in
  this case too. The remaining case occurs when $i\in \Supp U_A\setminus\Supp U_B$, so that
  $u_i=z_{\sigma_A(i)}$ and $u'_i=0$. We have to prove that $z_{\sigma_A(i)}=0$,
  that is $\sigma_A(i)\notin\Supp z$. By (C4) and the assumption on $i$,
  $\sigma_A(i)\in P(A)\setminus P(B)$. By (C3), $P(A)\setminus P(B)=P(A)\setminus H(B)$. So
  $\sigma_A(i)\notin H(B)$, and we are done since $\Supp z=\Supp T=H(B)$.
\end{proof}

For every $A\in\cA$, $\mu_A\circ\psi$ is $N$\--monomial mod $U_{e,n}$ so there
are $\zeta\in K$ and some integers $\beta_{i,A}$ for $i\in\Supp U_A$ such that for every
$u\in U_A$
\begin{displaymath}
  \mu_A\circ\psi(u)=U_{e,n}(u)\cdot \zeta\cdot\!\!\!\!\!\!\!\prod_{i\in\Supp U_A}\!\!\!\!u_i^{N\beta_{i,A}}. 
\end{displaymath}
If $\mu_A\neq0$ then $v\mu_A(\widehat{A})=vG_A=v\lambda_A+N\cZ$ by
Proposition~\ref{pr:fit-cell}, and by the above displayed equation
$v(\zeta)\equiv v(\lambda_A)\;[N]$. So there is $\beta_{0,A}\in\cZ$ such that
$v(\zeta)=v(\lambda_A)+N\beta_{0,A}$. Let $\mu_A^v:vU_A\to\cZ$ be defined by\footnote{We
  remind the reader that $A$ is a cell mod $Q_{N,M'}$ with
$M'=M+v(N)$.}
$\mu_A^v(a)=M'+\beta_{0,A}+\sum_{i\in\Supp U_A}\beta_{i,A}a_i$. If $\mu_A=0$ then we let
$\mu_A(a)=+\infty$ for every $a\in vU_A$. Define $\nu_A^v$ accordingly. By
construction, for every $u\in U_A$ we have
\begin{eqnarray}
  v\mu_A\big(\psi(u)\big) &=&  v\lambda_A + N\mu_A^v(vu)-NM' \label{eq:def-muA-v}
  \\
  v\nu_A\big(\psi(u)\big) &=&  v\lambda_A + N\nu_A^v(vu)-NM' \label{eq:def-nuA-v}
\end{eqnarray}
In particular $\mu_A^v$ (resp. $\nu_A^v$) is uniquely determined by $\mu_A$
(resp. $\nu_A$), even if the coefficients $\beta_{i,A}$ are not.

\begin{remark}\label{re:muAv-pos}
  Since $A$ is a fitting cell mod $Q_{N,M'}$ contained in $R$,
  $v\mu_A+M'\geq0$ by Proposition~\ref{pr:mu-bounded}. On the other hand
  $0\leq v\lambda_A\leq N-1$ (see Section~\ref{se:notation}). So, for every $u\in
  U_A$ we have  by (\ref{eq:def-muA-v}):
  \begin{eqnarray*}
    \mu_A^v(vu) &=& v\mu_A(\psi(u))+NM'-v\lambda_A \\  
    &\geq& -M' +NM' -(N-1) \\
    &=& (N-1)(M'-1) \ \geq \ 0
  \end{eqnarray*}
\end{remark}

\noindent
Let $S_A\subseteq D^MR^{q_2}$ be defined as follows.
\begin{itemize}
  \item 
    If $\Tp A=0$, $S_A$ is the set of $y\in F_{H(A)}(D^MR^{q_2}) =
    F_{P(A)}(D^MR^{q_2})$ such that $\Phi(y)\in U_A$. 
  \item 
    If $\Tp A=1$, $S_A$ is the set of $y\in F_{H(A)}(D^MR^{q_2})$
    such that $\Phi(y)\in U_A$ and $\mu_A^v(v\Phi(y))\leq vy_{r_A}\leq \nu_A^v(v\Phi(y))$.
\end{itemize}
In both cases, for every $y\in S_A$ let
\begin{displaymath}
  \varphi_A(y)=\big(\psi\circ\Phi(y),\ c_A(\psi\circ\Phi(y))+\pi^{-NM'}\lambda_Ay_{r_A}^N\big)
\end{displaymath}
where $r_A=\max H(A)$ (if $H(A)=\emptyset$, which happens when $A$ is a point,
then $r_A$ is not defined but in that case $\Tp A=0$, hence $\lambda_A=0$
and we can let $\lambda_Ay_{r_A}^N=0$ by convention).

\begin{claim}\label{cl:image-SA}
  $\Phi(S_A)=U_A$ and $\varphi_A$ is a bijection from $S_A$ to $A$. 
\end{claim}

\begin{proof}
  If $\Tp A=0$ the result is trivial because in that case $H(A)=P(A)$
  hence the restriction of $\Phi$ to $F_{H(A)}(D^MR^{q_2})$ is a bijection
  onto $F_{\Supp U_A}(D^MR^{q_1})$. So from now onwards we assume that $\Tp
  A=1$, hence $H(A)=P(A)\cup\{r_A\}$ and $r_A\notin P(A)$ by (C1).

  Let $y,y'\in S_A$ be such that $\varphi_A(y)=\varphi_A(y')$. Then $\psi(\Phi(y))=\psi(\Phi(y'))$
  and
  \begin{displaymath}
    c_A(\psi\circ\Phi(y))+\pi^{-NM'}\lambda_Ay_{r_A}^N=c_A(\psi\circ\Phi(y'))+\pi^{-NM'}\lambda_Ay'^N_{r_A}
  \end{displaymath}
  These two equations imply that $y_{r_A}^N=y'^N_{r_A}$. Since
  $y_{r_A},y'_{r_A}\in D^MR=Q_{1,M}\cap R$ and $M> v(N)$ it follows that
  $y_{r_A}=y'_{r_A}$ by Lemma~\ref{le:Hensel-DP}. On the other hand
  $\psi(\Phi(y))=\psi(\Phi(y'))$ implies $\Phi(y)=\Phi(y')$ (because $\psi$ is one-to-one),
  that is $y_i=y'_i$ for every $i\in P(A)$ (because $\Phi(y)=[\sigma_A](y)$ by
  construction). Thus $y_i=y'_i$ for every $i\in P(A)\cup\{r_A\}=H(A)$, that
  is $y=y'$ since $\Supp y=\Supp y'=H(A)$. This proves that $\varphi_A$ is
  one-to-one. 

  Let us check now that $A\subseteq\varphi_A(S_A)$. Pick any $(x,t)\in A$, let
  $u=\psi^{-1}(x)$ and $\delta=t-c_A(x)$. Since $\delta\in \lambda_AQ_{N,M'}$ and $\pi^{NM'}\in
  Q_{N,M'}$ we have $\pi^{NM'}\delta\in \lambda_AQ_{N,M'}$. Recall that $M'=M+v(N)$,
  hence by Lemma~\ref{le:Hensel-DP} there is a unique $z\in Q_{1,M}$ such that
  $\pi^{NM'}\delta=\lambda_Az^N$, hence $t=c_A(x)+\pi^{-NM'}\lambda_Az^N$. On the other
  hand we have $v\mu_A(\psi(u))=v\mu_A(x)\leq v\delta$ so by (\ref{eq:def-muA-v}) 
  \begin{displaymath}
    vz = \frac{v(\pi^{NM'}\delta/\lambda_A)}{N}
    \geq \frac{NM'+ v\mu_A(\psi(u))-v\lambda_A}{N} = \mu_A^v(vu).
  \end{displaymath}
  In particular $vz\geq0$ by Remark~\ref{re:muAv-pos} so $z\in Q_{1,M}\cap
  R=D^MR$. Similarly $vz \leq \nu_A^v(vu)$ by (\ref{eq:def-nuA-v}). 
  Let $y\in D^MR^{q_2}$ be such that $y_i=u_{\sigma_A(i)}$ if $i\in
  P(A)$, $y_i=z$ if $i=r_A$, $y_i=0$ otherwise. Then $y\in
  F_{H(A)}(D^MR^{q_2})$, $\Phi(y)=[\sigma_A](y)=u$ and $\mu_A^v(vu)\leq
  vy_{r_A} \leq \nu_A^v(vu)$ since $y_{r_A}=z$, so $y$ belongs to $S_A$. By
  construction $\varphi_A(y)=(x,t)\in A$, which proves that $A\subseteq\varphi_A(S_A)$.
  
  We turn now to $\Phi(S_A)$. For every $u\in U_A$,
  $\psi(u)\in\widehat{A}$ so there is $(x,t)\in A$ such that $u=\psi^{-1}(x)$.
  The above construction gives $y\in S_A$ such that $\varphi_A(y)=(x,t)$. In
  particular $\psi\circ\Phi(y)=x$, so $\Phi(y)=\psi^{-1}(x)=u$, which proves that
  $\Phi(S_A)\supseteq U_A$. Since $\Phi(S_A)\subseteq U_A$ by definition of $S_A$ we get
  that $\Phi(S_A)=U_A$.

  It only remains to show that $\varphi_A(S_A)\subseteq A$. Pick any $y\in S_A$, let
  $(x,t)=\varphi_A(y)$ and $\delta=t-c_A(x)=\pi^{-NM'}\lambda_Ay_{r_A}^N$. Since
  $\Phi(y)\in\Phi(S_A)=U_A$, we have $x=\psi(\Phi(y))\in\psi(U_A)=\widehat{A}$. Since
  $y_{r_A}\in D^MR=Q_{1,M}\cap R$, by Lemma~\ref{le:Hensel-DP} $y_{r_A}^N\in
  Q_{N,M+v(N)}=Q_{N,M'}$. Hence $\delta=\pi^{-NM'}\lambda_Ay_{r_A}^N$ belongs to
  $\lambda_A Q_{N,M'}$. We have $\mu_A^v(v\Phi(y))\leq v y_{r_A}$ by definition of
  $S_A$, so by (\ref{eq:def-muA-v})
  \begin{displaymath}
    v\mu_A\big(\psi(\Phi(y))\big) = v\lambda_A + N\mu_A^v(v\Phi(y))-NM' \leq v\lambda_A +
    Nvy_{r_A}-NM'.
  \end{displaymath}
  The left hand side is equal to $v\mu_A(x)$. For the right hand side we
  have
  \begin{displaymath}
    v\lambda_A + Nvy_{r_A}-NM' = v\big(\pi^{-NM'}\lambda_Ay_{r_A}^N\big)=v\delta.
  \end{displaymath}
  So $v\mu_A(x)\leq v\delta$, that is $|\delta|\leq|\mu_A(x)|$. Similarly $|\nu_A(x)|\leq|\delta|$
  hence $(x,t)\in A$.
\end{proof}

\begin{claim}\label{cl:SA-simplex}
  $S_A$ is a simplex in $D^MR^{q_2}$, whose faces are exactly the sets
  $S_B$ with $B\leq A$ in $\cA$. 
\end{claim}

\begin{proof}
Let $q=\Card P(A)$ and $q'=\Card H(A)$.
Let $\tau_A$ (resp. $\tau'_A$) be the strictly increasing map from $P(A)$ to
$[\![1,q]\!]$ (resp. from $H(A)$ to $[\![1,q']\!]$). By construction
and by Claim~\ref{cl:cart-cont} the following diagram is commutative
(vertical arrows are the natural coordinate projections). 
\begin{displaymath}
  \xymatrix{
    D^M\hbox{$R^{q'}$} \ar[r]^-{[\tau'_A]} \ar[d]
      & \overline{F_{H(A)}(D^M\hbox{$R^{q_2}$})} \ar[drr]^{[\sigma_A]=\Phi} \ar[d]
      &  &  \\
      D^MR^q \ar[r]^-{[\tau_A]} 
      & \overline{F_{P(A)}(D^M\hbox{$R^{q_2}$})} \ar[rr]^-{[\sigma_A]=\Phi}
      & & \overline{F_{\Supp U_A}(D^M\hbox{$R^{q_1}$})}
    }
\end{displaymath}
The horizontal arrows in this diagram are isomorphisms. All of them
are obtained simply by an order-preserving renumbering of the coordinates, hence they
preserve the faces and the property of being a simplex. It will then be
convenient here to identify isomorphic spaces, hence to consider
$U_A\subseteq D^MR^q$ and $S_A\subseteq D^MR^{q'}$. Since $\Phi(S_A)=U_A$ by
Claim~\ref{cl:image-SA}, after this identification $U_A$ is just the
image of $S_A$ by the coordinate projection of $D^MR^{q'}$ to
$D^MR^q$. Since $H(A)=\Supp S_A$ we identify also $H(A)$
with $[\![1,q']\!]$, and $P(A)$ with $[\![1,q]\!]$.

If $\Tp A=0$ then $H(A)=P(A)$, $q'=q$ and the vertical arrows are
identity maps. Thus $S_A$ identifies with $U_A$. In particular $S_A$
is a simplex. Every $B\leq A$ is also of type $0$ and $S_B$ identifies to
$U_B$. The conclusion follows by Claim~\ref{cl:faces-UA}.

From now on, let us assume that $\Tp A=1$. Then $q'=q+1$ hence $U_A$
is just the socle of $S_A$. Similarly, $U_B$ is the socle of $S_B$ for
every $B\leq A$ (if $\Tp B=0$ we have $S_B=U_B\times\{0\}$). By construction
$S_A$ is the inverse image of $vS_A$ by the valuation (restricted to
$D^MR^{q+1}$) and
\begin{displaymath}
  vS_A=\big\{a\in \cZ^{q+1}\tq \widehat{a}\in vU_A\mbox{ and }
  \mu_A^v(\widehat{a})\leq a_{q+1}\leq \nu_A^v(\widehat{a})\big\}.
\end{displaymath}

Since $\mu_A\circ\psi$ and $\nu_A\circ\psi$ are largely continuous on $U_A$,
(\ref{eq:def-muA-v}) and (\ref{eq:def-nuA-v}) imply that $\mu_A^v$ is
largely continuous on $vU_A$. They are affine maps by definition.
Since $0\leq \mu_A^v$  by Remark~\ref{re:muAv-pos}, and $\mu_A^v\leq \nu_A^v$
because $|\nu_A|\leq|\mu_A|$, it follows that $vS_A$ is a polytope in
$\Gamma^{q+1}$. We are going to check that its faces are exactly the sets
$vS_B$ for $B\leq A$ in $\cA$. This will finish the proof since $S_A$ will
then have the expected faces, which implies that $S_A$ is a simplex
because these faces form a chain by specialisation (because $\cA$ is a
tree). 

\paragraph{Step~1.} 
Let $B\leq A$ in $\cA$, then $B=\partial_Y^iA$ with $Y=\widehat{B}$ and $i=\Tp
B$. Let $J=H(B)\subseteq H(A)=[\![1,q+1]\!]$ and $\widehat{J}=P(A)=J\setminus\{q+1\}$. 
Since $(\mu_B,\nu_B)=(\bar\mu_A,\bar\nu_A)_{|Y}$, if $\Tp B=1$ we have by
construction 
\begin{equation}
  vS_B=\big\{a\in F_J(\Gamma^{q+1})\tq \widehat{a}\in vU_B\mbox{ and }
  \bar\mu_A^v(\widehat{a})\leq a_{q+1}\leq\bar\nu_A^v(\widehat{a})\big\}.
  \label{eq:vSB}
\end{equation}
This remains true also if $\Tp B=0$ because in that case $q+1\notin J$ and
$\bar\nu_A^v=+\infty$ on $vU_B$ (because $\bar\nu_A=\nu_B=0$ on $Y$) so the right
hand side is just $vU_B\times\{+\infty\}$, that is $vS_B$ (because $S_B=U_B\times\{0\}$
when $\Tp B=0$). In both cases we also have $vU_B=F_{\widehat{J}}(U_A)$,
because $vU_B$ is a face of $U_A$ by Claim~\ref{cl:faces-UA} and
$\Supp vU_B=P(B)=\widehat{J}$. So the description of $vS_B$ given by
(\ref{eq:vSB}) coincides with the description of $F_J(vS_A)$ given by
Proposition~\ref{pr:pres-face}, which proves that $vS_B=F_J(vS_A)$. 

\paragraph{Step 2.}
Conversely let $F_J(vS_A)\neq\emptyset$ be a face of $vS_A$, for some
$J\subseteq[\![1,q+1]\!]$, and let $\widehat{J}=J\setminus\{q+1\}$. By
Proposition~\ref{pr:pres-face} the socle of $F_J(vS_A)$ is
$F_{\widehat{J}}(vU_A)$ (because $vU_A$ is the socle of $vS_A$) and
two cases can happen: $q+1\in J$ and $\bar\mu_A^v<+\infty$ on
$F_{\widehat{J}}(vU_A)$, or $q+1\notin J$ and $\bar\nu_A^v=+\infty$ on
$F_{\widehat{J}}(vU_A)$. In both cases
\begin{equation}
  F_J(vS_A)=\big\{a\in F_J(\Gamma^{q+1})\tq \widehat{a}\in F_{\widehat{J}}(vU_A)\mbox{ and }
  \bar\mu_A^v(\widehat{a})\leq a_{q+1}\leq\bar\nu_A^v(\widehat{a})\big\}.  
  \label{eq:FJvSA}
\end{equation}
Since $F_{\widehat{J}}(vU_A)$ is a face of $vU_A$, by
Claim~\ref{cl:faces-UA} there is $C\leq A$ in $\cA$ such that
$F_{\widehat{J}}(vU_A)=vU_C$. Let $Y=\widehat{C}=\psi(U_C)$.

If $q+1\notin J$ then by Proposition~\ref{pr:pres-face}, $\bar\nu_A^v=+\infty$ on
$F_{\widehat{J}}(vU_A)=vU_C$. That is $\bar\nu_A=0$ on $Y=\psi(U_C)$, hence
$\partial_Y^0A\in\cA$. Let $B=\partial_Y^0A$ and apply Step~1 to $B$. Since
$J=\widehat{J}$ is the support of $vU_C=vU_B$ and of $S_B$ (because
$\Tp B=0$), we deduce from (\ref{eq:vSB}) and (\ref{eq:FJvSA}) that
$vS_B=F_J(S_A)$. 

If $q+1\in J$ then by Proposition~\ref{pr:pres-face}, $\bar\mu_A^v\neq+\infty$ on
$F_{\widehat{J}}(vU_A)=vU_C$. That is $\bar\mu_A\neq0$ on $Y=\psi(U_C)$, hence
$\partial_Y^1A\in\cA$. Let $B=\partial_Y^1A$ and apply Step~1 to $B$. Since
$\widehat{J}$ is the support of $vU_C=vU_B$ and
$J=\widehat{J}\cup\{q+1\}$ is the support of $S_B$ (because $\Tp B=1$), we
deduce from (\ref{eq:vSB}) and (\ref{eq:FJvSA}) that $vS_B=F_J(S_A)$.
\end{proof}

Finally let $\cS=\{S_A\tq A\in\cA\}$ and  $\varphi:\bigcup\cS\to \bigcup\cA$ be defined by
$\varphi_{|S_A}=\varphi_A$ for each $A\in\cA$. 

\begin{claim}\label{cl:phi-homeo}
  $\varphi$ is a homeomorphism from $\bigcup\cS$ to $\bigcup\cA$. 
\end{claim}

\begin{proof}
We already know by Claim~\ref{cl:image-SA} that $\varphi$ is a bijection
from $\bigcup\cS$ to $\bigcup\cA$. It follows from Claim~\ref{cl:SA-simplex}
that $\bigcup\cS$ is closed, and it is obviously bounded. Thus by
Theorem~\ref{th:extr-val} it suffices to show that $\varphi$ is continuous. 
Since each $\varphi_A$ is obviously continuous on $S_A$, we only have to
prove that for every $z\in\partial S_A$ and $y\in S_A$, $\varphi_A(y)$ tends to $\varphi(z)$
as $y$ tends to $S_A$. By Claim~\ref{cl:SA-simplex} there is $B\leq A$
in $\cA$ such that $z\in S_B$, hence $\varphi(z)=\varphi_B(z)$. 
By Claim~\ref{cl:cart-cont}, $\psi\circ\Phi(y)$ tends to $\psi\circ\Phi(z)$. By
Claim~\ref{cl:image-SA}, $\psi\circ\Phi(y)\in \widehat{A}$ and
$\psi\circ\Phi(z)\in\widehat{B}$ hence $c_A(\psi\circ\Phi(y))$ tends to $\bar c_A(\psi\circ\Phi(z))$,
which is equal to $c_B(\psi\circ\Phi(z))$ since $\bar c_{A|\widehat{B}}=c_B$.
Thus it only remains to check that $\lambda_Ay_{r_A}^N$ tends to
$\lambda_Bz_{r_B}^N$.

If $\Tp A=0$ then also $\Tp B=0$ hence and we are done, since
$\lambda_Ay_{r_A}^N=0=\lambda_Bz_{r_B}^N$. If $\Tp B=1$ then $\lambda_B=\lambda_A$ (because
$\cA$ is a cellular complex) and $r_A=r_B$ (because by (C1) and (C3),
$H(B)\neq P(B)=H(B)\cap P(A)$ implies that $H(B)$ is not contained in
$P(A)$, hence $r_A\in H(B)$ since $H(B)\subseteq H(A)=P(A)\cup\{r_A\}$). Hence
obviously $\lambda_Ay_{r_A}^N$ tends to $\lambda_Az_{r_A}^N=\lambda_Bz_{r_B}^N$ in that
case. Finally if $\Tp A=1$ and $\Tp B=0$ then $r_A\notin H(B)$ (because by
(C0) and (C3), $H(B)=P(B)\subseteq P(A)$), hence $z_{r_A}=0$ since $\Supp
z=\Supp S_B=H(B)$. Thus $\lambda_Ay_{r_A}^N$ tends to $\lambda_Az_{R_A}^N=0$,
which proves the result because $\lambda_Bz_{r_B}^N=0$ since $\Tp B=0$. 
\end{proof}

\begin{remark}\label{re:well-disp-proof}
  By construction $\Supp S_A=H(A)$ and $\Supp S_{A'}=H(A')$ for every
  $A,A'\in\cA$, hence $\Supp S_{A'}\subseteq\Supp S_A$ if and only if $A'\leq A$
  (because $H$ is strictly increasing), which implies that $S_{A'}\leq
  S_A$ by Claim~\ref{cl:phi-homeo}. So for every $S,S'\in\cS$ we have
  \begin{displaymath}
    S'\leq S\iff \Supp S'\subseteq \Supp S. 
  \end{displaymath}
\end{remark}

We can recap now our construction and state the result which
was the aim of this section.

\begin{lemma}\label{le:lift-triang}
  Let $\cA$ be a cellular monoplex mod $Q_{N,M+v(M)}^\times$ 
  such that $\bigcup\cA$ is a closed subset of $R^{m+1}$. Let $(\cU,\psi)$ be a triangulation of
  $\CB(\cA)$ with parameters $(n,N,e,M)$ such that $M>v(N)$ and for
  every $A\in\cA$, $\psi^{-1}(A)\in\cU$ (let us denote it $U_A$). Then there
  exists a simplicial complex $\cS$ of index $M$, a Cartesian morphism
  $\Phi:\biguplus\cS\to\biguplus\cU$ and a semi-algebraic homeomorphism $\varphi:\biguplus\cS\to\bigcup\cA$
  such that for every $A\in\cA$, $\varphi^{-1}(A)\in\cS$ (let us denote it
  $S_A$) and for every $y\in S_A$
  \begin{displaymath}
    \varphi(y)=\big(\psi\circ\Phi(y),c_A(\psi\circ\Phi(y))+\pi^{-NM'}\lambda_Ay_{r_A}^N\big)
  \end{displaymath}
  where\footnote{If $\Supp S_A=\emptyset$ then $r_A$ is not
  defined but in that case $S_A$ is a point, hence so is $A$ so
  $\lambda_A=0$ and we can let $\lambda_Ay_{r_A}^N=0$ by
  convention.\label{fn:rA}} $r_A=\max(\Supp S_A)$.
\end{lemma}

\begin{proof}
  Let $(A_k)_{1\leq k\leq r}$ be the list of minimal elements in $\cA$, and
  for each $k$ let $\cA_k$ be the family of elements in $\cA$ greater
  than $A_k$. This is a rooted, cellular monoplex mod
  $Q_{N,M+v(N)}^\times$. For every $A\in\cA_k$, $\overline{A}$ is the union
  of the cells $B\leq A$ in $\cA$ since $\cA$ is a cellular complex and
  $\bigcup\cA$ is closed. All these cells $B$ belong to $\cA_k$ hence
  $\bigcup\cA_k$ is closed. Since $\bigcup\cA\setminus\bigcup\cA_k$ is the union of the
  finitely many other $\cA_l$ it is closed, hence $\bigcup\cA_k$ is clopen
  in $\bigcup\cA$. Let $\cU_k=\{\psi^{-1}(\widehat{A})\tq A\in\cA_k\}$, this is a
  lower subset of $\cU$ with smallest element $\psi^{-1}(\widehat{A}_k)$
  hence a rooted simplicial complex in $D^{M_1}R^{q_1,k}$ for some
  $q_{1,k}$. Finally let $\psi_k$ be the restriction of $\psi$ to $\bigcup\cU_k$. 
  
  Claims~\ref{cl:faces-UA} to \ref{cl:phi-homeo} apply to
  $(\cU_k,\psi,\cA_k)$ and give a simplicial complex  $\cS_k$ in
  $D^MR^{q_2,k}$ for some $q_{2,k}$, a Cartesian morphism
  $\Phi_k:\bigcup\cS_k\to\bigcup\cU_k$ and a semi-algebraic homeomorphism
  $\varphi_k:\biguplus\cS_k\to\bigcup\cA_k$ satisfying all the required properties. Since
  each $\bigcup\cA_k$ is clopen in $\bigcup\cA$, and each $\bigcup\cU_k$ is clopen in
  $\biguplus\cU$, the  conclusion follows by taking for $\cU$ the family
  $\{\cU_k\}_{1\leq k\leq r}$ and for $\Phi$ (resp. $\varphi$) the map obtained by
  glueing together the various $\Phi_k$ (resp. $\varphi_k$).
\end{proof}

\section{Triangulation}\label{se:triang}

We have come up to the moment when we can show that $\TR m \Rightarrow
\TR{m+1}$. As $\TR 0$ is rather obvious, this will finish the proof of
$\TR m$ for every $m$.

\begin{theorem}\label{th:tm-plus-un}
  Assume $\TR m$. Let $(\theta_i:A_i\subseteq K^{m+1}\to K)_{i\in I}$ be a finite
  family of semi-algebraic functions, and $n,N\geq1$ be any integers. Then
  for some integers $e,M\geq 1$ which can be chosen arbitrarily large (in
  the sense of footnote~\ref{ft:arbit-large}), there exists a
  simplicial complex $\cT$ of index $M$ and a semi-algebraic
  homeomorphism $\varphi:\biguplus\cT\to\bigcup_{i\in I}A_i$ such that for every $i$ in $I$:
  \begin{enumerate}
    \item
      $\displaystyle \{ \varphi(T)\tq T\in\cT \mbox{ and } \varphi(T)\subseteq A_i\}$ is a
      partition of $A_i$.
    \item 
      $\forall T\in\cT$ such that $\varphi(T)\subseteq A_i$, $\theta_i\circ\varphi_{|T}$ is $N$\--monomial
      mod $U_{e,n}$.
  \end{enumerate}
\end{theorem}

\begin{proof}
By using the same partition of $K^{m+1}$ as in the proof of
Lemma~\ref{le:split-level} we are reduced to the case where each $A_i$
is contained in $R^{m+1}$. We can also extend each $\theta_i$ to $R^{m+1}$
by an arbitrary value, and add to this family the indicator functions
of each $A_i$ inside $R^{m+1}$, hence assume that all these functions
have domain $R^{m+1}$, which is closed and bounded. Let $e_*\geq1$ and
$M_*\geq1$ be any integers.

Theorem~\ref{th:large-cont-prep} applies to $(\theta_i)_{i\in I}$. It gives
an integer $e_0\geq1$, a tuple $\eta\in R^m$, a linear automorphism
$u_\eta(x,t)=(x+t\eta,t)$ of $K^{m+1}$ (note that $u_\eta(R^{m+1})=R^{m+1}$
since $\eta\in R^{m+1}$), a pair of integers $N_0\geq1$ and $M_0>2v(e_0)$ such
that $e_0N$ divides $N_0$, and a finite family $\cA$ of largely
continuous cells mod $Q_{N_0,M_0}^\times$ partitioning
$u_\eta^{-1}(R^{m+1})=R^{m+1}$ such that for every $i\in I$, every $A\in\cA$
and every $(x,t)\in A$
\begin{equation}
  \theta_i\circ u_\eta(x,t)=\cU_{e_0,n}(x,t)h_{i,A}(x)
  \big[\lambda_A^{-1}\big(t-c_A(x)\big)\big]^\frac{\alpha_{i,A}}{e_0}
  \label{eq:prep-theta}
\end{equation}
where $h_{i,A}:\widehat{A}\to K$ is a semi-algebraic function and
$\alpha_{i,A}\in\ZZ$. 

Let $n_1=\max(1+2v(e_0),n+v(e_0))$, Lemma~\ref{le:pretriang} applied
to $\cA$ and the family $\cF_0=\{h_{i,A}\tq i\in I,\ A\in\cA\}$ gives a pair
of integers $e_1\geq1$ and $M_1>2v(e_1)$, a cellular monoplex $\cB$ mod
$Q_{N_0,M_0}$ refining $\cA$ such that $\cB\lhd^{n_1}\cA$, a 
$\lhd^{n_1}$\--system $\cF_1$ for $(\cB,\cA)$, and a triangulation
$(\cU,\psi)$ of $\cF_0\cup\cF_1\cup\CB(\cB)$ with parameters
$(n_1,N_0,e_1,M_1)$. Moreover $e_1,M_1$ can be chosen arbitrarily
large, in the sense of footnote~\ref{ft:arbit-large}, so we can
require that $e_*$ divides $e_1$ and $M_1\geq M_*$, and that $M_1\geq
M_0-v(N_0)$ and $M_1>v(N_0)\geq v(e_0)$. 

$Q_{N_0,M_1+v(N_0)}^\times$ is a subgroup of $Q_{N_0,M_0}^\times$ (because
$M_1+v(N_0)\geq M_0$) with finite index. Hence every cell in $\cB$ is the
disjoint union of finitely many cells $C$ mod $Q_{N_0,M_1+v(N_0)}^\times$
with the same socle and bounds as $B$. Since
$vQ_{N_0,M_1+v(N_0)}^\times=N_0\cZ=vQ_{N_0,M_0}^\times$, these cells $C$ are
still fitting cells by Proposition~\ref{pr:fit-cell}. One easily sees
that they form a cellular monoplex $\cC$ refining $\cA$ such that
$\cC\lhd^{n_1}\cA$ and $\cF_1$ is a  $\lhd^{n_1}$\--system for
$(\cC,\cA)$. Moreover $\CB(\cC)=\CB(\cB)$ and
$\widehat{\cC}=\widehat{\cB}$ so $(\cU,\psi)$ is a triangulation of
$\cB(\cC)$ with parameters $(n_1,N_0,e_1,M_1)$ such that
$\psi^{-1}(\widehat{C})\in\cU$ for every $C\in\cC$.

Since $M_1>v(N_0)$, Lemma~\ref{le:lift-triang} applies to
$\cC$ and $(\cU,\psi)$. It gives a simplicial complex $\cT$ of index $M_1$, a
Cartesian morphism $\Phi:\biguplus\cT\to\biguplus\cU$ and a semi-algebraic
homeomorphism $\varphi:\biguplus\cT\to\bigcup\cC$ such that $\varphi^{-1}$ maps each $C$ in
$\cC$ onto some $T$ in $\cT$, and for every $y$ in $T$ 
\begin{equation}\label{eq:def-phi}
  \varphi(y)=\big(\psi\circ\Phi(y),c_C(\psi\circ\Phi(y))+\pi^{-NM'}\lambda_Cy_{r_C}^{N_0}\big) 
\end{equation}
where\footnote{See footnote~\ref{fn:rA}.} $r_C=\max(\Supp T)$. Let
$\varphi_\eta=u_\eta\circ\varphi$, this is a semi-algebraic homeomorphism from $\biguplus\cT$ to
$R^{m+1}$. We are going to check that $\theta_i\circ \varphi_{\eta|T}$ is $N$\--monomial mod
$U_{e_0e_1,n}$ for every $i\in I$ and every $T\in\cT$. This will prove the
result, with $e=e_0e_1$ and $M=M_1$.

So pick any $T\in\cT$, let $C=\varphi(T)$ and $r_C$ be as above. There is a
unique $B\in\cB$ containing $C$, a unique $A\in\cA$ containing $B$. For
every $(x,t)\in C$ let $\delta_C(x,t)=t-c_C(x)$. Let $\delta_A$ and $\delta_B$ be
defined accordingly. Note that $\delta_C=\delta_B$ on $C$ because $C$ has the
same center as $B$ by construction. For every $y\in T$, by
(\ref{eq:prep-theta}) and (\ref{eq:def-phi}) we have
\begin{displaymath}
  \theta_i\circ\varphi_\eta(y)=\cU_{e_0,n}\big(\varphi(y)\big)h_{i,A}\big(\psi\circ\Phi(y)\big)
  \big[\lambda_A^{-1}\delta_A\big(\varphi(y)\big)\big]^\frac{\alpha_{i,A}}{e_0}.
\end{displaymath}
We have $\cU_{e_0,n}(\varphi(y))\in U_{e_0,n}\subseteq U_{e_0e_1,n}$ so the factor
$\cU_{e_0,n}(\varphi(y))$ can be replaced by $\cU_{e_0e_1,n}(y)$. Recalling
that $(\cV,\psi)$ is a triangulation of $\cF_0\cup\cF_1$ with parameters
$(n_1,N_0,e_1,M_1)$, that $\Phi$ is a Cartesian morphism and
$h_{i,A}\in\cF_0$, we get that the second factor
$h_{i,A}(\psi\circ\Phi(y))=h_{i,A}\circ\psi(\Phi(y))$ is $N_0$\--monomial mod
$U_{e_1,n_1}$ hence {\it a fortiori} $N$\--monomial mod $U_{e_0e_1,n}$
since $N$ divides $N_0$ and $n_1\geq n$. So it only remains to prove that the last
factor $[\lambda_A^{-1}\delta_A\circ\varphi_{|T}]^{\alpha_{i,A}/e_0}$ is $N$\--monomial mod
$U_{e_0e_1,n}$. It suffices to prove it for
$[\lambda_A^{-1}\delta_A\circ\varphi]^{1/e_0}$.

We can assume that $\Tp A=1$ otherwise $\lambda_A^{-1}\delta_A=1$ and the result
is trivial (see Remark~\ref{re:zero-infty}). Recall that
$\cC\lhd^{n_1}\cA$ and $\cF_1$ is a  $\lhd^n$\--system for
$(\cC,\cA)$. For every $(x,t)\in C$ we then have
\begin{displaymath}
  t-c_A(x)=\cU_{n_1}(x,t)h_{C,A}(x)^\beta\big(t-c_C(x)\big)^{1-\beta} 
\end{displaymath}
with $h_{C,A}\in\cF_1$ and $\beta\in\{0,1\}$ (depending on $A$, $C$). So by
(\ref{eq:def-phi}) we have
\begin{equation}
  \delta_A\big(\varphi(y)\big)=\cU_{n_1}\big(\varphi(y)\big)
  h_{C,A}\big(\psi\circ\Phi(y)\big)^\beta\big(\pi^{-NM}\lambda_Cy_{r_C}^{N_0}\big)^{1-\beta}.
  \label{eq:delta-mon}
\end{equation}
$(\cV,\psi)$ is a triangulation of $\cF_1$ with parameters
$(n_1,N_0,e_1,M_1)$ hence $h_{C,A}(\psi\circ\Phi(y))$ is $N_0$\--monomial mod
$U_{e_1,n_1}$. So (\ref{eq:delta-mon}) implies that $\delta_A\circ\varphi_{|T}$ is
$N_0$\--monomial mod $U_{e_1,n_1}$, hence so is $\lambda_A^{-1}\delta_A\circ\varphi_{|T}$.
Let $\chi:T\to\UU_{e_1}$ and $g:T\to K$ be semi-algebraic functions that for
every $y\in T$ 
\begin{displaymath}
  \lambda_A^{-1}\delta_A\circ\varphi(y)=\chi(y)\cU_{n_1}(y)\zeta g(y)
  \mbox{ \ and \ }
  g(y)=\prod_{1\leq i\leq q}y_i^{\alpha_i N_0}
\end{displaymath}
with $\zeta\in K$, $\alpha_1,\dots,\alpha_q\in\ZZ$. Let $k=N_0/(e_0N)$, by construction $e_0N$
divides $N_0$ hence $k\in\NN^*$. 
Since $T\subseteq D^{M_1}R^{q'}$, each $y_i\in D^{M_1}R\subseteq Q_{1,M_1}\subseteq
Q_{1,v(e_0)+1}$ (because $M_1>v(e_0)$) hence $y_i^{e_0}\in
Q_{e_0,2v(e_0)+1}$. {\it A fortiori} $y^{\alpha_iN_0}=y^{e_0Nk\alpha_i}$
belongs to $Q_{e_0,2v(e_0)+1}$ hence $g$ takes values in
$Q_{e_0,2v(e_0)+1}$ and $g^{1/e_0}$ is $N$\--monomial:
\begin{displaymath}
  \big(g(y)\big)^\frac{1}{e_0}
  =\bigg(\prod_{1\leq i\leq q}y_i^{e_0Nk\alpha_i}\bigg)^\frac{1}{e_0}
  =\prod_{1\leq i\leq q}y_i^{Nk\alpha_i}
\end{displaymath}
But $\lambda_A^{-1}\delta_A$ also takes values in $Q_{e_0,2v(e_0)+1}$ 
because $\delta_A(x,t)\in\lambda_A Q_{N_0,M_0}$ for every $(x,t)\in A$, and
$Q_{N_0,M_0}\subseteq Q_{e_0,2v(e_0)+1}$ since $e_0$ divides $N_0$ and $M_0>2v(e_0)$. 
Thus $(\lambda_A^{-1}\delta_A\circ\varphi_{|T})/g=\cU_{n_1}\zeta\chi$ takes values in 
$Q_{e_0,2v(e_0)+1}$ as well. So does the factor $\cU_{n_1}$ since
$n_1>2v(e_0)$. Hence finally $\zeta\xi(y)\in Q_{e_0,2v(e_0)+1}$ for every
$y\in T$, so $(\zeta\chi)^{1/e_0}$ is well defined. Note that
$\zeta^{e_1}=\zeta^{e_1}\chi^{e_1}=[(\zeta\chi)^{1/e_0}]^{e_0e_1}$ hence 
$\zeta^{e_1}\in P_{e_0e_1}$. Pick any $\eta\in K$ such that $\zeta^{e_1}=\eta^{e_0e_1}$,
and for every $y\in T$ let $\chi'(y)=(\zeta\chi(y))^{1/e_0}/\eta$. This is a
semi-algebraic function taking values in $\UU_{e_0e_1}$ because
\begin{displaymath}
  \big[(\zeta\chi)^{1/e_0}\big]^{e_0e_1}
  = \zeta^{e_1}=\eta^{e_0e_1} .
\end{displaymath}
By Remark~\ref{re:racine-de-U}, $\cU_{n_1}^{1/e_0}=\cU_{n_1-v(e_0)}$
because $n_1>2v(e_0)$, and by definition
$\chi'\cU_{n_1-v(e_0)}=\cU_{e_0e_1,n_1-v(e_0)}$. Altogether this gives
that
\begin{eqnarray*}
  \big[\lambda_A^{-1}\delta_A\circ\varphi_{|T}\big]^\frac{1}{e_0}
  &=& \cU_{n_1}^\frac{1}{e_0}(\zeta\chi)^\frac{1}{e_0}g^\frac{1}{e_0} \\
  &=& \chi'\cU_{n_1-v(e_0)}\big((\zeta\chi)^\frac{1}{e_0}/\chi'i\big) g^\frac{1}{e_0} \\
  &=& \cU_{e_0e_1,n_1-v(e_0)}\eta g^\frac{1}{e_0} 
\end{eqnarray*}
Thus $[\lambda_A^{-1}\delta_A\circ\varphi]^{1/e_0}$ is $N$\--monomial mod
$U_{e_0e_1,n_1-v(e_0)}$ (because so is $g^\frac{1}{e_0}$). It is {\it
a fortiori}  $N$\--monomial mod $U_{e_0e_1,n}$ since $n_1-v(e_0)\geq n$
by construction. 
\end{proof}

\paragraph{Acknowledgement.}
The idea of the proof of the Good Direction Lemma~\ref{le:gd-open} was
given to me at the beginning of this research (more than fifteen
years ago) by my colleague Daniel Naie. I would also like to thank
Immanuel Halupczok and Georges Comte for very helpful recent
discussions on stratifications.


\end{document}